%% file: a3dmm.tex
\documentclass[11pt,reqno,twoside]{article}

\usepackage{mysty_arxiv}

\usepackage[latin1]{inputenc}

\usepackage{fullpage}
\usepackage[top=2.35cm, bottom=2.35cm, left=2.35cm, right=2.35cm]{geometry}

\usepackage[noindentafter]{titlesec}
\usepackage{bold-extra}

\titlespacing\section{0pt}{14pt plus 4pt minus 2pt}{6pt plus 2pt minus 2pt}
\titlespacing\subsection{0pt}{12pt plus 4pt minus 2pt}{4pt plus 2pt minus 2pt}
\titlespacing\subsubsection{0pt}{8pt plus 4pt minus 2pt}{4pt plus 2pt minus 2pt}
\titlespacing\paragraph{0pt}{6pt plus 4pt minus 2pt}{6pt plus 2pt minus 2pt}

\title{Trajectory of Alternating Direction Method of Multipliers \\ and Adaptive Acceleration}
\author{Clarice Poon\thanks{Department of Mathematics, University of Bath, Bath, UK. E-mail: \texttt{cmhsp20@bath.ac.uk} }
		\and
		Jingwei Liang\thanks{DAMTP, University of Cambridge, Cambridge, UK. E-mail: \texttt{jl993@cam.ac.uk} }}
\date{}

\begin{document}

	\setlength{\abovedisplayskip}{5pt}
	\setlength{\belowdisplayskip}{3pt}

\maketitle

\begin{abstract}
The alternating direction method of multipliers (ADMM) is one of the most widely used first-order optimisation methods in the literature owing to its simplicity, flexibility and efficiency. Over the years, numerous efforts are made to improve the performance of the method, such as the inertial technique. 
By studying the geometric properties of ADMM, we discuss the limitations of current inertial accelerated ADMM and then present and analyze an adaptive acceleration scheme for the method. Numerical experiments on problems arising from image processing, statistics and machine learning demonstrate the advantages of the proposed acceleration approach. 
\end{abstract}

\section{Introduction}
\label{sec:introduction}


Consider the following constrained and composite optimisation problem
\beq\label{eq:problem-admm}\tag{$\calP_{\mathrm{ADMM}}$}
\min_{x \in \bbR^n , y\in \bbR^m } ~ R(x) + J(y)  \quad
\textrm{such~that} \quad Ax + By = b  ,
\eeq
where the following basic assumptions are imposed
\begin{enumerate}[leftmargin=4em,label= ({${\bcA}$\textbf{.\arabic{*}})},ref= ${\bcA}$\textbf{.\arabic{*}}]
\item \label{ADMM:RJ} $R\in \lsc\pa{\bbR^n}$ and $J \in \lsc\pa{\bbR^m}$ are proper convex and lower semi-continuous functions.
\item \label{ADMM:AB} $A : \bbR^n \to \bbR^p $ and $B : \bbR^m \to \bbR^p$ are {injective} linear operators.
\item \label{ADMM:minimizers} $\ri\pa{\dom(R) \cap \dom(J)} \neq \emptyset$, and the set of minimizers is non-empty.
\end{enumerate}
Over the past years, problem \eqref{eq:problem-admm} has attracted a great deal of interests as it covers many important problems arising from data science, machine learning, statistics, inverse problems and imaging science, etc.; See Section \ref{sec:experiment} for examples. 
In the literature, different numerical schemes are proposed to handle the problem, among them the alternating direction method of multipliers (ADMM) is the most prevailing one. 


Earlier works of ADMM include \cite{glowinski1975approximation,gabay1975dual,gabay1983chapter,eckstein1992douglas}, and recently it has gained increasing popularity, in part due to \cite{boyd2011distributed}. 
To derive ADMM, first consider the augmented Lagrangian associated to \eqref{eq:problem-admm} which reads 
\[ 
\calL(x,y; \psi) \eqdef R(x) + J(y) + \iprod{\psi}{Ax+By-b} + \sfrac{\gamma}{2}\norm{Ax+By-b}^2 , 
\] 
where $\gamma>0$ and $\psi \in \bbR^p$ is the Lagrangian multiplier. 
To find a saddle-point of $\calL(x,y; \psi) $, ADMM applies the following iteration
\beq\label{eq:admm}
\begin{aligned}
\xk &= \argmin_{x\in\bbR^n }~ R(x) + \tfrac{\gamma}{2} \norm{Ax+B\ykm-b + \tfrac{1}{\gamma}\psikm }^2 , \\
\yk &= \argmin_{y\in\bbR^m }~ J(y) + \tfrac{\gamma}{2} \norm{A\xk+By-b + \tfrac{1}{\gamma}\psikm }^2 , \\
\psik &= \psikm + \gamma\pa{A\xk + B\yk - b}  .
\end{aligned}
\eeq
By defining a new point $\zk \eqdef \psikm + \gamma A\xk$, we can rewrite ADMM iteration \eqref{eq:admm} into the following form
\beq\label{eq:admm2}
\begin{aligned}
\xk &= \argmin_{x\in\bbR^n }~ R(x) + \sfrac{\gamma}{2} \norm{Ax - \tfrac{1}{\gamma} \pa{ \zkm - 2\psikm } }^2 , \\
\zk &= \psikm + \gamma A\xk   ,   \\
\yk &= \argmin_{y\in\bbR^m }~ J(y) + \sfrac{\gamma}{2} \norm{By + \tfrac{1}{\gamma} \pa{ \zk - \gamma b } }^2 , \\
\psik &= \zk + \gamma\pa{B\yk - b}  .
\end{aligned}
\eeq
For the rest of the paper, we shall focus on the above four-point formulation of ADMM. 
In the literature, it is well known that ADMM is equivalent to applying Douglas--Rachford splitting \cite{douglas1956numerical} to the dual problem of \eqref{eq:problem-admm} \cite{gabay1983chapter} which reads
\beq\label{eq:problem-admm-dual}\tag{$\calD_{\mathrm{ADMM}}$}
\max_{\psi\in\bbR^p } - \Pa{R^*(-A^T \psi) + J^*(-B^T \psi) + \iprod{\psi}{b} } ,
\eeq
where $R^*(v) \eqdef \sup_{x\in\bbR^n} ~ \pa{ \iprod{x}{v} - R(x) } $ is the Fenchel conjugate, or simply conjugate, of $R$. The corresponding iteration of Douglas--Rachford splitting reads
\[
\begin{aligned}
\uk &= \argmin_{u\in\bbR^p}~ \gamma R^*(-A^T u) + \sfrac{1}{2} \norm{u - (2\psikm - \zkm)}^2  , \\
\zk &= \zkm + \uk - \psikm , \\
\psik &= \argmin_{\psi\in\bbR^p}~ \gamma J^*(-B^T \psi) + \iprod{\psi}{\gamma b} + \sfrac{1}{2} \norm{\psi - \zk}^2 ,
\end{aligned}
\]
where $\zk$ is exactly the same one of \eqref{eq:admm2}. 
The above iteration can be written as the fixed-point iteration of $\zk$, that is
\beq\label{eq:fp-dr}
\zk = \fDR (\zkm) ,
\eeq
with $\fDR$ being the fixed-point operator. We refer to Appendix \ref{P-sec:tra-admm} for the expression of $\fDR$ and more discussions between the equivalence between ADMM and Douglas--Rachford splitting. 
Based on such an equivalence, the convergence property of ADMM has been well studied in the literature, we refer to \cite{he20121,eckstein2012augmented,deng2016global,hong2017linear} and the references therein. 
 In \cite{bauschke2014local}, $\uk, \psik$ are called the ``\emph{shadow sequences}'' of Douglas--Rachford splitting, in this paper, following the terminology, we shall call $\xk, \yk$ of \eqref{eq:admm} the shadow sequences of ADMM. 


\subsection{Contributions}
The contribution of our paper is threefold. First, 
for the sequence $\seq{\zk}$ of \eqref{eq:admm2}, we show that it has two different types of trajectory:
\begin{itemize}[leftmargin=2em]
\item When both $R, J$ are non-smooth functions, under the assumption that they are partly smooth (see Definition \ref{dfn:psf}), we show that the eventual trajectory of $\seq{\zk}$ is approximately a spiral which can be characterized precisely if $R, J$ are moreover locally polyhedral around the solution. 

\item When at least one of $R, J$ is smooth, we show that depends on the choice of $\gamma$, the eventual trajectory of $\seq{\zk}$ can be either straight line or spiral. 
\end{itemize}
%
Second, based on trajectory of $\seq{\zk}$, we discuss the limitations of the current combination of ADMM and inertial acceleration technique. In Section \ref{sec:fail-inertial}, we distinguish the situations where inertial acceleration will work and when it fails. More precisely, we find that inertial technique will work if the trajectory of $\seq{\zk}$ is or close to a straight line, and will fail if the trajectory is a spiral. 
%

Our core contribution is an adaptive acceleration for ADMM, which is inspired by the trajectory of ADMM and dubbed ``A$\!^3$DMM''. The limitation of inertial technique, particularly its failure, implies that the right acceleration scheme should be able to follow the trajectory of the iterates. In Section \ref{sec:ada-admm}, we propose an adaptive extrapolation scheme for accelerating ADMM which is able to following the trajectory of the generated sequence. 
Our proposed A$\!^3$DMM belongs to the realm of extrapolation method, and provides an alternative geometrical interpretation for polynomial extrapolation methods such as Minimal Polynomial Extrapolation (MPE) \cite{cabay1976polynomial} and Reduced Rank Extrapolation (RRE) \cite{eddy1979extrapolating,mevsina1977convergence}.

\subsection{Related works} 
Over the past decades, owing to the tremendous success of inertial acceleration \cite{nesterov83,fista2009}, the inertial technique has been widely adapted to accelerate other first-order algorithms. In terms of ADMM, related work can be found in \cite{pejcic2016accelerated,kadkhodaie2015accelerated,pmlrv80franca18a}, either from proximal point algorithm perspective or continuous dynamical system. However, to ensure that inertial acceleration works, stronger assumptions are imposed on $R, J$ in \eqref{eq:problem-admm}, such as smooth differentiability or strong convexity. When it comes to general non-smooth problems, these works may fail to provide acceleration. {Recently in \cite{franca2018dynamical}, an $O(1/k^{2})$ convergence rate is established for ADMM using Nesterov acceleration, however the result holds only for the continuous dynamical system while the discrete-time optimization scheme remains unavailable.}

For more generic acceleration techniques, there are extensive works in numerical analysis on the topic of convergence acceleration for sequences. The goal of convergence acceleration is, given an arbitrary sequence $\seq{\zk}\subset\RR^n$ with limit $\zsol$, finding a transformation  $\Ee_k : \{z_{k-j}\}_{j=1}^q\to  \zbark \in \RR^n$ such that $\zbark$ converges faster to $\zsol$. In general, the process by which $\{\zk\}$ is generated is unknown, $q$ is chosen to be a small integer, and $\zbark$ is referred to as the extrapolation of $\zk$. Some of the best known examples  include Richardson's extrapolation \cite{richardson1927viii}, the $\Delta^2$-process of Aitken \cite{aitken1927xxv} and Shank's algorithm \cite{shanks1955non}. 
We refer to \cite{brezinski2001convergence,brezinski2013extrapolation,sidi2003practical} and references therein for a detailed historical perspective on the development of these techniques. 
Much of the works on the extrapolation of vector sequences was initiated by Wynn \cite{wynn1962acceleration} who generalized the work of Shank to vector sequences.  
In the appendix, the formulation of some of these methods  are provided. In particular,  minimal polynomial extrapolation (MPE) \cite{cabay1976polynomial} and Reduced Rank Extrapolation (RRE) \cite{eddy1979extrapolating,mevsina1977convergence} (which is also a variant of Anderson acceleration developed independently in \cite{Anderson}), which are particularly relevant to this present work (see Section \ref{sec:acc-gua} for brief discussion). 

More recently, there has been a series of work on a regularized version of RRE stemming from \cite{scieur2016regularized}. We remark however that the regularisation parameter in these works rely on a grid search based on objective function, their applicability to the general ADMM setting is unclear.

\paragraph{Notations}
Denote $\bbR^n$ a $n$-dimensional Euclidean space equipped with scalar inner product $\iprod{\cdot}{\cdot}$ and norm $\norm{\cdot}$. $\Id$ denotes the identity operator on $\bbR^n$. $\lsc(\bbR^n)$ denotes the class of proper convex and lower-semicontinuous functions on $\bbR^n$. 
For a nonempty convex set $S \subset \bbR^n$, denote $\ri(S)$ its relative interior, $\LinHull(S)$ the smallest subspace parallel to $S$ and $\PT{S}$ the projection operator onto $S$. 
The sub-differential of a function $R \in \lsc(\bbR^n)$ is a set-valued mapping defined by $\partial R (x) \eqdef \Ba{ g\in\bbR^n | R(x') \geq R(x) + \iprod{g}{x'-x} , \forall x' \in \bbR^n }$. 
The spectral radius of a matrix $M$ is denoted by $\rho(M)$.

\paragraph{Organization}
The rest of paper is organized as following: in Section \ref{sec:trajectory-admm} we study the trajectory of ADMM for the sequence $\seq{\zk}$. The limitation of the inertial technique is discussed in Section \ref{sec:fail-inertial}. Then in Section \ref{sec:ada-admm}, we propose the adaptive acceleration scheme for ADMM followed by detailed discussions provided in Section \ref{sec:discussion}. Numerical experiments from machine learning and imaging are provided in Section \ref{sec:experiment}. In the appendix, we first provide extra discussions on why inertial fails, and then the proofs of main propositions.

\section{Trajectory of ADMM}
\label{sec:trajectory-admm}

In this section, we discuss the trajectory of the sequence $\seq{\zk}$ generated by ADMM based on the concept ``partial smoothness'' which was first introduced in \cite{LewisPartlySmooth}.

\subsection{Partial smoothness}

The difficulty of analyzing the trajectory, or in general the behaviors (\eg convergence rate), of ADMM is that the iteration of the method is non-linear. We need a proper tool of (locally) linearize the iteration of ADMM, and partial smoothness provides a powerful framework to achieve the goal. 
Let $\calM \subset \bbR^n$ be a $C^2$-smooth Riemannian manifold, denote $\tanSp{\calM}{x}$ the tangent space of $\calM$ at a point $x \in \calM$.

\begin{definition}[{Partly smooth function \cite{LewisPartlySmooth}}]\label{dfn:psf}
A function $R \in \lsc(\bbR^n)$ is partly smooth at $\xbar$ relative to a set $\calM_{\xbar}$ if $\partial R(\xbar) \neq \emptyset$ and $\calM_{\xbar}$ is a $C^2$ manifold around $\xbar$, and moreover
\begin{itemize}[label={}, leftmargin=2.5cm]
\item[\bf Smoothness] \label{PS:C2}
$R$ restricted to $\calM_{\xbar}$ is $C^2$ around $\xbar$.
%
%
\item[\bf Sharpness] \label{PS:Sharp}
The tangent space $\tanSp{\calM_{\xbar}}{\xbar} = \LinHull\pa{\partial R(\xbar)}^\bot$.
\item[\bf Continuity] \label{PS:DiffCont}
The set-valued mapping $\partial R$ is continuous at $x$ relative to $\calM_{\xbar}$.
\end{itemize}
\end{definition}
{\noindent}The class of partly smooth functions at $\xbar$ relative to $\calM_{\xbar}$ is denoted as $\PSF{\xbar}{\calM_{\xbar}}$.
Popular examples of partly smooth functions can be found in \cite[Chapter 5]{liang2016thesis}.
Loosely speaking, a partly smooth function behaves \emph{smoothly} along $\calM_{\xbar}$, and \emph{sharply} normal to it. 
The essence of partial smoothness is that the behaviour of the function and of its minimizers depend essentially on its restriction to this manifold, hence providing the possibilities to study the trajectory of sequences.


\subsection{Trajectory of sequence $\zk$}\label{subsec:tra-admm}

Next we discuss the trajectory of ADMM in terms of $\seq{\zk}$. The iteration of ADMM is non-linear in general owing to the non-linearity of the proximity mappings. However, when $R, J$ are partly smooth, the local $C^2$-smoothness allows us to linearize the proximity mappings, hence the ADMM iteration. In turn, this allows us to study the trajectory of sequence generated by the method. 
We denote $(\xsol, \ysol, \psisol)$ a saddle-point of $\calL(x,y; \psi)$ and let $\zsol = \psisol + \gamma A\xsol$.

Denote $\vk \eqdef  \zk-\zkm$ and let $\theta_k \eqdef \arccos\Pa{ \frac{\iprod{\vk}{\vkm}}{\norm{\vk}\norm{\vkm}} }$ be the angle between $\vk, \vkm$. We shall use $\seq{\theta_{k}}$ to characterize the trajectory of $\seq{\zk}$. 
Given a saddle point $(\xsol, \ysol, \psisol)$, the corresponding KKT condition entails $- A^T \psisol \in {\partial R(\xsol)}  $ and $- B^T \psisol \in {\partial J(\ysol)} $, below we impose that
\beq\label{eq:ndc-admm}\tag{$\textrm{ND}$}
- A^T \psisol \in \ri\Pa{\partial R(\xsol)}  
\qandq
- B^T \psisol \in \ri\Pa{\partial J(\ysol)}   .
\eeq


\subsubsection{Both {$R, J$} are non-smooth}
Let $\MmRx, \MmJy$ be two smooth manifolds around $\xsol, \ysol$ respectively, and suppose $R \in \PSF{\xsol}{\MmRx}, J \in \PSF{\ysol}{\MmJy}$ are partly smooth. Denote $T_{\xsol}^R, T_{\ysol}^J$ the tangent spaces of $\MmRx, \MmJy$ at $\xsol, \ysol$, respectively. 
Let $A_{R} \eqdef A \circ \PTR{\xsol}, B_J \eqdef B \circ \PTJ{\ysol}$ and $T_{A_R}, T_{B_J}$ be the range of $A_R, B_J$ respectively. Denote $(\alpha_j)_{j=1,..., \min\ba{\dim(T_{A_R}), \dim(T_{B_J})}}$ the principal angles (see Definition \ref{def:principal_angle}) between $T_{A_R}, T_{B_J}$, and let $\alpha_F, \alpha'$ be the smallest and 2nd smallest of all non-zero~$\alpha_j$.

\begin{theorem}\label{prop:trajectory-admm}%
For problem \eqref{eq:problem-admm} and ADMM iteration \eqref{eq:admm2}, assume that conditions \iref{ADMM:RJ}-\iref{ADMM:minimizers} are true, then $(\xk,\yk,\psik)$ converges to a saddle point $(\xsol, \ysol, \psisol)$ of $\calL(x,y; \psi)$. 
Suppose that $R \in \PSF{\xsol}{\MmRx}, J \in \PSF{\ysol}{\MmJy}$ and condition \eqref{eq:ndc-admm} holds, then
\begin{enumerate}[label={\rm (\roman{*})}, leftmargin=3em]
\item 
There exists a matrix $\mADMM$ such that $\vk = \mADMM \vkm + o(\norm{\vkm})$ holds for all $k$ large enough. 

\item If moreover, $R, J$ are locally polyhedral around $\xsol, \ysol$, then $\vk = \mADMM \vkm$ with $\mADMM$ being normal and having eigenvalues of the form $\cos(\alpha_j) e^{\pm \mathrm{i} \alpha_j}$, and $\cos(\theta_k) = \cos(\alpha_F) + O(\eta^{2k})$ with $\eta = \cos(\alpha')/\cos(\alpha_F)$. 
\end{enumerate}
\end{theorem} 
%


\begin{remark}
The result indicates that, when both $R, J$ are locally polyhedral, the trajectory of $\seq{\zk}$ is a spiral. 
For the case $R,J$ being general partly smooth functions, though we cannot prove, numerical evidence shows that the trajectory of $\seq{\zk}$ could be either straight line or spiral; See Section \ref{sec:experiment} for evidences. 
\end{remark}

\subsubsection{$R$ or/and $J$ is smooth}
Now consider the case that at least one function out of $R, J$ is smooth. For simplicity, consider that $R$ is smooth and $J$ remains non-smooth. We have the following result. 
%

\begin{proposition}\label{prop:trajectory-admm-real}
For problem \eqref{eq:problem-admm} and ADMM iteration \eqref{eq:admm2}, assume that conditions \iref{ADMM:RJ}-\iref{ADMM:minimizers} are true, then $(\xk,\yk,\psik)$ converges to a saddle point $(\xsol, \ysol, \psisol)$ of $\calL(x,y; \psi)$. 
Suppose $R$ is locally $C^2$ around $\xsol$, $J \in \PSF{\ysol}{\MmJy}$ is partly smooth and condition \eqref{eq:ndc-admm} holds for $J$, then Theorem \ref{prop:trajectory-admm}(i) holds for all $k$ large enough. 
{If moreover, $A$ is full rank square matrix, then all the eigenvalues of $\mADMM$ are real for $\gamma > \norm{(A^T A)^{-\frac{1}{2}} {\nabla^2 R(\xsol)} (A^T A)^{-\frac{1}{2}}}$.} 
\end{proposition}


\begin{remark}
When the spectrum of $M$ is real, numerical evidence shows that the eventual trajectory of $\seq{\zk}$ is a straight line, which is different from the case where both functions are non-smooth. 
If moreover $o(\norm{\vkm})$ vanishes fast enough, we can prove that $\theta_k \to 0$. 
%
%

It should be emphasized that the trajectory of $\seq{\zk}$ is determined by the property of the \emph{leading eigenvalue} of $\mADMM$. Therefore, for $\gamma \leq \norm{(A^T A)^{-\frac{1}{2}} {\nabla^2 R(\xsol)} (A^T A)^{-\frac{1}{2}}}$, though $\mADMM$ will have complex eigenvalues, the leading one is not necessarily to be complex. As a result, the trajectory of $\seq{\zk}$ could be either spiral (complex leading eigenvalue) or straight line (real leading eigenvalue). 
%
\end{remark}


In Figure \ref{fig:trajectory-admm}, we present two examples of the trajectory of ADMM. Subfigure (a) shows a spiral trajectory in $\bbR^2$ which is obtained from solving a polyhedral feasibility problem; See Section \ref{subsec:feas}. For this case, both $R$ and $J$ are indicator functions of affine subspaces. 
For subfigure (b), we use ADMM to solve a toy LASSO problem in $\bbR^3$ with $\gamma$ chosen such that the eventual trajectory of $\seq{\zk}$ is straight line.


\begin{figure}[!ht]
\centering
\subfloat[Spiral trajectory]{ \includegraphics[width=0.31\linewidth]{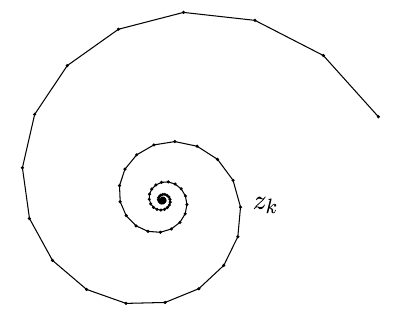} }
{\hspace{2pt}}
\subfloat[Eventual straight line trajectory]{ \includegraphics[width=0.3\linewidth]{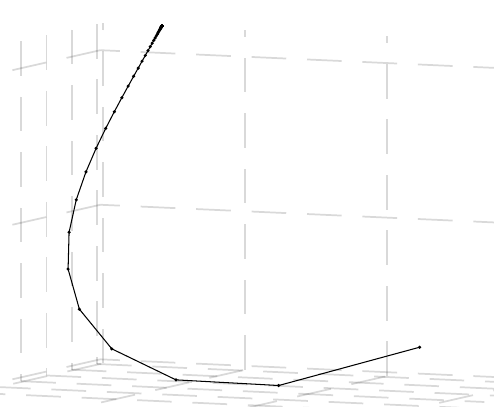} }
{\hspace{2pt}}
\subfloat[Trajectory of shadow sequences]{ \includegraphics[width=0.31\linewidth]{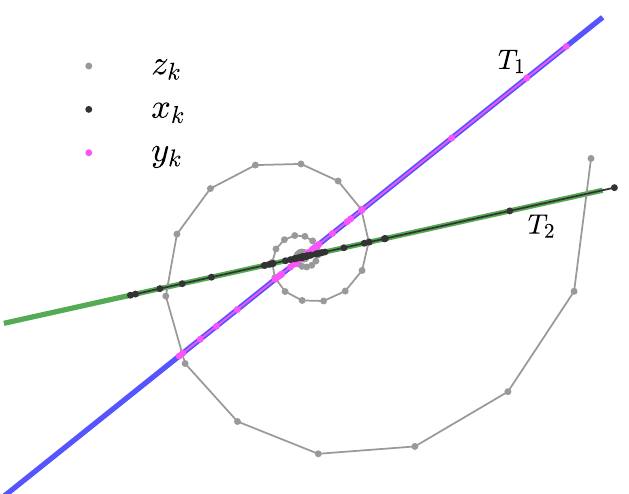} }    \\
\caption{Two different trajectories of ADMM and trajectory of shadow sequences.}
\label{fig:trajectory-admm}
\end{figure}



\begin{remark}[Trajectory of shadow sequences $\xk, \yk$]
We conclude this section by remarking the trajectories of the shadows sequences $\xk, \yk$, which are more complicated than that of $\zk$. 
The trajectories of $\xk, \yk$ depend on both the trajectory of $\zk$ and the properties of the functions $R, J$. For example, for the feasibility problem of Section \ref{subsec:feas}, the trajectory of $\zk$ is a spiral, see Figure \ref{fig:trajectory-admm} (a) and (c). The trajectory of $\yk$ is the projection of the trajectory of $\zk$ onto subspace $T_{1}$, which means that $\yk$ is swinging around $\ysol$ along $T_{1}$. Similar trajectory for $\xk$; see Figure \ref{fig:trajectory-admm} (c).  
The trajectories of $\xk$ and/or $\yk$ can also be spirals when $R$ and/or $J$ are smooth and the trajectory of $\zk$ is spiral. 
When the trajectory of $\zk$ is a straight-line, so are the trajectories of $\xk, \yk$. 
\end{remark}

\section{The failure of inertial acceleration}
\label{sec:fail-inertial}


In this section, we use the LASSO and feasibility problems as examples to demonstrate the effects of applying inertial technique to ADMM, especially when it fails. 
One simple approach to combine inertial technique and ADMM is via the equivalence between ADMM and Douglas--Rachford splitting method. Applying the inertial scheme of \cite[Chapter 4]{liang2016thesis} to the Douglas--Rachford iteration \eqref{eq:fp-dr}, we obtain the following inertial scheme 
\[
\begin{aligned}
\zbark &= \zk + \ak (\zk - \zkm)   ,   \\
\zkp &= \fDR(\zbark) .
\end{aligned}
\]
The above scheme can be reformulated as an instance of inertial Proximal Point Algorithm, guaranteed to be convergent for $\ak < \frac{1}{3}$ \cite{alvarez2001inertial}; We refer to \cite{pejcic2016accelerated} or \cite[Chapter 4.3]{liang2016thesis} for more details. 
Adapting the above scheme to ADMM we obtain the following inertial ADMM (iADMM) 
\beq\label{eq:iadmm}
\begin{aligned}
\xk &= \argmin_{x\in\bbR^n }~ R(x) + \sfrac{\gamma}{2} \norm{Ax - \tfrac{1}{\gamma} \pa{ \zbarkm - 2\psikm } }^2 , \\
\zk &= \psikm + \gamma A\xk   ,   \\
\zbark &= \zk + \ak (\zk - \zkm)   ,   \\
\yk &= \argmin_{y\in\bbR^m }~ J(y) + \sfrac{\gamma}{2} \norm{By + \tfrac{1}{\gamma} \pa{ \zbark - \gamma b } }^2 , \\
\psik &= \zbark + \gamma\pa{B\yk - b}  ,
\end{aligned}
\eeq
which considers only the momentum of $\seq{\zk}$ without any stronger assumptions on $R,J$. 
To our knowledge, there is no acceleration guarantee for \eqref{eq:iadmm}.

\begin{remark} $~$
\begin{itemize}
\item
In the inertial scheme \eqref{eq:iadmm}, we can also consider momentum of more than two points, that is using more points than $\zk,\zkm$ to update $\zbark$. For example, the following three-point momentum can be considered
\[
\zbark = \zk + \ak (\zk - \zkm) + \bk (\zkm - \zkmm)   . 
\]
We shall use the feasibility problem to demonstrate the benefits of the above approach. 

\item
In literature, besides \eqref{eq:iadmm}, other combinations of inertial technique and ADMM are also proposed, see for instance \cite{pejcic2016accelerated,kadkhodaie2015accelerated}. To ensure acceleration guarantees, stronger assumption needs to be imposed, such as Lipschitz smoothness and strong convexity. 
\end{itemize}
\end{remark}

\begin{figure}[!ht]
\centering
\subfloat[$\gamma = \frac{\norm{K}^{2}}{10}: \cos(\theta_k)$]{ \includegraphics[width=0.325\linewidth]{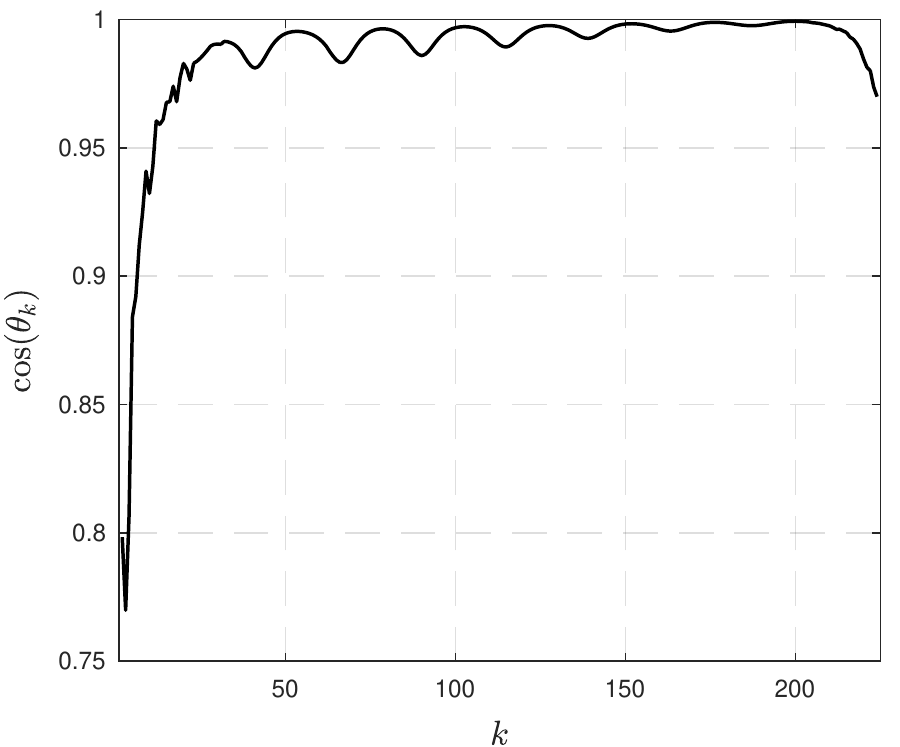} }   \hspace{8pt}
\subfloat[$\gamma = \frac{\norm{K}^{2}}{10}: \norm{\zk-\zsol}$]{ \includegraphics[width=0.325\linewidth]{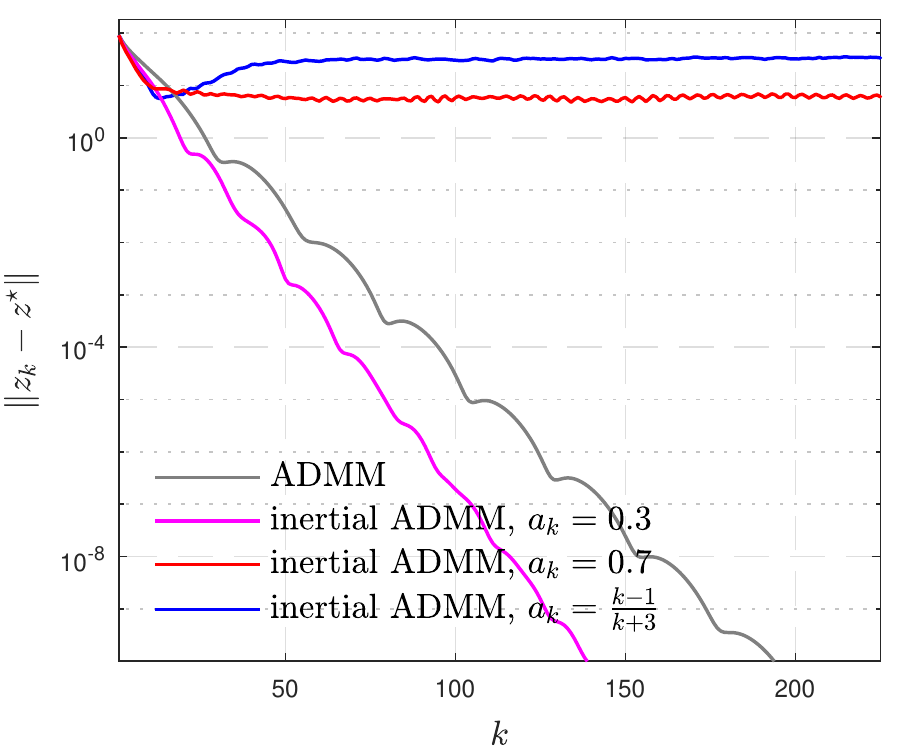} }    \\
\subfloat[$\gamma = \norm{K}^2+0.1: \cos(\theta_k)$]{ \includegraphics[width=0.325\linewidth]{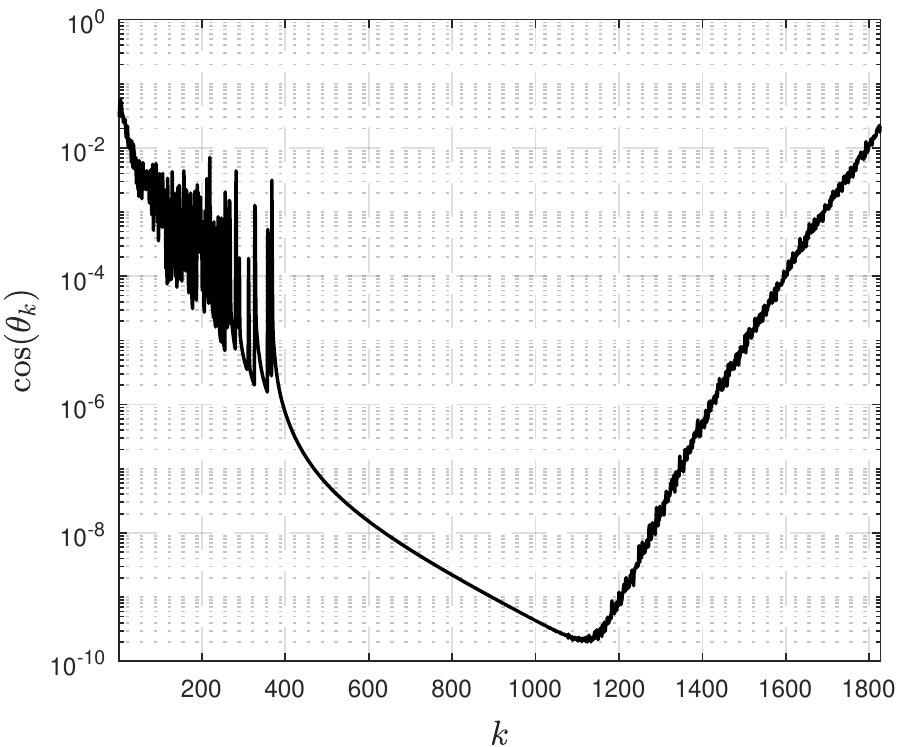} }   \hspace{8pt}
\subfloat[$\gamma = \norm{K}^2+0.1: \norm{\zk-\zsol}$]{ \includegraphics[width=0.325\linewidth]{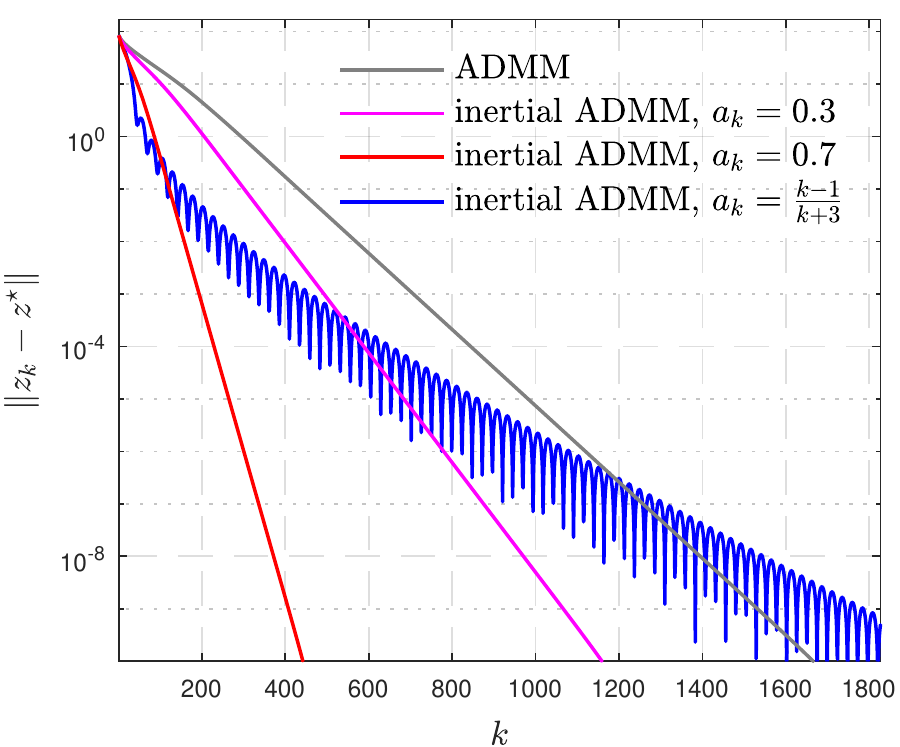} }    \\
\caption{Trajectory of sequence $\seq{\zk}$ and performance of inertial on ADMM. (a) $\cos(\theta_k)$ for $\gamma=\frac{\norm{K}^{2}}{10}$; (b) failure of inertial ADMM on spiral trajectory; (c) $\cos(\theta_k)$ for $\gamma=\norm{K}^2+0.1$; (d) success of inertial ADMM on straight line trajectory.}
\label{fig:tra-iadmm}
\end{figure}

\subsection{LASSO problem}
The formulation of LASSO in the form of \eqref{eq:problem-admm} reads
\beq\label{eq:lasso}
\begin{aligned}
\min_{x, y \in \bbR^n} ~ \mu \norm{x}_1 + \sfrac{1}{2} \norm{Ky - f}^2 \quad\textrm{such~that}\quad x - y = 0  ,
\end{aligned} 
\eeq
where $K \in \bbR^{m\times n} ,~ m < n$ is a random Gaussian matrix. 
%
Since $\frac{1}{2} \norm{K y - f}^2$ is quadratic, owing to Proposition \ref{prop:trajectory-admm-real}, the eventual trajectory of $\seq{\zk}$ is a straight line if $\gamma > \norm{K}^2$, and a spiral for some $\gamma \leq \norm{K}^2$. 
Therefore, we consider two different choices of $\gamma$ which are $\gamma=\frac{\norm{K}^{2}}{10}$ and $\gamma = \norm{K}^2 + 0.1$, and for each $\gamma$, four different choices of $\ak$ are considered
\[
\ak \equiv 0.3 ,\quad
\ak \equiv 0.7 \qandq
\ak = \tfrac{k-1}{k+3}   .
\]
The 3rd choice of $\ak$ corresponds to FISTA \cite{chambolle2015convergence}. 
For the numerical example, we consider $K \in \bbR^{640\times 2048}$ and $\mu = 1$, $f$ is the measurement of an $128$-sparse signal. The results are shown in Figure \ref{fig:tra-iadmm}, 
\begin{itemize}[leftmargin=2em]
	\item Case $\gamma=\frac{\norm{K}^{2}}{10}$: the value of $\cos(\theta_k)$ is plotted in Figure \ref{fig:tra-iadmm}(a), from which we can observed that $\theta_k$ eventually is changing in an interval which implies that the leading eigenvalue of $\mADMM$ is complex, and the trajectory of $\zk$ is a spiral.
		The inertial scheme works only for $\ak\equiv 0.3$, which is due to that fact that the trajectory of $\seq{\zk}$ is a spiral for $\gamma=\frac{\norm{K}^{2}}{10}$. As a result, the direction $\zk-\zkm$ is not pointing towards $\zsol$, hence unable to provide satisfactory acceleration. 
	\item Case $\gamma=\norm{K}^2+0.1$: For this case, the leading eigenvalue of $\mADMM$ is real, hence the trajectory of $\zk$ is straight line. From $k=1000$, the increasing of $\cos(\theta_k)$ is due to machine errors. 
	All choices of $\ak$ work since $\seq{\zk}$ eventually forms a straight line. 
	Among these four choices of $\ak$, $\ak\equiv 0.7$ is the fastest, while $\ak=\frac{k-1}{k+3}$ eventually is the slowest. 
\end{itemize}
It should be noted that, though ADMM is faster for $\gamma=\frac{\norm{K}^{2}}{10}$ than $\gamma = \norm{K}^2+0.1$, our main focus here is to demonstrate how the trajectory of $\seq{\zk}$ affects the outcome of inertial acceleration. 


The above comparisons, particularly for $\gamma=\frac{\norm{K}^{2}}{10}$ imply that the trajectory of the sequence $\seq{\zk}$ is crucial for the acceleration outcome of the inertial scheme. 
Since the trajectories of ADMM depends on the properties of $R, J$ and choice of $\gamma$, this implies that the right scheme that can achieve uniform acceleration despite $R,J$ and $\gamma$ should be able to adapt itself to the trajectory of the method. 

\subsection{Feasibility problem}\label{subsec:feas}

For the LASSO problem, though $\ak = 0.7, \frac{k-1}{k+2}$ fail to provide acceleration for $\gamma = \frac{\norm{K}^{2}}{10}$, $\ak = 0.3$ is marginally faster than the standard ADMM. Below we consider a feasibility problem to demonstrate that \eqref{eq:iadmm} will fail to provide acceleration as long as $\ak > 0$.

\begin{figure}[!ht]
\centering
\subfloat[$\norm{\zk-\zsol}$]{ \includegraphics[width=0.35\linewidth]{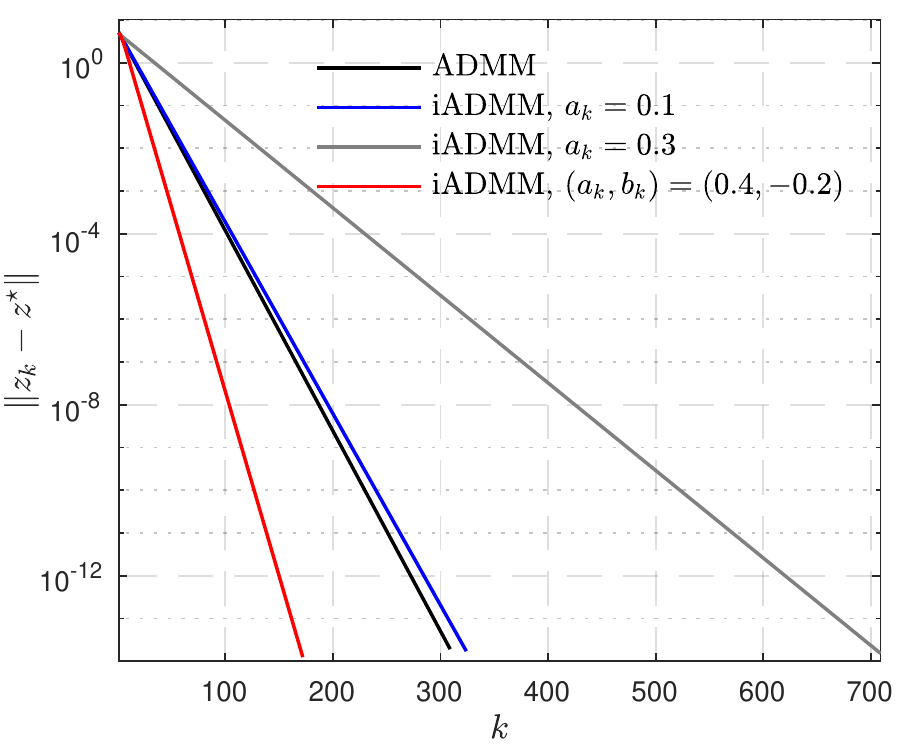} }
{\hspace{2pt}}
\subfloat[Trajectory of $\zk$]{ \includegraphics[width=0.37\linewidth]{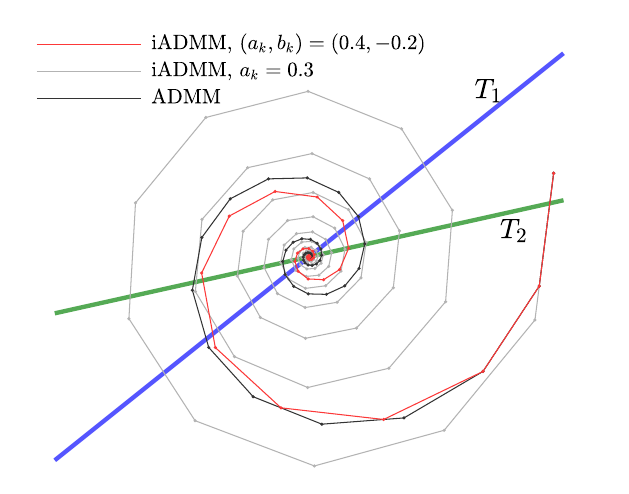} }    \\
\caption{Performance of different schemes and trajectory of $\zk$ for ADMM applying to feasibility problem.}
\label{fig:fail-iadmm-feasibility}
\end{figure}

Let $T_1, T_2$ be two subspaces of $\bbR^2$ such that $T_1 \cap T_2 = \ba{0}$, and consider the following problem
\beq\label{eq:feasibility}
\textrm{find} ~~ x \in \bbR^2 ~\enskip \textrm{such that} \enskip x \in \T_1 \cap T_2  ,
\eeq
which is finding the common point of $T_1$ and $T_2$. The problem can be written as
\[
\min_{x \in \bbR^n , y\in \bbR^m } ~ R(x) + J(y)  \quad
\textrm{such~that} \quad x - y = 0  ,
\]
where $R(x), J(y)$ are the indicator functions of $T_1, T_2$, respectively. Specify \eqref{eq:iadmm} to this case we obtain the following iteration
\[
\begin{aligned}
\xkp&= \proj_{T_1} \Pa{ \pa{ \zbark - 2\psik } / \gamma }  , \\
\zkp &= \psik + \gamma \xkp   ,   \\
\zbarkp &= \zkp + \ak (\zkp - \zk)   ,   \\
\ykp &= \proj_{T_2} \pa{ \zbarkp / \gamma } , \\
\psikp &= \zbarkp - \gamma \ykp  .
\end{aligned}
\]
To demonstrate the failure of inertial ADMM, besides the standard ADMM, the following schemes are considered
\[
\begin{aligned}
\textrm{iADMM}&: \zbark = \zk + 0.1(\zk-\zkm)  , \\
\textrm{iADMM}&: \zbark = \zk + 0.3(\zk-\zkm)  , \\
\textrm{iADMM}&: \zbark = \zk + 0.4(\zk-\zkm) - 0.2(\zkm - \zkmm) . 
\end{aligned}
\]
The performance of the four schemes in terms of $\norm{\zk-\zsol}$ is provided in Figure \ref{fig:fail-iadmm-feasibility} (a). It can be observed that for the inertial schemes with $\ak = 0.1, 0.3$, they are both slower than the standard ADMM, while the last scheme is faster than all the others. 
Such a difference is due to the fact that the trajectory of $\seq{\zk}$ of ADMM is a spiral, see Figure \ref{fig:fail-iadmm-feasibility} (b) the black dot line. As a result, the direction $\zk - \zkm$ is not pointing towards $\zsol$, making the inertial step useless. For the last inertial scheme, since the coefficient of $\zkm - \zkmm$ is negative, the direction $0.4(\zk-\zkm) - 0.2(\zkm - \zkmm)$ points towards $\zsol$, hence providing acceleration.

In Section \ref{sec:failure-inertial}, substantial discussions on why inertial ADMM fails, when the trajectory of $\zk$ is spiral, are provided. Especially for the case when both functions in \eqref{eq:problem-admm} are (locally) polyhedral around the solution, \eg the feasibility problem discussed above.

\section{A$\!^3$DMM: adaptive acceleration for ADMM}
\label{sec:ada-admm}

The previous section shows the trajectory of $\seq{\zk}$  eventually settles onto a regular path \ie either straight line or spiral. In this section, we exploit this regularity to design adaptive acceleration for ADMM, which is called ``A$\!^3$DMM''; See Algorithm \ref{alg:a3dmm}.

The update of $\zbark$ in \eqref{eq:iadmm} can be viewed as a special case of the following extrapolation 
\beq\label{eq:zbarkp-E}
\zbark = \calE(\zk, \zkm, \dotsm, z_{k-q}), 
\eeq
for the choice of $q=1$. The idea is: given $\{z_{k-j}\}_{j=0}^{q+1}$, define $v_j\eqdef z_j - z_{j-1}$ and predict the future iterates by considering how the past directions $v_{k-1},\ldots, v_{k-q}$ approximate the latest direction $\vk$. 
In particular, define  $V_{k-1} \eqdef \begin{bmatrix} v_{k-1} , \dotsm , v_{k-q}  \end{bmatrix} \in \bbR^{n \times q}$, and let $c_k\eqdef \argmin_{c\in \RR^q}\norm{V_{k-1} c - v_k}^2 = \norm{ \sum_{j=1}^q c_j v_{k-j} - v_k}^2$. The idea is then that $V_k c_k \approx v_{k+1}$ and so, $\bar z_{k,1} \eqdef z_k + V_k c \approx z_{k+1}$. By iterating this $s$ times, we obtain $\bar z_{k,s}\approx z_{k+s}$.


More precisely, given $c \in \bbR^{q}$, define the mapping $H$ by
$
H(c) 
= 
\left[
\begin{array}{c|c}
c_{1:q-1}  & \Id_{q-1} \\
c_{q} & 0_{1,q-1}
\end{array}
\right]
\in \bbR^{q \times q}  
$.  
Let $C_k = H(c_k)$, note that $V_{k} = V_{k-1} C_k$. Define $\bar V_{k,0} \eqdef V_k$ and for $s \geq 1$,  define
\[
\bar{V}_{k,s} \eqdef \bar{V}_{k,s-1} C_k \eqdef V_{k} C_k^s ,
\]
where $C_k^s$ is the power of $C_k$. 
Let $(C)_{(:, 1)}$ be the first column of matrix $C$, then
\beq\label{eq:zk_s}
\begin{aligned}
\zbar_{k,s} 
= \zk + \msum_{i=1}^{s} (\bar{V}_{k,i})_{(:, 1)} 
= \zk + \msum_{i=1}^{s} V_{k} (C_{k}^{i})_{(:, 1)}  
&= \zk + V_{k} \bPa{ \msum_{i=1}^{s} C_{k}^i }_{(:, 1)}  ,
\end{aligned}
\eeq
which is the desired trajectory following extrapolation scheme. Define the extrapolation parameterized by $s, q$ as 
\[
\calE_{s, q}(\zk, \dotsm, z_{k-q-1}) 
\eqdef  V_{k} \bPa{ \msum_{i=1}^{s} C_{k}^i }_{(:, 1)}  ,
\]
we obtain the following trajectory following adaptive acceleration for ADMM.

\begin{center}
\begin{minipage}{0.995\linewidth}
\begin{algorithm}[H]
\caption{A$\!^3$DMM - Adaptive Acceleration for ADMM} \label{alg:a3dmm}
{\noindent{\bf{Initial}}}: Let $s\geq1, q \geq 1$ be integers and $\bar{q} = q + 1$. Let $\zbar_0 = z_0 \in \bbR^p$ and $V_{0} = 0 \in \bbR^{p \times q}$. \\ 
{\noindent{\bf{Repeat}}}: 
\begin{itemize}[leftmargin=2em]
\item For $k \geq 1$:  \vspace{-11pt}
\beq\label{eq:a3dmm}
\begin{aligned}
\yk &= \argmin_{y\in\bbR^m }~ J(y) + \sfrac{\gamma}{2} \norm{By + \tfrac{1}{\gamma} \pa{ \zbarkm - \gamma b } }^2 , \\
\psik &= \zbarkm + \gamma\pa{B\yk - b}  , \\
\xk &= \argmin_{x\in\bbR^n }~ R(x) + \sfrac{\gamma}{2} \norm{Ax - \tfrac{1}{\gamma} \pa{ \zbarkm - 2\psik } }^2 , \\
\zk &= \psik + \gamma A\xk   ,   \\
\vk &= \zk - \zkm \qandq V_{k} = [\vk, V_{k-1}(:, 1:q-1)]  .
\end{aligned}
\eeq

\item \textrm{If $\textrm{mod}(k, \bar{q})=0$: } \textrm{Compute $C_k$ as described above, if $\rho(C_k)<1$:}  
			$$\zbark = \zk + \ak \calE_{s, q}(\zk, \dotsm, z_{k-q-1}).$$
\item \textrm{If $\textrm{mod}(k, \bar{q})\neq0$: } $ \zbark = \zk $. 
\end{itemize}
{\noindent{\bf{Until}}}: $\norm{\vk} \leq \tol$.
\end{algorithm}
\end{minipage}
\end{center}



\begin{remark} $~$
\begin{itemize}[leftmargin=2em, topsep=-4pt, itemsep=2pt]
\item When $\textrm{mod}(k, \bar{q})\neq0$, one can also consider $\zbark = \zk + \ak(\zk-\zkm)$ with properly chosen $\ak$. In stead of $\bar{q} = q + 1$, one can also consider $\bar{q} = q + i$ with $i\in\bbN_{+}$.  



\item A$\!^3$DMM carries out $\bar{q}$ standard ADMM iterations to set up the extrapolation step $\calE_{s, q}$. As $\calE_{s, q}$ contains the sum of the powers of $C_k$, it is guaranteed to be convergent when $\rho(C_k) < 1$. Therefore, we only apply $\calE_{s, q}$ when the spectral radius $\rho(C_k) < 1$ is true. In this case, there is a closed form expression for $\calE_{s, q}$ when $s=+\infty$; See Eq. \eqref{eq:lp-mpe}. 

\item The purpose of adding $\ak$ in front of $\calE_{s, q}(\zk, \dotsm, z_{k-q-1})$ is so that we can control the value of $\ak$ to ensure the convergence of the algorithm; See below the discussion. 
\end{itemize}
\end{remark}

In Algorithm \ref{alg:a3dmm}, we change the order of updates so that the extrapolation step only needs to be carried out on $\zk$. This is due to the fact, the update of $\yk$ only depends on $\zk$, and such an arrangement requires the minimal computational overhead. 
Under such setting, the extra memory cost is $pq$ for the storage of $V_{k}$. The extra computational cost of A$^{3}$DMM is very small, which is about $n q^{2}$ for computing the pseudo-inverse of $V_{k}$. 
Moreover, the value of $q$ usually is taken very small, \eg $q \leq 10$, therefore the total overheads of A$\!^3$DMM is rather limited.


\subsection{Convergence of A$\!^3$DMM}

To discuss the convergence of A$\!^3$DMM, we shall treat the algorithm as a perturbation of the original ADMM. 
If the perturbation error is absolutely summable, then we obtain the convergence of A$\!^3$DMM. 
More precisely, let $\varepsilon_k \in \bbR^n$ whose value takes
\[
\varepsilon_k =
\left\{
\begin{aligned}
0 &: \mathrm{mod}(k, \bar{q}) \neq 0  \enskip\textrm{or}\enskip  \mathrm{mod}(k, \bar{q}) = 0 ~~\&~~ \rho(C_k)\geq1, \\
\ak\calE_{s, q}(\zk, \dotsm, z_{k-q-1})  &: \mathrm{mod}(k, \bar{q}) = 0  \enskip\&\enskip \rho(C_k) < 1  .
\end{aligned}
\right. 
\] 
Suppose the fixed-point formulation of ADMM can be written as $\zk = \calF (\zkm)$ for some $\calF$ (see Section \ref{P-sec:tra-admm} of the appendix for details). 
Then Algorithm \ref{alg:a3dmm} can be written as
\beq\label{eq:perturbation-km}
\zk = \calF(\zkm + \varepsilon_{k-1}) .
\eeq
Owing to \eqref{eq:perturbation-km}, we can obtain the following convergence for Algorithm \ref{alg:a3dmm} which is based on the classic convergence result of inexact \KM fixed-point iteration \cite[Proposition 5.34]{bauschke2011convex}.

\begin{proposition}\label{prop:convergence-lp}
For problem \eqref{eq:problem-admm} and Algorithm \ref{alg:a3dmm}, suppose that the conditions \iref{ADMM:RJ}-\iref{ADMM:minimizers} are true.
If moreover, $\sum_{k} \norm{\varepsilon_k} < \pinf$, $\zk \to \zsol \in \fix(\calF) \eqdef \ba{z \in \bbR^p: z = \calF(z)}$ and $(\xk,\yk,\psik)$ converges to $(\xsol, \ysol, \psisol)$ which is a saddle point of $\calL(x, y; \psi)$. 
\end{proposition}

\paragraph{On-line updating rule}
The summability condition $\sum_{k} \norm{\varepsilon_k} < \pinf$ in general cannot be guaranteed. However, it can be enforced by a simple online updating rule. Let $a \in [0, 1]$ and $ b, \delta > 0$, then $\ak$ can be determined by $\ak = \min\Ba{ a, {b}/\pa{ k^{1+\delta} \norm{ \zk-\zkm }  } }$. 

\paragraph{Inexact A$\!^3$DMM}
Observe that in A$\!^3$DMM, when $A, B$ are non-trivial, in general there are no closed form solutions for $\xk$ and $\yk$. 
Take $\xk$ for example, suppose it is computed approximately, then in $\zk$ there will be another approximation error $\varepsilon_k'$, and consequently
\[
\zk = \calF(\zkm + \varepsilon_{k-1} + \gamma \varepsilon_{k-1}') .
\]
If there holds $\sum_{k} \norm{\varepsilon_{k-1}'} < \pinf$, Proposition \ref{prop:convergence-lp} remains true for the above perturbation form.

\subsection{Acceleration guarantee for A$\!^3$DMM}\label{sec:acc-gua}

%
We have so far alluded to the idea that the extrapolated point $\bar z_{k,s}$ defined in \eqref{eq:zk_s} (which depends only on $\{z_{k-j}\}_{j=0}^q$) is an approximation to $z_{k+s}$. In this section, we make precise this statement. 

\paragraph{Relationship to MPE and RRE}
We first show that $\bar z_{k,\infty}$ is (almost) equivalent to MPE. 
Recall that given a square matrix $C$, if its Neumann series is convergent, then there holds $(\Id-C)^{-1} = \sum_{i=0}^{\pinf} C^i$.  
For the summation of the power of $C_{k}$ in \eqref{eq:zk_s}, when $s=\pinf$, we have
\[
\msum_{i=1}^{\pinf} C_{k}^i 
= C_{k} \msum_{i=0}^{\pinf} C_{k}^i 
= C_{k} (\Id - C_{k})^{-1}
= (\Id - C_{k})^{-1} - \Id .
\]
Back to \eqref{eq:zk_s}, then we get
\beq\label{eq:lp-mpe}
\begin{aligned}
\zbar_{k,\infty} 
&\eqdef \zk + V_{k} \Pa{ (\Id - C_{k})^{-1} - \Id }_{(:, 1)}  
= \zk - \vk + V_{k} \Pa{ (\Id - C_{k})^{-1} }_{(:, 1)}   \\
&= \zkm + V_{k} \Pa{ (\Id - C_{k})^{-1} }_{(:, 1)}   
= \sfrac{1}{1- \sum_{i=1}^s c_{k,i}} \Pa{ z_k - \msum_{j=1}^{q-1} c_{k,j} z_{k-j}} ,
\end{aligned}
\eeq
which turns out to be MPE, with the slight difference of taking the weighted sum of $\ba{z_j}_{j=k-q+1}^k$ as opposed to the weighted sum of $\ba{z_j}_{j=k-q}^{k-1}$ {(See appendix for more details of MPE)}.  
Note that if the coefficients $c$ is computed in the following way:
 $b\in\argmin_{a\in\RR^{q+1}, \sum_j a_j =1}\norm{ \sum_{j=0}^{q} a_j v_{k-j} }$ and $b_0 \neq 0$ and define $c_j\eqdef - b_j/b_0$ for $j=1,\ldots, q$. Then,
$$
\Pa{ 1-\msum_{i=1}^{q} c_i }^{-1} =\sfrac{b_0}{ b_0 + \sum_{j=1}^{q} b_j} = b_0,
$$
and $\zbar_{k,\infty} = \sum_{j=0}^{q-1} b_j z_{k-j}$ is precisely the RRE update (again with the slight difference of summing over iterates shifted by one iteration).

%
%
%

\paragraph{Acceleration guarantee for A$\!^3$DMM}
%
Let $\seq{\zk}$ be a sequence in $\bbR^n$ and let $\vk \eqdef \zk-\zkm$. Assume that 
$v_k = M v_{k-1}  $ for some $M \in \bbR^{n \times n}$. Denote $\lambda(M)$ the spectrum of $M$.   
The following proposition provides control on the extrapolation error for $\bar z_{k,s}$ from \eqref{eq:zk_s}.

\begin{proposition}\label{prop:extrap-error}
Define the coefficient fitting error by $\epsilon_k \eqdef \min_{c\in \RR^q} \norm{V_{k-1} c - v_{k}}$.
\begin{enumerate}[label={\rm (\roman{*})}, leftmargin=2em]
\item
For $s\in \NN$,
we have
\begin{equation}\label{eq:err-finite}
\norm{\zbar_{k,s} - \zsol} \leq \norm{z_{k+s} - \zsol} + B_s \epsilon_k ,  
\end{equation}
where
$B_s \eqdef \sum_{\ell=1}^s \norm{M^\ell} \abs{ \sum_{i=0}^{s-\ell} (C_k^i)_{(1,1)}}$. If $\rho(M)<1$ and $\rho(C_k)<1$, then  $\sum_i c_{k,i}\neq 1$ and $B_s$ is uniformly bounded in $s$. For $s=+\infty$, $B_\infty\eqdef \abs{1-\sum_i c_{k,i}}^{-1} \sum_{\ell=1}^\infty \norm{M}^\ell $
\item
Suppose that $M$ is diagonalizable. Let $(\lambda_j)_j$ denote its distinct eigenvalues ordered such that $\abs{\lambda_j}\geq \abs{\lambda_{j+1}}$ and $\abs{\lambda_1}= \rho(M)<1$. Suppose that $\abs{\lambda_{q}}>\abs{\lambda_{q+1}}$. 
\begin{itemize}[leftmargin=2.25em]
\item
Asymptotic bound (fixed $q$ and as $k\to +\infty$): $\epsilon_k =\Oo\pa{\abs{\lambda_{q+1}}^k}$. 
\item
Non-asymptotic bound (fixed $q$ and $k$):
Suppose that $\lambda(M)$ is real-valued and contained in the interval  $[\alpha,\beta]$ with $-1<\alpha<\beta<1$.
Then,
\begin{equation}\label{eq:opt-err}
\qfrac{\epsilon_k}{1-\sum_i c_{k,i}} \leq K \beta^{k-q} \Pa{\tfrac{\sqrt{\eta}-1}{\sqrt{\eta}+1}}^q
\end{equation}
where
  $K\eqdef {2   \norm{z_0-\zsol}}\norm{(\Id-M)^{\frac12}}$ and $\eta = \frac{1-\alpha}{1-\beta}$. 
\end{itemize}
\end{enumerate}
\end{proposition}

\begin{remark}$~$
\begin{itemize}[leftmargin=2em]
\item From Theorem \ref{prop:trajectory-admm}(ii), when $R$ and $J$ are both polyhedral, we have a perfect local linearisation with the corresponding linearisation matrix being normal and hence, the conditions of Proposition \ref{prop:extrap-error} holds for all $k$ large enough.
The first bound (i) shows that the extrapolated point $\zbar_{k,s}$ moves along the true trajectory as $s$ increases, up to the fitting error $\epsilon_k$. Although $\zbar_{k,\infty}$ is essentially an MPE update which is known to satisfy  error bound \eqref{eq:opt-err} (see \cite{sidi2017vector}), this proposition offers a further interpretation of these extrapolation methods in terms of following the ``sequence trajectory'', and combined with our local analysis of ADMM, provides justification of these methods for the acceleration of non-smooth optimisation problems.

\item Proposition \ref{prop:extrap-error} (ii) shows that extrapolation improves the convergence rate from $\Oo(\abs{\lambda_1}^k)$ to $\Oo(\abs{\lambda_{q+1}}^k)$, and the non-asymptotic bound shows that the improvement of extrapolation is optimal in the sense of Nesterov \cite{nesterov83}. Recalling the form of the eigenvalues of $M$ from Theorem \ref{prop:trajectory-admm}, in the case of two non-smooth polyhedral terms,
we must have $\abs{\lambda_{2j-1}} = \abs{\lambda_{2j}} > \abs{\lambda_{2j+1}}$ for all $j\geq 1$. Hence, no acceleration can be guaranteed or observed when $q=1$, while the choice of $q=2$ provides guaranteed acceleration. 
\end{itemize}

   \end{remark}

\section{Discussions}
\label{sec:discussion}

In this section, two variants of ADMM, including the relaxed ADMM and symmetric ADMM which updates $\psi$ twice every iteration, are discussed. 
An extension of A$^3$DMM Algorithm \ref{alg:a3dmm} to these variants is provided at the end of the section.

\subsection{Variants of ADMM}

\paragraph{Relaxed ADMM}
In the literature, a popular variant of ADMM is the \emph{relaxed ADMM} which takes the following iteration procedure:  
\beq\label{eq:admm_relax}
\begin{aligned}
\xk &= \argmin_{x\in\bbR^n }~ R(x) + \tfrac{\gamma}{2} \norm{Ax+B\ykm-b + \tfrac{1}{\gamma}\psikm }^2 , \\
\xbark &= \phi A\xk - (1-\phi)(B\ykm-b) , \\
\yk &= \argmin_{y\in\bbR^m }~ J(y) + \tfrac{\gamma}{2} \norm{\xbark+By-b + \tfrac{1}{\gamma}\psikm }^2 , \\
\psik &= \psikm + \gamma\pa{\xbark + B\yk - b}  ,
\end{aligned}
\eeq
where $\phi \in [0, 2]$ is the relaxation parameter. The above iteration is called \emph{over-relaxed} ADMM for $\phi \in ]1, 2]$. 

In its dual form, the relaxed ADMM is equivalent to the \emph{relaxed} Douglas--Rachford splitting applied to solve \eqref{eq:problem-admm-dual}, see Section \ref{subsec:admm-dr}. The convergence of \eqref{eq:admm_relax} can be guaranteed for $\phi \in ]0, 2[$ \cite{bauschke2011convex}.
Similar to \eqref{eq:admm2}, define $\zk \eqdef \psikm + \gamma \xbark$, we can rewrite the relaxed ADMM into the following form 
\beq\label{eq:admm_relax2}
\begin{aligned}
\xk &= \argmin_{x\in\bbR^n }~ R(x) + \sfrac{\gamma}{2} \norm{Ax - \tfrac{1}{\gamma} \pa{ \zkm - 2\psikm } }^2 , \\
\zk &= \psikm + \gamma \Pa{ \phi A\xk - (1-\phi)(B\ykm-b) }   ,   \\
\yk &= \argmin_{y\in\bbR^m }~ J(y) + \sfrac{\gamma}{2} \norm{By + \tfrac{1}{\gamma} \pa{ \zk - \gamma b } }^2 , \\
\psik &= \zk + \gamma\pa{B\yk - b}  .
\end{aligned}
\eeq
We can easily adapt the result of the previous sections to the relaxed ADMM via $\zk$.

\begin{remark}
Similar to inertial acceleration, the performance of relaxation also depends on the trajectory of the sequence of $\zk$. For example, when both $R, J$ are (locally) polyhedral around $\xsol, \ysol$, the (eventual) trajectory of $\zk$ is spiral, according to \cite{Bauschke14} the (eventual) optimal relaxation parameter $\phi$ is $1$, that is no relaxation provides the best performance. 
\end{remark}

\paragraph{Symmetric ADMM} 
As aforementioned, the ADMM iteration \eqref{eq:admm} is equivalent to applying Douglas--Rachford splitting to the dual problem \eqref{eq:problem-admm-dual} \cite{gabay1983chapter}. 
It is also pointed out in \cite{gabay1983chapter} that, if the Peaceman--Rachford splitting method \cite{peaceman1955numerical} is applied to solve \eqref{eq:problem-admm-dual}, then it leads to the following iteration in the primal form 
\beq\label{eq:admm_symmetric}
\begin{aligned}
\xk &= \argmin_{x\in\bbR^n }~ R(x) + \tfrac{\gamma}{2} \norm{Ax+B\ykm-b + \tfrac{1}{\gamma}\psikm }^2 , \\
\psikmh &= \psikm + \gamma\pa{A\xk + B\ykm - b} , \\ 
\yk &= \argmin_{y\in\bbR^m }~ J(y) + \tfrac{\gamma}{2} \norm{A\xk+By-b + \tfrac{1}{\gamma}\psikmh }^2 , \\
\psik &= \psikmh + \gamma \pa{A\xk + B\yk - b}  ,
\end{aligned}
\eeq
which is also called \emph{symmetric ADMM}. A brief derivation is provided in Section \ref{subsec:admm-pr}, and we refer to \cite{gabay1983chapter,he2014strictly} and the references therein for more detailed discussions.

In general, the conditions needed for the convergence of \eqref{eq:admm_symmetric} is stronger than the standard ADMM \eqref{eq:admm}, which is due to the fact that stronger conditions are needed to guarantee the convergence of Peaceman--Rachford splitting method \cite{gabay1983chapter}. However, when \eqref{eq:admm_symmetric} converges, it tends to provide faster performance than \eqref{eq:admm}. Similar to \eqref{eq:admm2}, if we define $\zk = \psik - \gamma B\yk + \gamma b = \psikmh + \gamma A\xk$, then \eqref{eq:admm_symmetric} is equivalent to 
\beq\label{eq:admm_symmetric2}
\begin{aligned}
\xk &= \argmin_{x\in\bbR^n }~ R(x) + \tfrac{\gamma}{2} \norm{Ax + \tfrac{1}{\gamma}(2\psikm - \zkm) }^2 , \\
\zk &= \psikm + \gamma\pa{2A\xk + B\ykm - b} , \\
\yk &= \argmin_{y\in\bbR^m }~ J(y) + \tfrac{\gamma}{2} \norm{By + \tfrac{1}{\gamma}(\zk - \gamma b) }^2 , \\
\psik &= \zk + \gamma \pa{B\yk - b}  ,
\end{aligned}
\eeq
which can be written as the fixed-point iteration in terms of $\zk$, see Section \ref{subsec:admm-pr}. 
Suppose the iteration is convergent, then following the analysis of Section \ref{sec:trajectory-admm}, we can obtain the trajectory property of symmetric ADMM in terms of the fixed-point sequence $\zk$. 

\subsection{An extension of A$^3$DMM}\label{subsec:implementation}

From the above discussions, we have that relaxed ADMM \eqref{eq:admm_relax2} and symmetric ADMM \eqref{eq:admm_symmetric2} only differ from the standard ADMM on the update of $\zk$. As a result of these similarities, we can easily extend the A$^{3}$DMM Algorithm \ref{alg:a3dmm} to these variants.

Let
\[
\zk = \calZ(\gamma, \phi; \xk, \ykm, \psikm) 
\]
represent the way of updating $\zk$ in \eqref{eq:admm2}, \eqref{eq:admm_relax2} and \eqref{eq:admm_symmetric2}. 
By replacing \eqref{eq:a3dmm} in Algorithm \ref{alg:a3dmm} with the following equations
\beq\label{eq:a3dmm_extension}
\begin{aligned}
\yk &= \argmin_{y\in\bbR^m }~ J(y) + \sfrac{\gamma}{2} \norm{By + \tfrac{1}{\gamma} \pa{ \zbarkm - \gamma b } }^2 , \\
\psik &= \zbarkm + \gamma\pa{B\yk - b}  , \\
\xk &= \argmin_{x\in\bbR^n }~ R(x) + \sfrac{\gamma}{2} \norm{Ax - \tfrac{1}{\gamma} \pa{ \zbarkm - 2\psik } }^2 , \\
\zk &= \calZ(\gamma, \phi; \xk, \ykm, \psikm)   ,   \\
\vk &= \zk - \zkm \qandq V_{k} = [\vk, V_{k-1}(:, 1:q-1)]  ,
\end{aligned}
\eeq
we obtain the extension of A$^3$DMM to the variants of ADMM.

%
%
%
%

\section{Numerical experiments}
\label{sec:experiment}

We present numerical experiments on affine constrained minimisation (\eg Basis Pursuit), LASSO, quadratic programming and image processing problems to demonstrate the performance of A$^3$DMM. 
In the numerical comparison below, we mainly compare with the original ADMM and its inertial version \eqref{eq:iadmm} with fixed $\ak\equiv0.3$. For the proposed A$\!^3$DMM, two settings are considered: $(q, s) = (6, 100)$ and $(q, s) = (6, +\infty)$. 
MATLAB source codes for reproducing the results can be found at: \texttt{https://github.com/jliang993/A3DMM}.

\subsection{Affine constrained minimisation}
Consider the following constrained problem
\beq\label{eq:bp}
\min_{x \in \bbR^{n} } ~ R(x)    \quad
\textrm{such~that} \quad Kx = f  .
\eeq
Denote the set $\Omega \eqdef \ba{ x \in \bbR^{n} : K x = f}$, and $\iota_{\Omega}$ its indicator function.  
Then \eqref{eq:bp} can be written as
\beq\label{eq:bp2}
\begin{aligned}
\min_{x, y \in \bbR^{n} } ~ R(x) + \iota_{\Omega}(y)  \quad
\textrm{such~that} \quad x - y = 0  ,
\end{aligned}
\eeq
which is special case of \eqref{eq:problem-admm} with $A=\Id, B=-\Id$ and $b=0$. 
Here $K$ is generated from the standard Gaussian ensemble, and the following three choices of $R$ are considered:
\begin{itemize}[leftmargin=2.5cm]
\item[{\bf $\ell_{1}$-norm}] $(m, n)=(512, 2048)$, solution $\xsol$ is $128$-sparse;
\item[{\bf $\ell_{1,2}$-norm}] $(m, n)=(512, 2048)$, solution $\xsol$ has $32$ non-zero blocks of size $4$;
\item[{\bf Nuclear norm}] $(m, n)=(1448, 4096)$, solution $\xsol$ has rank of $4$.
\end{itemize}
The property of $\seq{\theta_k}$ is shown in Figure \ref{fig:cmp-bp} (a)-(c). Note that the indicator function $\iota_{\Omega}(y)$ in \eqref{eq:bp2} is polyhedral since $\Omega$ is an affine subspace,
\begin{itemize}[leftmargin=2em]
\item As $\ell_{1}$-norm is polyhedral, we have in Figure \ref{fig:cmp-bp}(a) that $\theta_{k}$ is converging to a constant which complies with Theorem \ref{prop:trajectory-admm}(ii). 

\item Since $\ell_{1,2}$-norm and nuclear norm are no longer polyhedral functions, we have that $\theta_k$ eventually oscillates in a range, meaning that the trajectory of $\seq{\zk}$ is an elliptical spiral. 
\end{itemize}
Comparisons of the four schemes are shown below in Figure \ref{fig:cmp-bp} (d)-(f):
\begin{itemize}[leftmargin=2em]
\item Since both functions in \eqref{eq:bp2} are non-smooth, the eventual trajectory of $\seq{\zk}$ for ADMM is spiral. Inertial ADMM fails to provide acceleration locally.
\item A$\!^3$DMM is faster than both ADMM and inertial ADMM. For the two different settings of A$\!^3$DMM, their performances are very close. 
\end{itemize}


\begin{figure}[!ht]
	\centering
	\subfloat[$\ell_{1}$-norm]{ \includegraphics[width=0.3\linewidth]{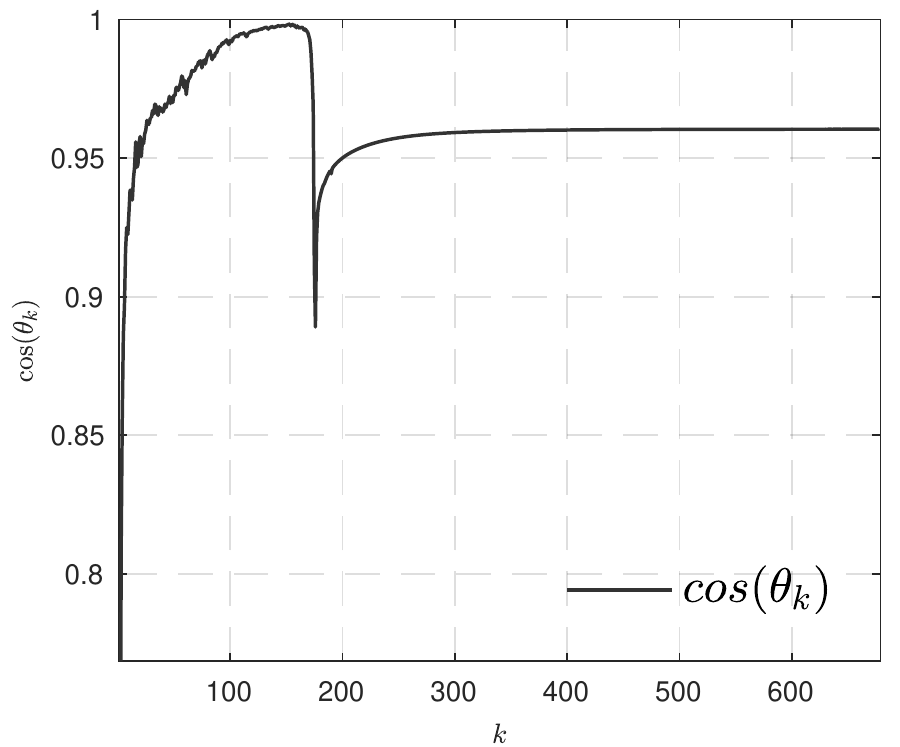} }  \hspace{2pt}
	\subfloat[$\ell_{1,2}$-norm]{ \includegraphics[width=0.3\linewidth]{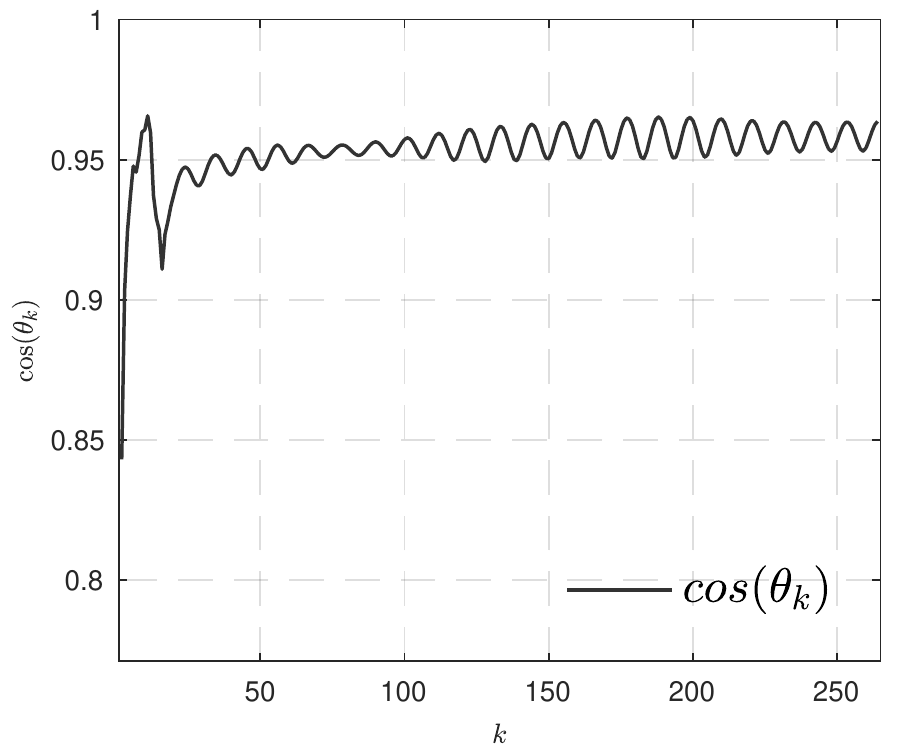} }  \hspace{2pt}
	\subfloat[Nuclear norm]{ \includegraphics[width=0.3\linewidth]{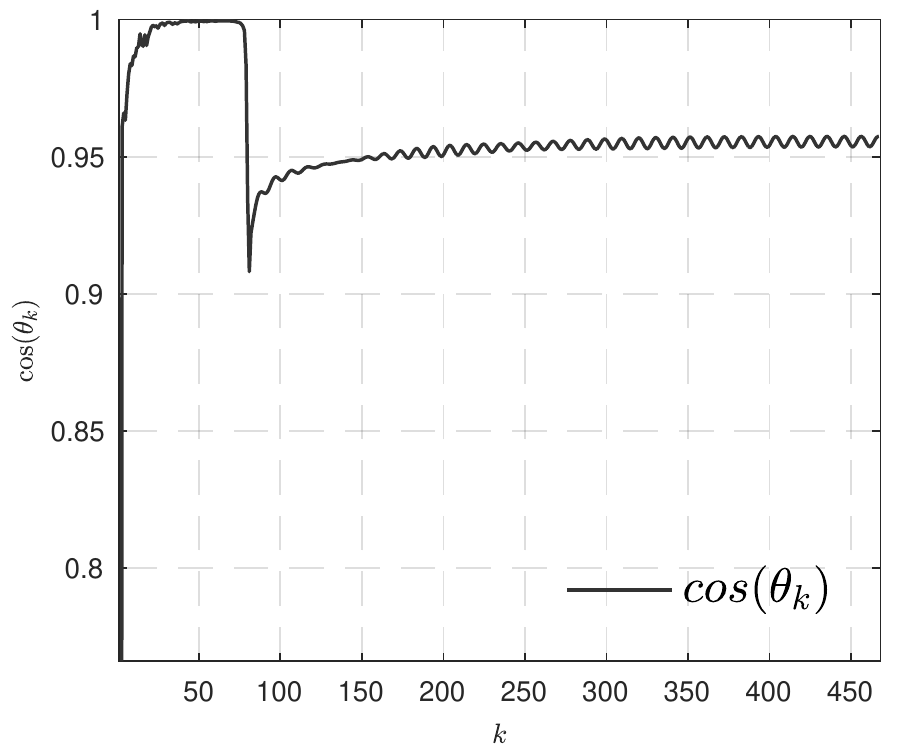} }    \\	
	\subfloat[$\ell_{1}$-norm]{ \includegraphics[width=0.3\linewidth]{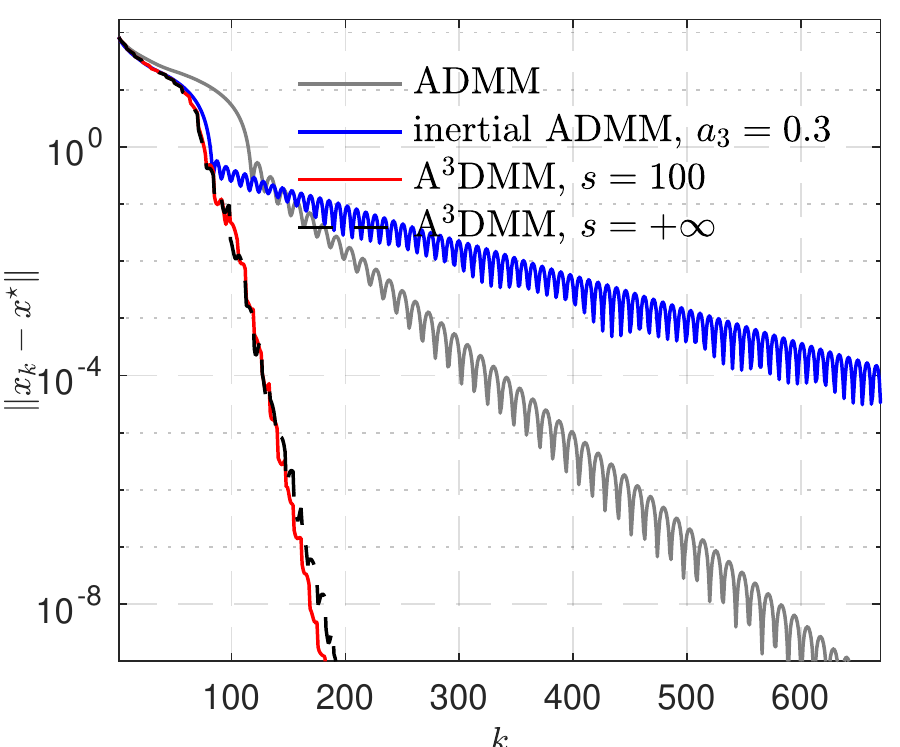} }  \hspace{2pt}
	\subfloat[$\ell_{1,2}$-norm]{ \includegraphics[width=0.3\linewidth]{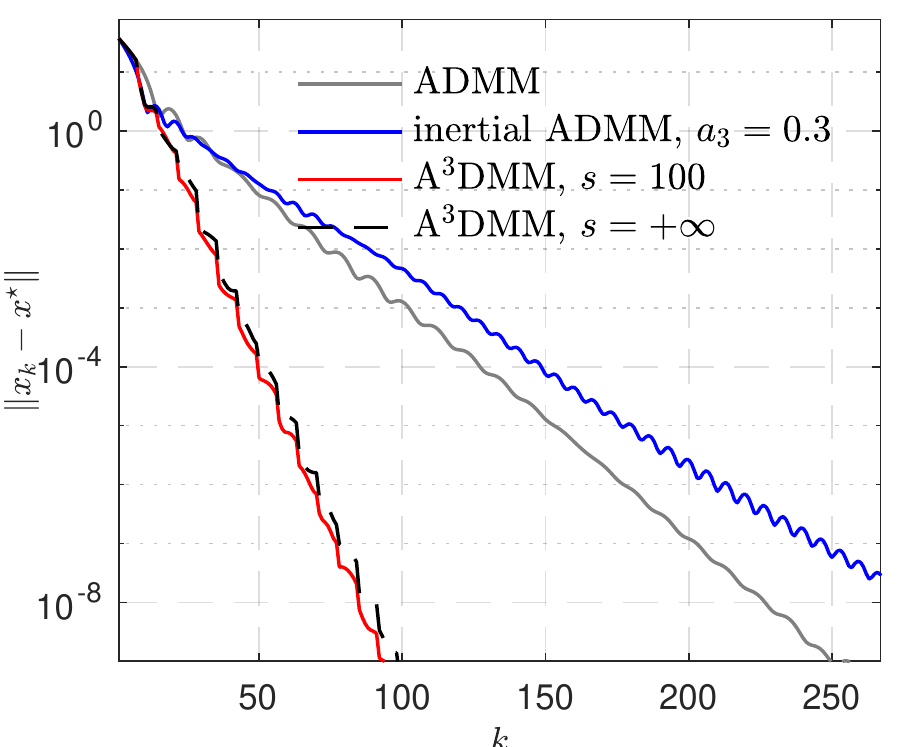} }  \hspace{2pt}
	\subfloat[Nuclear norm]{ \includegraphics[width=0.3\linewidth]{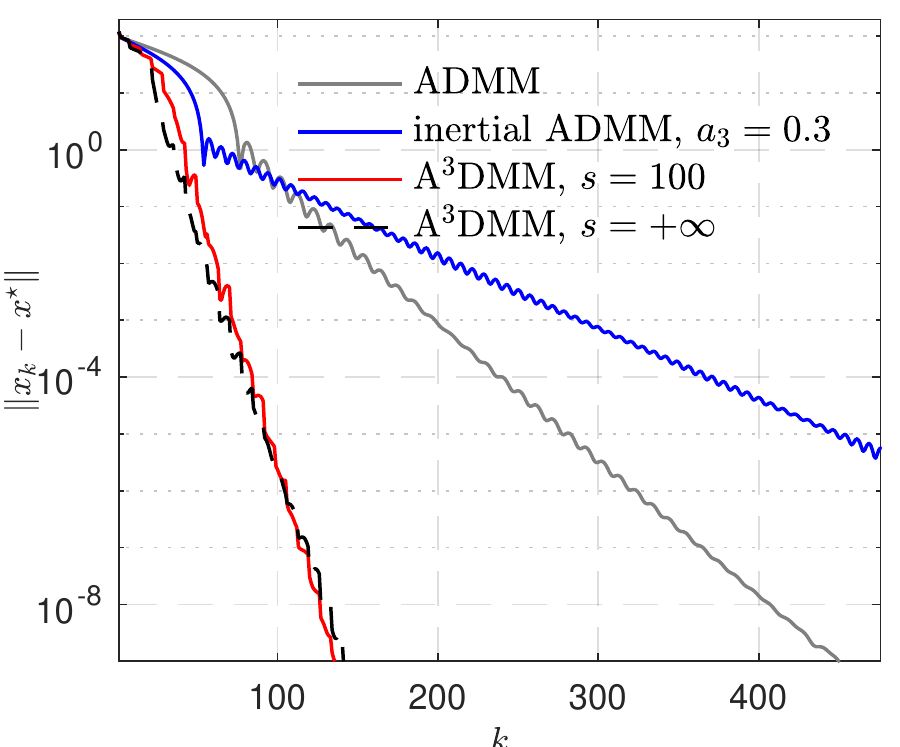} }    \\
	\caption{Performance comparisons and $\seq{\theta_k}$ of ADMM for affine constrained problem.} 
	\label{fig:cmp-bp}
\end{figure}

\subsection{LASSO} 
We consider again the LASSO problem \eqref{eq:lasso} with three datasets from LIBSVM\footnote{\url{https://www.csie.ntu.edu.tw/~cjlin/libsvmtools/datasets/}}. The numerical experiments are provided below in Figure \ref{fig:cmp-lasso}. 

%

It can be observed that the proposed A$\!^3$DMM is significantly faster than the other schemes, especially for $s = +\infty$. 
Between ADMM and inertial ADMM, different from the previous example, the inertial technique can provided consistent acceleration for all three examples.


\begin{figure}[!ht]
	\centering
	\subfloat[\texttt{covtype}: $1-\cos(\theta_k)$]{ \includegraphics[width=0.3\linewidth]{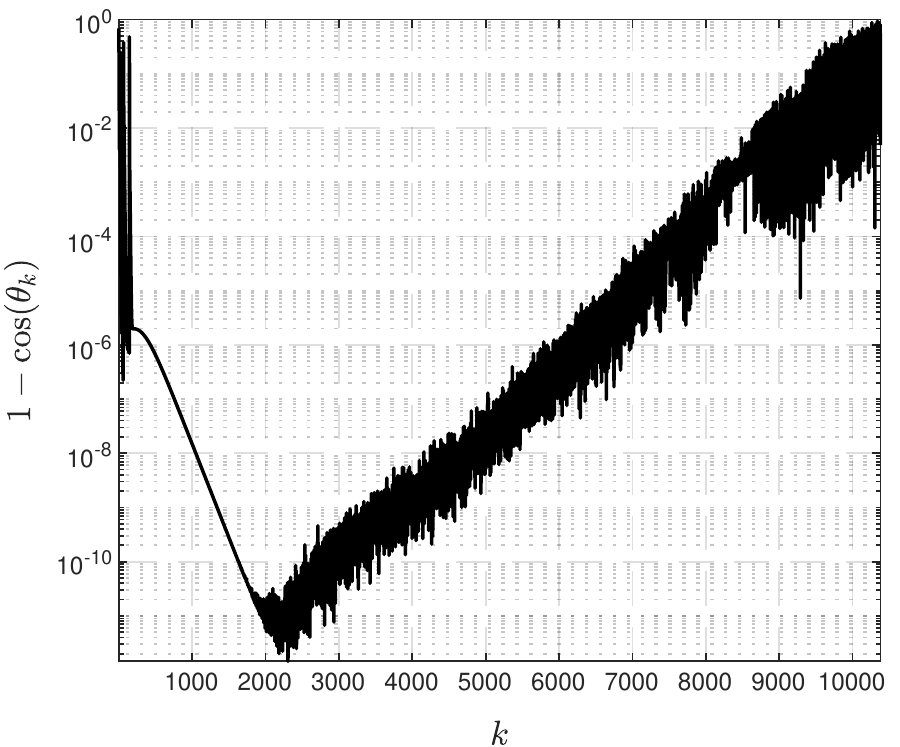} }   \hspace{2pt}
	\subfloat[\texttt{ijcnn1}: $1-\cos(\theta_k)$]{ \includegraphics[width=0.3\linewidth]{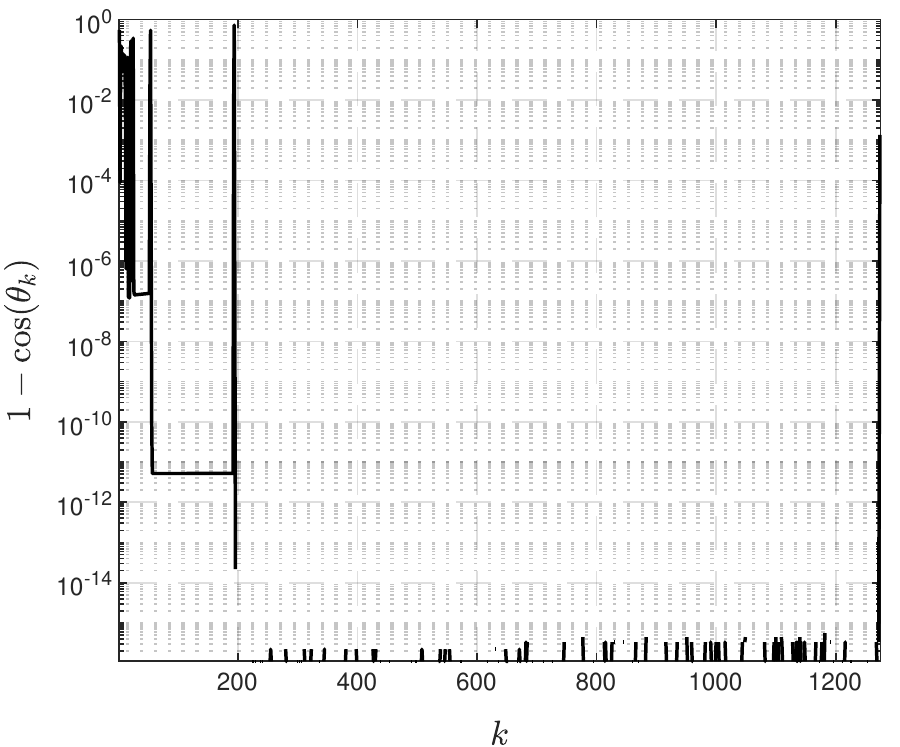} }   \hspace{2pt}
	\subfloat[\texttt{phishing}: $1-\cos(\theta_k)$]{ \includegraphics[width=0.3\linewidth]{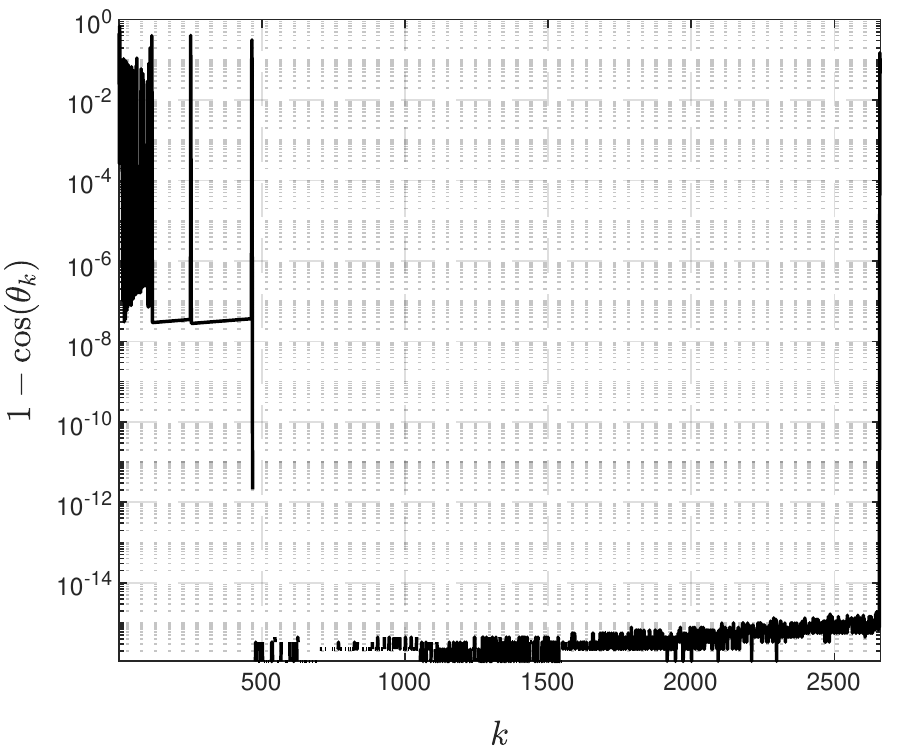} }  \\
\centering
	\subfloat[\texttt{covtype}: $\norm{\xk-\xsol}$]{ \includegraphics[width=0.3\linewidth]{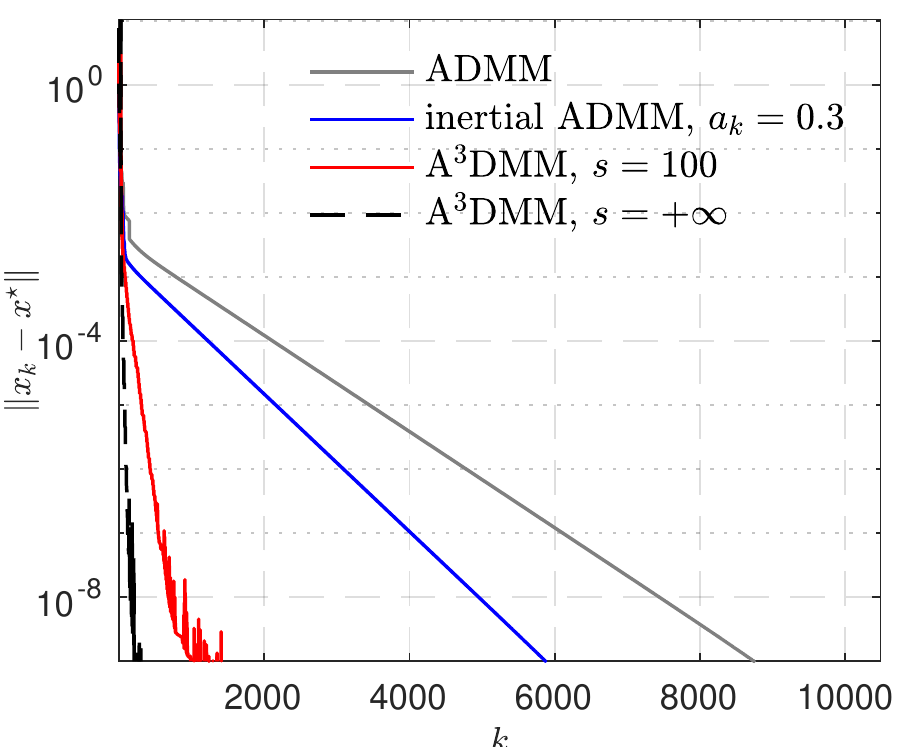} }   \hspace{2pt}
	\subfloat[\texttt{ijcnn1}: $\norm{\xk-\xsol}$]{ \includegraphics[width=0.3\linewidth]{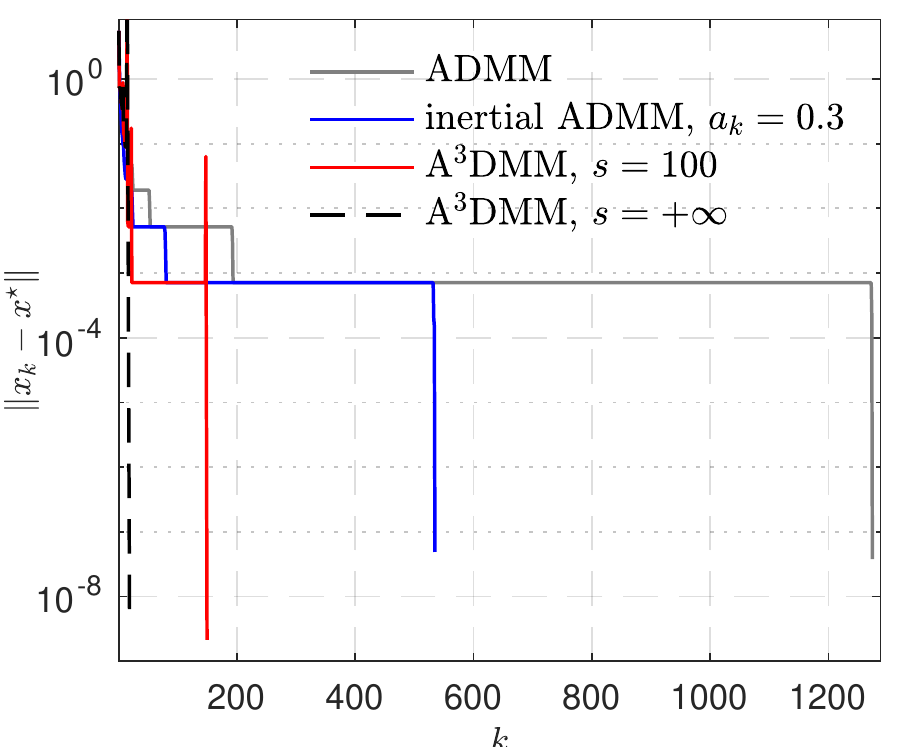} }   \hspace{2pt}
	\subfloat[\texttt{phishing}: $\norm{\xk-\xsol}$]{ \includegraphics[width=0.3\linewidth]{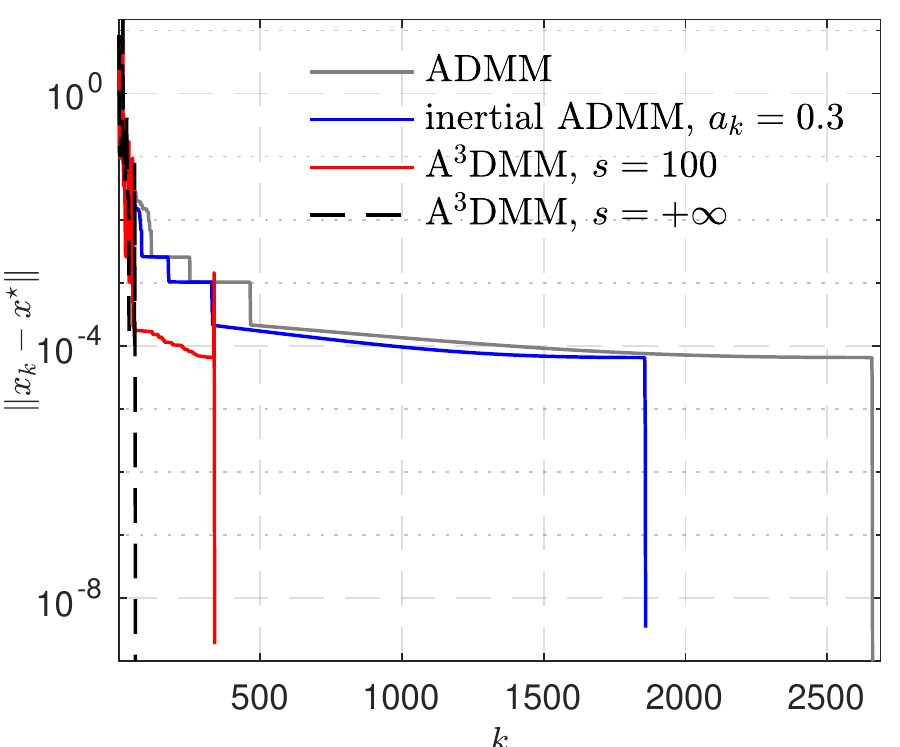} }  \\
	\caption{Performance comparisons for LASSO problem.}
	\label{fig:cmp-lasso}
\end{figure}

\subsection{Quadratic programming}

Consider the following quadratic optimisation problem
\beq\label{eq:qp}
\begin{aligned}
\min_{x \in \bbR^n} &\quad \sfrac{1}{2} x^T Q x + \iprod{q}{x}   ,  \\
\textrm{such~that} &\quad  x_{i} \in [\ell_{i}, r_{i}] ,~~~  i = 1, ..., n .
\end{aligned}
\eeq
Define the constraint set $\Omega = \ba{x \in \bbR^n: x_{i} \in [\ell_{i}, r_{i}] ,~~~  i = 1, ..., n}$, then \eqref{eq:qp} can be written as
\[
\begin{aligned}
\min_{x, y \in \bbR^n} \quad \sfrac{1}{2} x^T Q x + \iprod{q}{x} + \iota_{\Omega}(y)  \quad 
\textrm{such~that} \quad  x - y = 0 ,
\end{aligned}
\]
which is special case of \eqref{eq:problem-admm} with $A=\Id, B=-\Id$ and $b=0$.

The angle $\theta_k$ of ADMM and the performances of the four schemes are provided in Figure \eqref{fig:cmp-qp}, from which we observed that
\begin{itemize}[leftmargin=2em]
\item The angle $\theta_{k}$ is decreasing to $0$ at the beginning and then starts to increasing for $k \geq 2\times 10^4$. This is mainly due to the fact that for $k \geq 2\times 10^4$, the effects of machine error is becoming increasingly larger. 

\item Consistent with the previous observations, the proposed A$\!^3$DMM schemes provides the best performance. 
\end{itemize}


%
%

\begin{figure}[!ht]
	\centering
	\subfloat[Angle $\theta_k$]{ \includegraphics[width=0.325\linewidth]{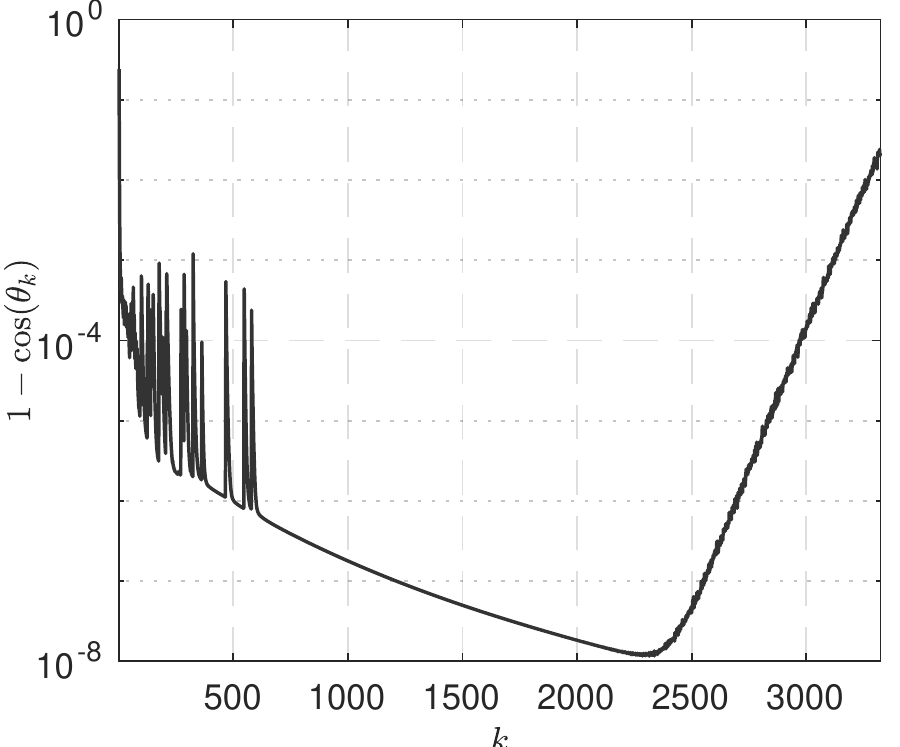} }   \hspace{2ex}
	 \subfloat[Comparison of $\norm{\xk-\xsol}$]{ \includegraphics[width=0.325\linewidth]{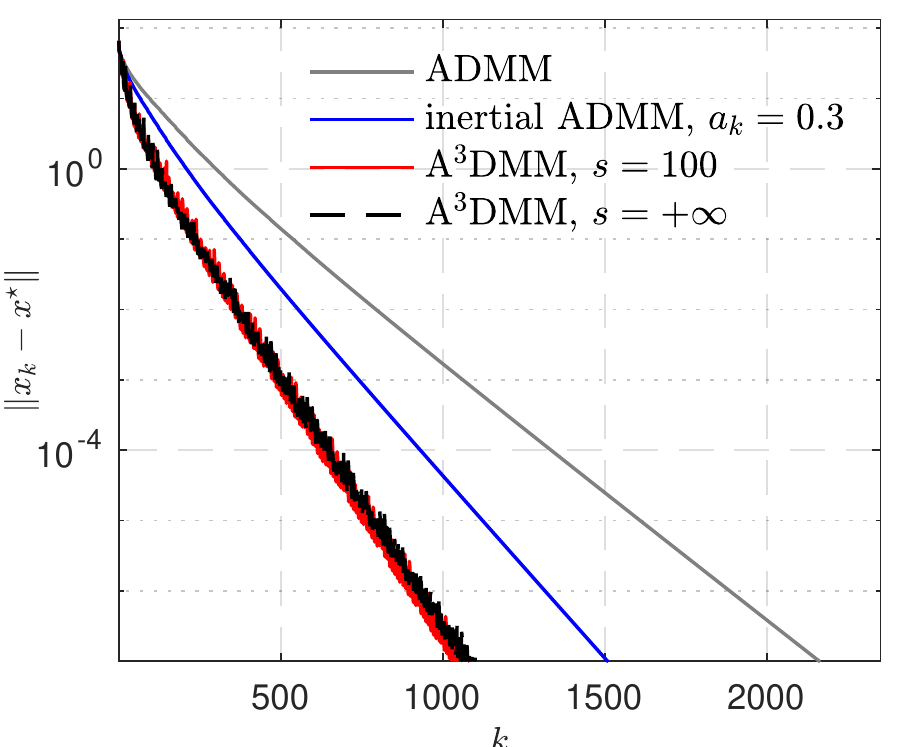} }  \\
	\caption{Performance comparisons and $\seq{\theta_k}$ of ADMM for quadratic programming.}
	\label{fig:cmp-qp}
\end{figure}

\subsection{Total variation based image inpainting}

Here consider a total variation (TV) based image inpainting problem. Let $u \in \bbR^{n \times n}$ be an image and $\calS \in \bbR^{n \times n}$ be a Bernoulli matrix, the observation of $u$ under $\calS$ is $f = \proj_{\calS}(u)$. The TV based image inpainting can be formulated as
\beq\label{eq:inp}
\min_{x\in\bbR^{n\times n}} ~ \norm{\nabla x}_{1} \quad \textrm{such~that} \quad \proj_{\calS}(x) = f  .
\eeq
Define $\Omega \eqdef \ba{x \in \bbR^{n\times n} : \proj_{\calS}(x) = f }$, then \eqref{eq:inp} becomes
\beq\label{eq:inp-xy}
\min_{x\in\bbR^{n\times n}, y\in\bbR^{2n\times n}} ~ \norm{y}_{1} + \iota_{\Omega}(x)  \quad \textrm{such~that} \quad \nabla x - y = 0  ,
\eeq
which is special case of \eqref{eq:problem-admm} with $A=\nabla, B=-\Id$ and $b=0$. For the update of $\xk$, we have from \eqref{eq:admm2} that
\[
\xk = \argmin_{x\in\bbR^{n\times n} }~ \iota_{\Omega}(x) + \sfrac{\gamma}{2} \norm{\nabla x - \tfrac{1}{\gamma} \pa{ \zbarkm - 2\psikm } }^2  ,
\]
which does not admit closed form solution. In the implementation, finite-step FISTA is applied to roughly solve the above problem.


In the experiment, the \texttt{cameraman} image is used, and $50\%$ of the pixels is removed randomly. 
The angle $\theta_k$ of ADMM and the comparisons of the four schemes are provided in Figure \ref{fig:cmp-inp}:
\begin{itemize}[leftmargin=2em]
\item Though both functions in \eqref{eq:inp-xy} are polyhedral, since the subproblem of $\xk$ is solved approximately, the eventual angle actually is oscillating instead of being a constant. 

\item Inertial ADMM again is slower than the original ADMM as the trajectory of ADMM is a spiral. 

\item For the two A$\!^3$DMM schemes, their performances are close as previous examples. 

\item For PSNR the image quality assessment, Figure \ref{fig:cmp-inp}(c) implies that A$\!^3$DMM is also the best. 
\end{itemize}
We also compare the visual quality of the images obtained by the four schemes for the $30$'th iteration, which is shown below in Figure \ref{fig:cmp-keq8}. It can be observed that the image quality (2nd row of Figure \ref{fig:cmp-keq8}) is much better than the 1st row of ADMM and inertial ADMM.


\begin{figure}[!ht]
	\centering
	\subfloat[Angle $\seq{\theta_k}$ of ADMM]{ \includegraphics[width=0.3\linewidth]{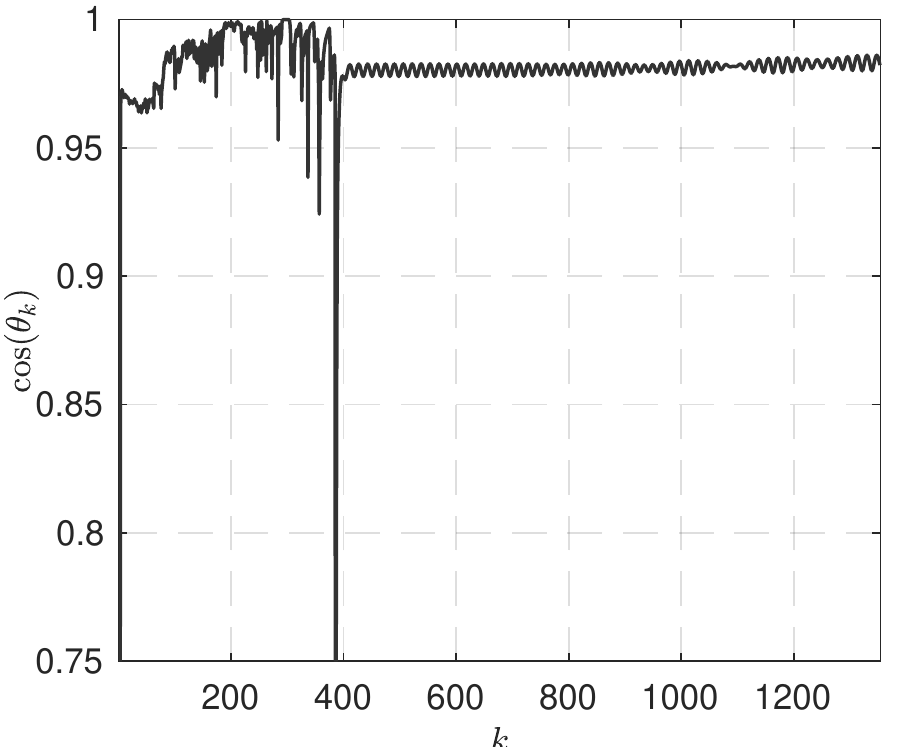} }  
	\subfloat[Comparison of $\norm{\xk-\xsol}$]{ \includegraphics[width=0.3\linewidth]{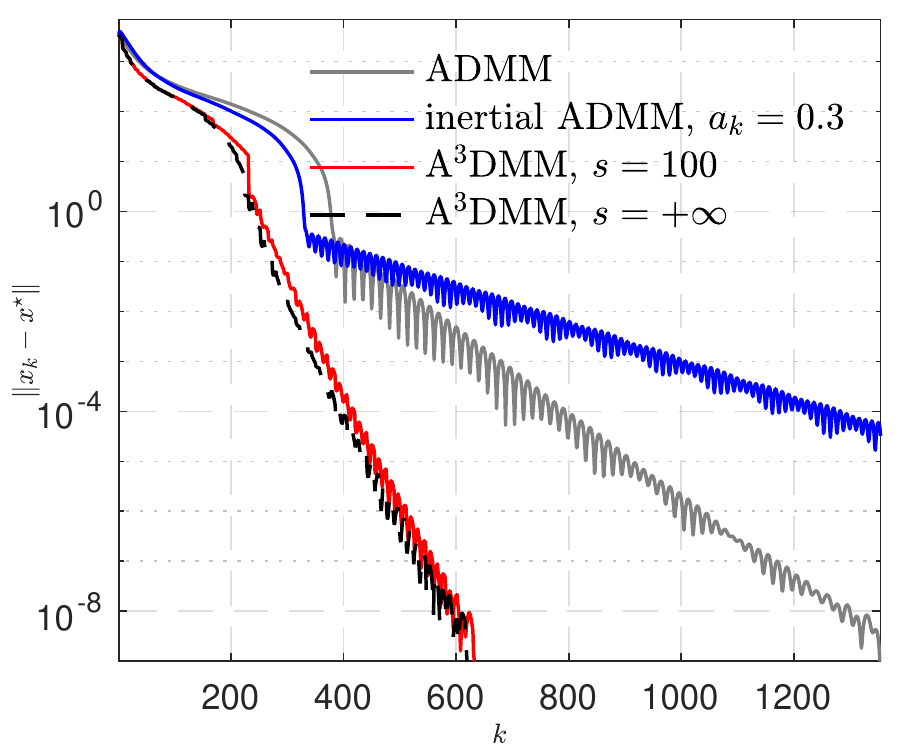} }   
	\subfloat[PSNR value]{ \includegraphics[width=0.3\linewidth]{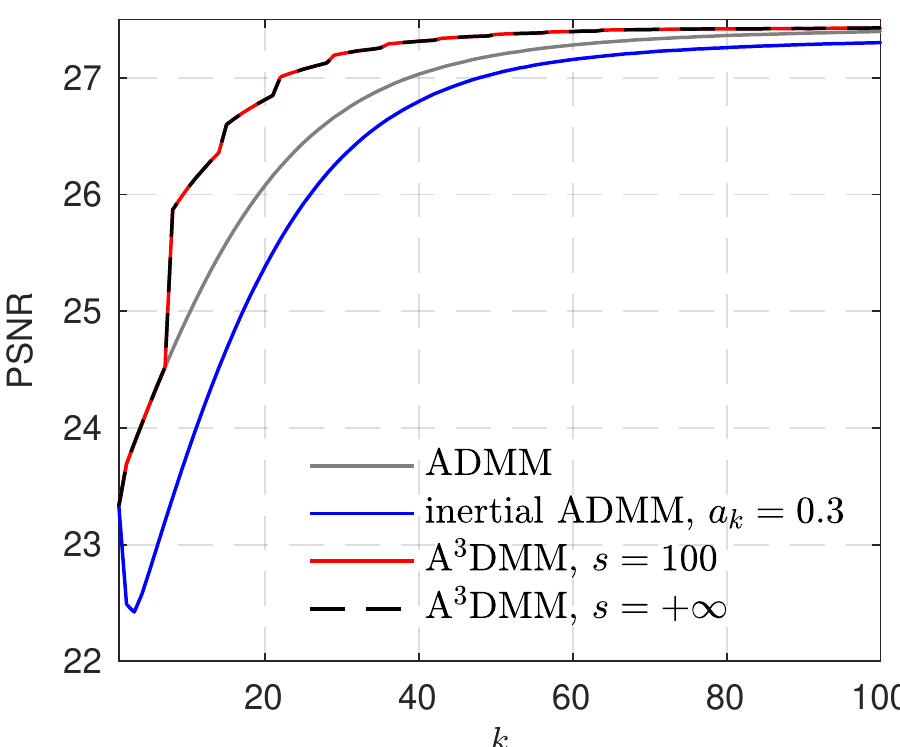} }   \\
	\caption{Property of $\seq{\theta_k}$, performance comparison and image quality of ADMM for TV based image inpainting.}
	\label{fig:cmp-inp}
\end{figure}

\begin{figure}[!ht]
	\centering
	\subfloat[Original image]{ \includegraphics[width=0.31\linewidth]{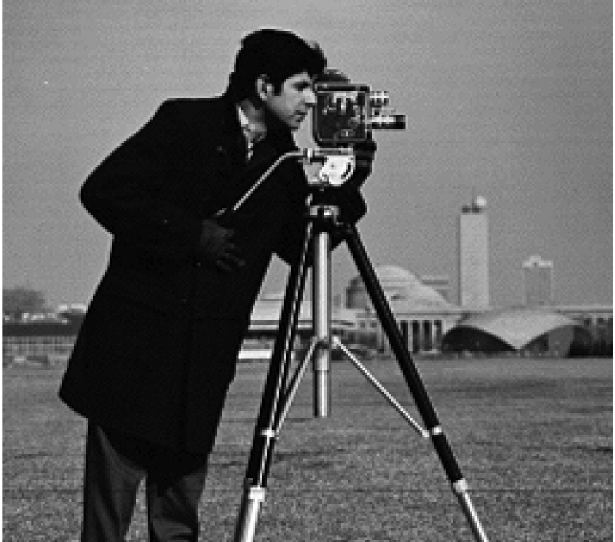} }  \hspace{2pt}
	\subfloat[Observed image]{ \includegraphics[width=0.31\linewidth]{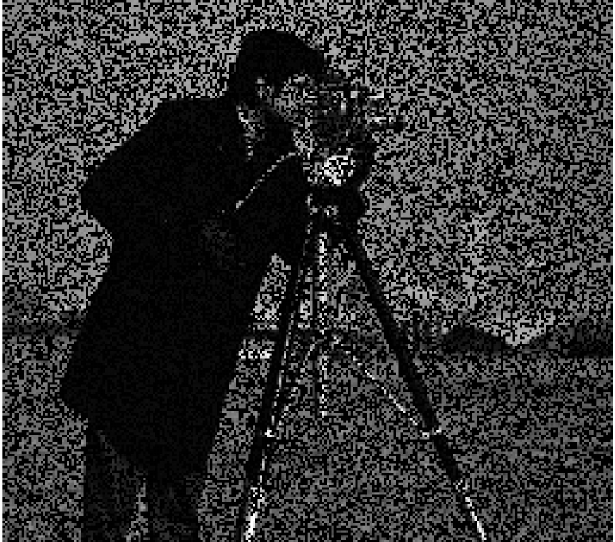} }   
\hspace{2pt}		
	\subfloat[ADMM, PSNR = 26.6935]{ \includegraphics[width=0.31\linewidth]{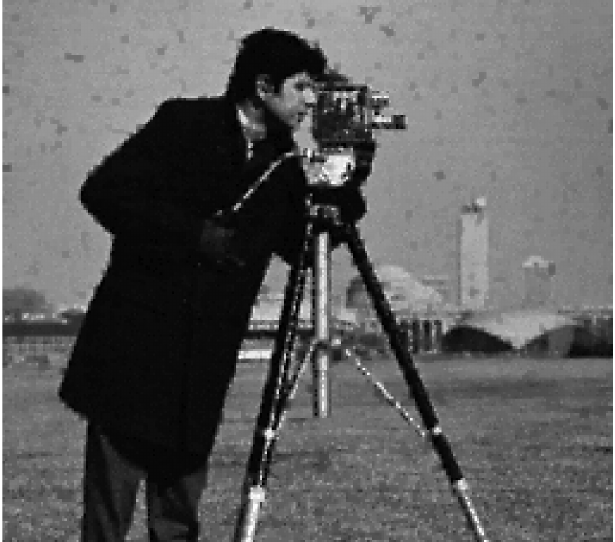} }  \\
	\subfloat[Inertial ADMM, PSNR = 26.3203]{ \includegraphics[width=0.31\linewidth]{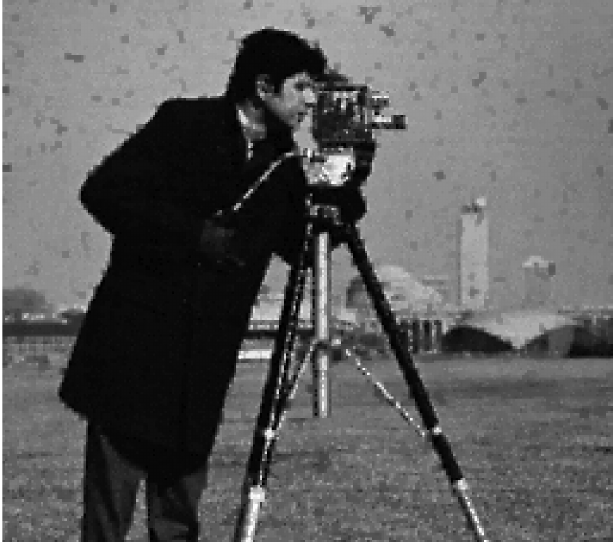} }   \hspace{2pt}
	\subfloat[A$\!^3$DMM $s = 100$, PSNR = 27.1668]{ \includegraphics[width=0.31\linewidth]{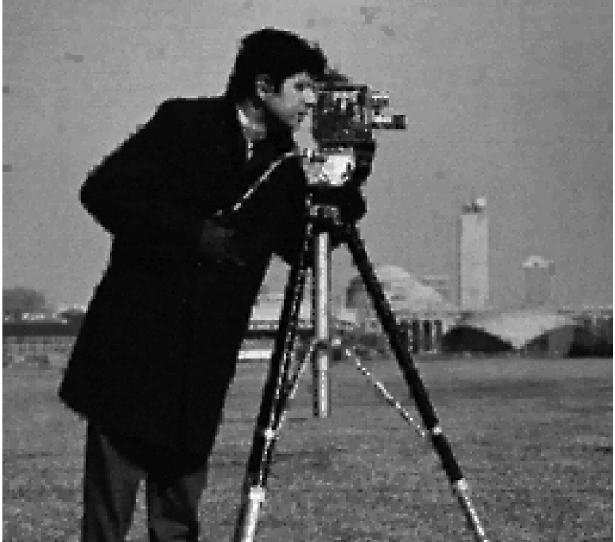} }   \hspace{2pt}
	\subfloat[A$\!^3$DMM $s = \pinf$, PSNR = 27.1667]{ \includegraphics[width=0.31\linewidth]{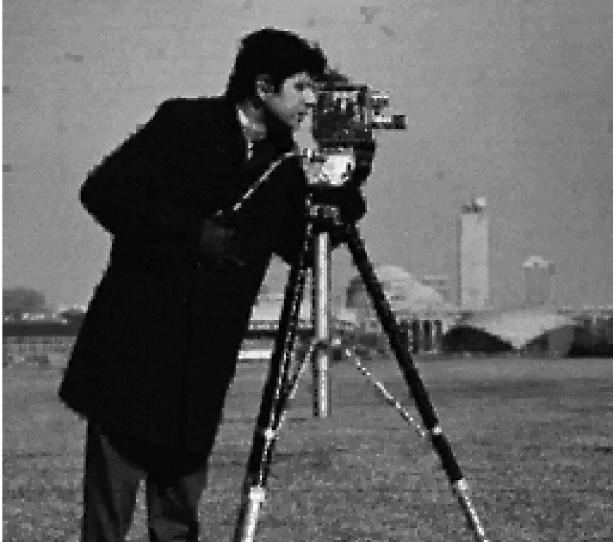} }   \\
	\caption{Comparison of image quality at the $30$'th iteration of ADMM, inertial ADMM and A$\!^3$DMM with two different prediction steps.}
	\label{fig:cmp-keq8}
\end{figure}

%
%
%
%

\section{Conclusions}
In this article, by analyzing the trajectory of the fixed point sequences associated to ADMM and extrapolating along the trajectory, we provide an alternative derivation of these methods.  Furthermore, our local linear analysis allows for the application of previous results on extrapolation methods, and hence provides guaranteed (local) acceleration. Extension of the proposed framework to general first-order methods is ongoing.

\section*{Acknowledgements}
We would like to thank Arieh Iserles for pointing out the connection between trajectory following adaptive acceleration and vector extrapolation. We also like to thank the reviewers whose comments helped to improve the paper. Jingwei Liang was partly supported by Leverhulme trust, Newton trust, the EPSRC centre ``EP/N014588/1'' and the Cantab Capital Institute for the Mathematics of Information (CCIMI).

\small
\bibliographystyle{plain}
\bibliography{bib}

%
%
%

    \setcounter{section}{0}
\renewcommand{\thesection}{\Alph{section}}


\input{sec-supplementary}

\end{document}

%% file: sec-supplementary.tex

\vspace{1ex}
{\noindent}{\bf\Large Appendix}
\vspace{1ex}

{\noindent}The organization of the appendix is as follows: In Section \ref{sec:failure-inertial} we provide more discussions on the conditions when inertial fails. The proofs of the main results of the paper are contained in Sections \ref{sec:pre}-\ref{sec:ada-admm}, where in Section \ref{sec:pre} some preliminary result on angles between subspaces and Riemannian geometry are provided, in Section \ref{P-sec:tra-admm} the proofs for the trajectory of ADMM are provided, and lastly in in Section \ref{sec:proof_a3dmm} we provide proofs for A$\!^3$DMM.

\section{The failure of inertial acceleration continue}\label{sec:failure-inertial}

In this part, to support the discussion of Section \ref{sec:fail-inertial}, we provide extra discussion on why inertial acceleration, in particular Nesterov/FISTA, will fail when the (leading) eigenvalue of $M$ is complex.

Let $M \in \bbR^{n\times n}$ be a square matrix and consider the following linear equation
\beq\label{eq:its-M}
\zkp = M \zk . 
\eeq
According to \cite{polyak1987introduction}, \eqref{eq:its-M} is linearly convergent when the spectral radius of $M$ is strictly smaller than $1$, \ie $\rho(M) < 1$. For simplicity, consider the inertial version of \eqref{eq:its-M} with fixed inertial parameter $\ak \equiv a \in [0, 1]$, we get
\beq\label{eq:its-iM}
\begin{aligned}
	\yk &= \zk + a(\zk - \zkm)  \\
	\zkp &= M \yk .
\end{aligned}
\eeq
The above scheme corresponds to the local linearization of the inertial ADMM \eqref{eq:iadmm} without the small $o$-term. Define the augmented variable $\wk = \begin{pmatrix} \zk \\ \zkm \end{pmatrix}$ and block matrix $\tM \eqdef \begin{bmatrix} (1+a) M & - a M \\ \Id & 0 \end{bmatrix}$, then \eqref{eq:its-iM} can be written as
\beq\label{eq:its-tM}
\wkp = \tM \wk  . 
\eeq
To guarantee the convergence of \eqref{eq:its-tM}, we require the spectral radius satisfying $\rho(\tM) < 1$. Therefore, in the following, motivated by \cite{polyak1987introduction,liang2016thesis,liang2017activity}, we discuss the property of the spectral radius $\rho(\tM)$ and the conditions such that $\rho(\tM)<1$.


Let $\eta, \rho$ be the leading eigenvalues of $M$ and $\tM$, respectively. 
According to \cite[Proposition 4.6]{liang2017activity}, we have the following lemma regarding the relation between $\eta$ and $\rho$. 

\begin{lemma}[{\cite[Proposition 4.6]{liang2017activity}}]
Suppose $\begin{pmatrix} r_1  \\  r_2 \end{pmatrix}$ is the eigenvector of $\tM$ corresponding to eigenvalue $\rho$, then it must satisfy $r_1 = \rho r_2$. Moreover, 
$r_2$ is an eigenvector of $M$ associated to eigenvalue $\eta$, where $\eta$ and $\rho$ satisfy the relation
\beq\label{eq:eigenvalue-relation-2}
\rho^2 - (1+a)\eta \rho + a\eta = 0  .
\eeq
%
%
%
\end{lemma}

The relation \eqref{eq:eigenvalue-relation-2} is a simple quadratic equation of $\rho$, we have
\beq\label{eq:root-rho}
\rho = \sfrac{ (1+a)\eta + \sqrt{(1+a)^2\eta^2 - 4a\eta} }{ 2 }   .
\eeq
%
The value of $\abs{\rho}$ depends on $a$ and $\eta$, and the discussion splits into two scenarios: $\eta$ is real and $\eta$ is complex.

\subsection{Real $\eta$}

When $\eta$ is real valued, the property of $\rho$ is well studied, we refer to \cite{liang2017activity} and references therein for detailed discussions. 
Basically, we have that
\[
\abs{\rho}
=
\left\{
\begin{aligned}
(1+a)^2\eta^2 \geq 4a\eta &: \textrm{$\rho$ is real, $\abs{\rho}< 1$ holds for any $a\in [0, 1]$} , \\
(1+a)^2\eta^2 < 4a\eta &: \textrm{$\rho$ is complex, $\abs{\rho} = \sqrt{a\eta} < 1$ holds for any $a\in [0, 1]$.}
\end{aligned}
\right.
\]
The above result can be summarized below. 

\begin{lemma}[{\cite[Proposition 4.6]{liang2017activity}}]\label{lem:eta-real}
	Given any $a \in [0, 1]$, we have $\abs{\rho} < 1$ as long as $0 \leq \eta < 1$. 
\end{lemma}


To demonstrate the above result, we consider fixing $\eta$ and varying $a \in [0, 1]$. Two choices of $\eta$ are considered $\eta = 0.9, 0.98$, the value of $\abs{\rho}$ is plotted in Figure \ref{fig:rho} in black line. It can be observed that $\abs{\rho}$ is strictly smaller than one for both choices of $\eta$. 
Note that $\abs{\rho}$ reaches a minimal value for some $a$, we refer to \cite{liang2017activity} for detailed discussion on this.

\begin{figure}[!ht]
	\centering
	\subfloat[$\abs{\eta} = 0.9$]{ \includegraphics[width=0.4\linewidth]{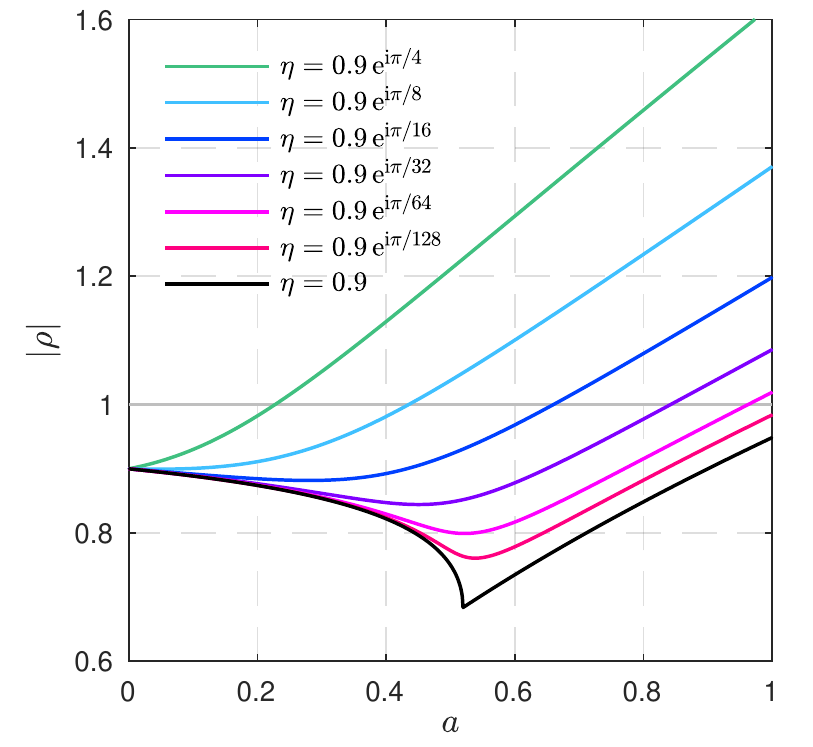} }   {\hspace{8pt}}
	\subfloat[$\abs{\eta} = 0.98$]{ \includegraphics[width=0.4\linewidth]{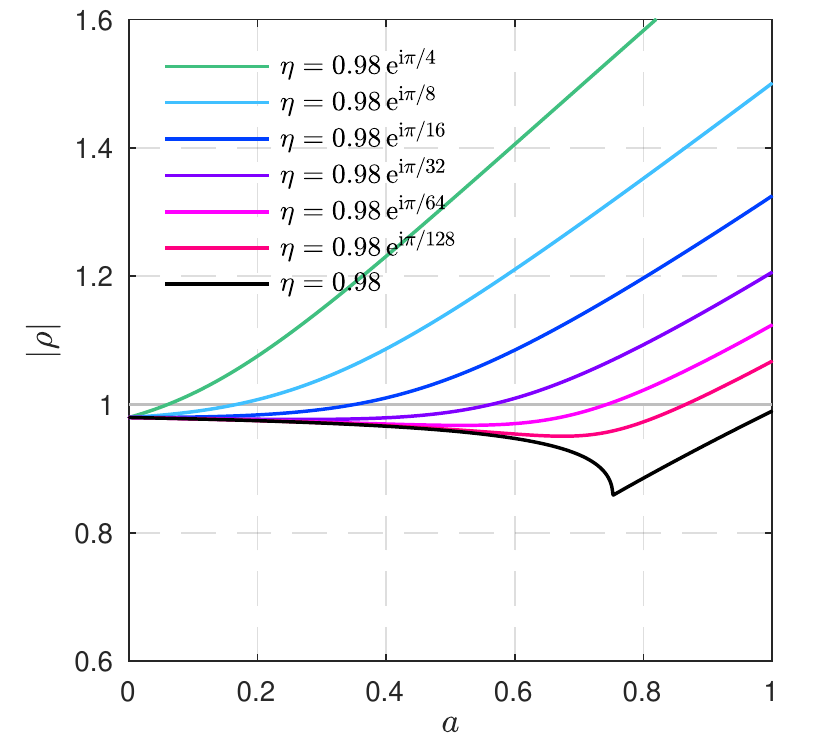} }    \\
	\caption{The value of $\abs{\rho}$ under fixed $\abs{\eta}$ and $a \in [0, 1]$.}
	\label{fig:rho}
\end{figure}

\renewcommand{\e}{\mathrm{e}}
\renewcommand{\i}{\mathrm{i}}

\subsection{Complex $\eta$}

When $\eta$ is complex, it can be written as $\eta = \abs{\eta} \e^{\i \alpha}$ where $\alpha$ is the argument of $\eta$. The dependence of $\abs{\rho}$ on $a$ and $\eta$ becomes much more complicated, below we briefly demonstrate where the difficulties arise and provide numerical proof for the properties of $\abs{\rho}$.

\paragraph{General form $\eta = \abs{\eta}e^{\i\alpha}$}
For this case, we have
\[
\begin{aligned}
\rho 
= \sfrac{ (1+a)\eta + \sqrt{(1+a)^2\eta^2 - 4a\eta} }{ 2 }    
&= \sfrac{ (1+a)\abs{\eta} \e^{\i \alpha} + \ssqrt{ (1+a)^2\abs{\eta}^2 \e^{\i 2\alpha} - 4a\abs{\eta}\e^{\i \alpha} } }{ 2 }    .
\end{aligned}
\]
Suppose $(x+\i y)^2 = (1+a)^2\abs{\eta}^2 \e^{\i 2\alpha} - 4a\abs{\eta} \e^{\i \alpha}$, we get
\[
\begin{aligned}
x^2 - y^2 &= (1+a)^2\abs{\eta}^2 \cos\pa{2\alpha} - 4a\abs{\eta}\cos(\alpha)  \\
x y &= \sfrac{ (1+a)^2\abs{\eta}^2 \sin\pa{2\alpha} - 4a\abs{\eta}\sin(\alpha) }{ 2 }  ,
\end{aligned}
\]
which can be simplified to a equation of $x$
\[
x^4 - \Pa{(1+a)^2\abs{\eta}^2 \cos\pa{2\alpha} - 4a\abs{\eta}\cos(\alpha)} x^2 - \sfrac{ ( (1+a)^2\abs{\eta}^2 \sin\pa{2\alpha} - 4a\abs{\eta}\sin(\alpha) )^2 }{ 4 } = 0  .
\]
Solving the above equation, we get
\[
\begin{aligned}
x &= \BPa{ \tfrac{ \pa{(1+a)^2\abs{\eta}^2 \cos\pa{2\alpha} - 4a\abs{\eta}\cos(\alpha)} + \ssqrt{ \pa{(1+a)^2\abs{\eta}^2 \cos\pa{2\alpha} - 4a\abs{\eta}\cos(\alpha)}^2 + \pa{ (1+a)^2\abs{\eta}^2 \sin\pa{2\alpha} - 4a\abs{\eta}\sin(\alpha) }^2 }}{2} }^{1/2} , \\
y &= \sfrac{ (1+a)^2\abs{\eta}^2 \sin\pa{2\alpha} - 4a\abs{\eta}\sin(\alpha) }{ 2 x } ,
\end{aligned}
\]
here we only take the positive root $x$. 
Back to the expression of $\rho$, we get
\[
\begin{aligned}
\rho 
= \sfrac{ (1+a)\abs{\eta} \e^{\i \alpha} + (x + \i y ) }{ 2 }    
= \sfrac{ \pa{ (1+a)\abs{\eta}\cos(\alpha) + x } + \i \pa{ (1+a)\abs{\eta}\sin(\alpha) + y } }{2}  .
\end{aligned}
\]
Given the complicated form of $x$, the analysis of $\abs{\rho}$ becomes rather difficult. Therefore, below we discuss the property of $\abs{\rho}$ through numerical verification.

Similar to the real $\eta$ case, $\abs{\eta} = 0.9, 0.98$ are considered. Denote $\alpha$ the argument of $\eta$, then we have $\eta = \abs{\eta} \mathrm{e}^{\mathrm{i} \alpha}$. In total, six choices of $\alpha$ are considered: $\alpha \in \ba{\frac{\pi}{4}, \frac{\pi}{8}, \frac{\pi}{16}, \frac{\pi}{32}, \frac{\pi}{64}, \frac{\pi}{128}}$. The value of $\abs{\rho}$ are shown in Figure \ref{fig:rho}. Taking Figure \ref{fig:rho} (a) for example, we have the following observations: 
\begin{itemize}[leftmargin=2em]
	\item For all choices of $\alpha$ except $\alpha = \frac{\pi}{128}$, there exists an $a_{\alpha} < 1$ such that $\abs{\rho} \geq 1$ for $a \in [a_{\alpha}, 1]$.
	
	\item The larger the value of $\alpha$, the smaller the value of $a_{\alpha}$, see the green line in both figures. 
	
\end{itemize}
From the above discussion, we can conclude that 
\begin{itemize}[leftmargin=2em]
	\item The inertial scheme is robust when all the eigenvalues of $M$ are real, and we can afford the inertial parameter up to $1$ which includes the FISTA \cite{fista2009} schemes as $\ak \to 1$, same for the Nesterov's accelerated gradient descent. 
	
	\item When $M$ has complex eigenvalue(s), which is not necessary to the leading eigenvalue, the largest value of $a$ can be allowed is smaller than $1$ and FISTA/Nesterov's scheme will fail. 
	
\end{itemize}
To complete the discussion, we consider the values of $\abs{\rho}$ under $\alpha \in [0, \pi/2]$ and $a \in [0, 1]$. The results are shown below in Figure \ref{fig:rho_alpha_a}. Again $\abs{\eta} = 0.9, 0.98$ are considered. The horizontal axis is for $\alpha$ while the vertical is for $a$, each point inside the square stands for the value of $\abs{\rho}$ with colorbar provided. In each figure: 
\begin{itemize}[leftmargin=2em]
	\item The {\em red} line stands for $\abs{\rho}=1$. Therefore, only for the area below the red line we have $\abs{\rho} < 1$. Given any $\alpha \in [0, \pi/2]$, the larger the value of $\alpha$, the smaller range of choice of $a$ such that $\abs{\rho} < 1$. This coincides with the observations from Figure \ref{fig:rho}. 
	
	\item The {\em magenta} line stands for $\abs{\rho}=\abs{\eta}$. Only the small area below the magenta line has $\abs{\rho} < \abs{\eta}$, meaning that acceleration can be obtained. As a result, given $\eta = \abs{\eta} \e^{\i\alpha}$, when $\alpha$ is large enough, such as about $\pi/8$ for $\abs{\eta}=0.9$, inertial will fail to provide acceleration. 
	
\end{itemize}

\begin{figure}[!ht]
	\centering
	\subfloat[$\abs{\eta} = 0.9$]{ \includegraphics[width=0.475\linewidth]{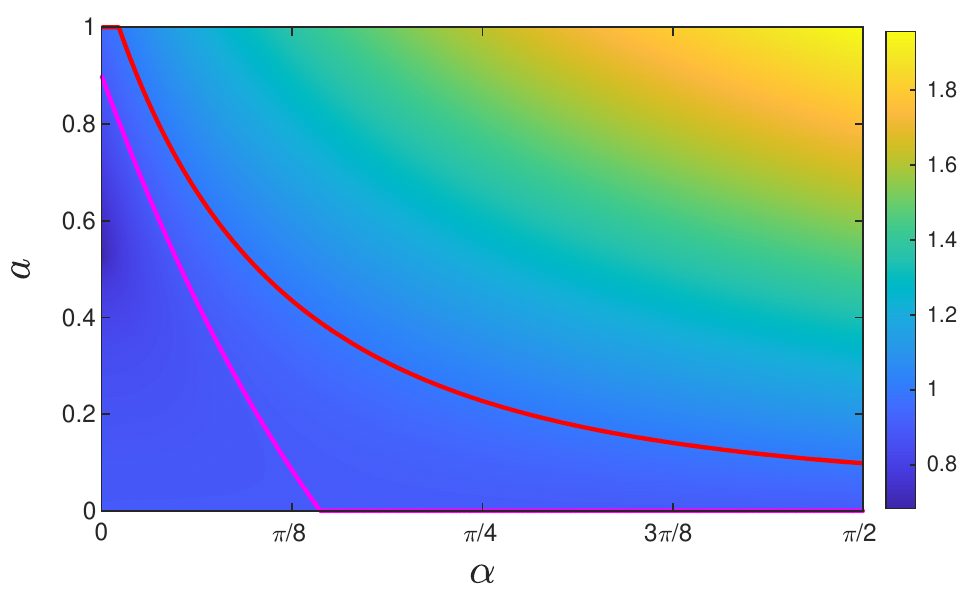} }   {\hspace{2pt}}
	\subfloat[$\abs{\eta} = 0.98$]{ \includegraphics[width=0.475\linewidth]{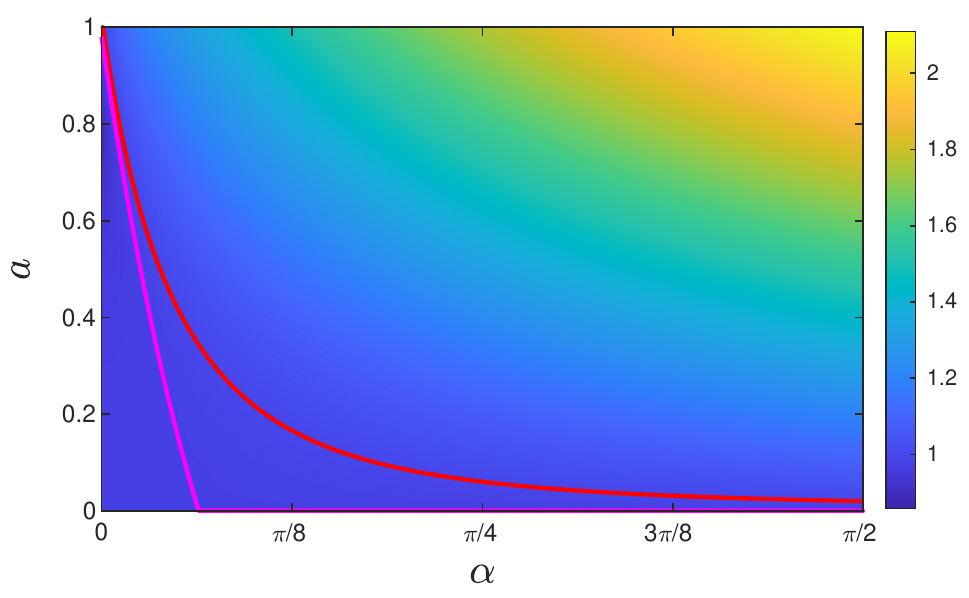} }    \\
	\caption{The value of $\abs{\rho}$ under fixed $\eta$ and $a \in [0, 1]$.}
	\label{fig:rho_alpha_a}
\end{figure}

It should be noted that, for the above discussion, we consider the case that the leading eigenvalue is {\em complex}, while the rest of the eigenvalues are {\em real}. 
For the case leading eigenvalue is {\em real} while the rest are {\em complex}, then the spectral radius of $\tM$ will be determined by the non-leading complex eigenvalues when the inertial parameter $a$ is large enough. Consequently, the FISTA inertial parameter rule still can not be applied, unless the magnitude of the leading eigenvalue is small enough; See Figure \ref{fig:rho_alpha_a} (a).

\begin{figure}[!ht]
	\centering
	\subfloat[$\alpha \in \ba{\frac{\pi}{4}, \frac{\pi}{8}, \frac{\pi}{16}, \frac{\pi}{32}, \frac{\pi}{64}, \frac{\pi}{128}}$]{ \includegraphics[width=0.35\linewidth]{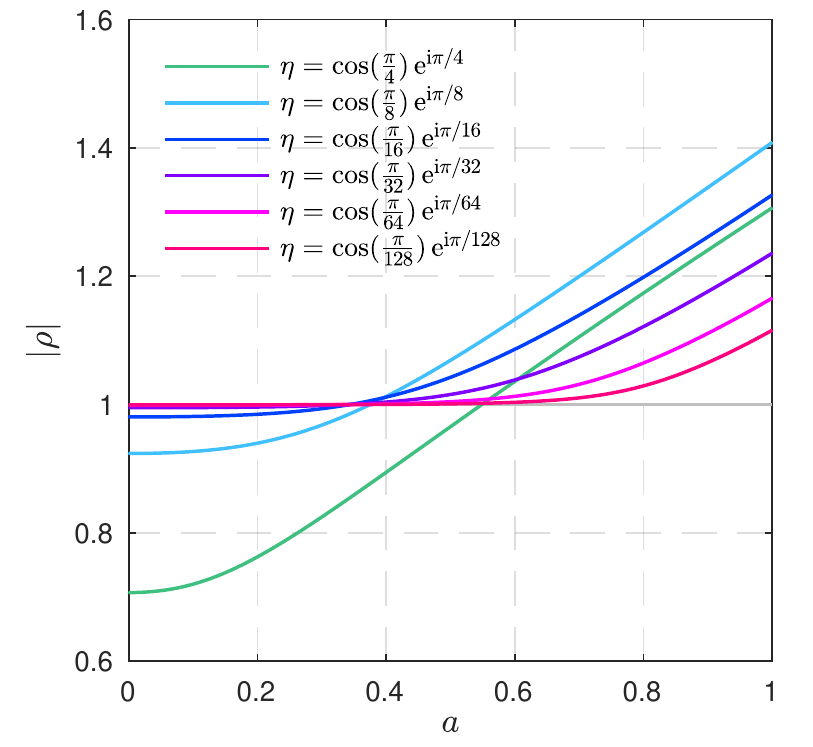} }   {\hspace{16pt}}
	\subfloat[{$\alpha \in [0, \frac{\pi}{2}]$}]{ \includegraphics[width=0.525\linewidth]{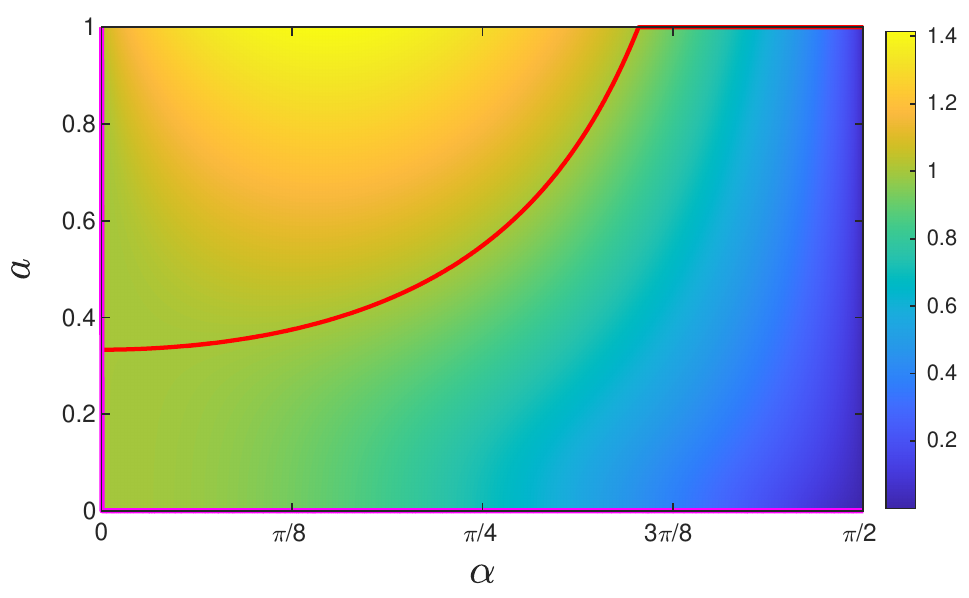} }    \\
	\caption{The value of $\abs{\rho}$ when $\eta = \cos(\alpha)\e^{\i \alpha}$ and $a \in [0, 1]$.}
	\label{fig:rho_special}
\end{figure}

\paragraph{Special case $\eta = \cos{\alpha}e^{\i\alpha}$}
Now we consider a special case where $\eta = \cos(\alpha) e^{\i\alpha}, \alpha \in [0, \pi/2]$ which corresponds to the case $R, J$ in \eqref{eq:problem-admm} are locally polyhedral around $\xsol, \ysol$. 
Similar to above, six choices of $\alpha$ are considered: $\alpha \in \ba{\frac{\pi}{4}, \frac{\pi}{8}, \frac{\pi}{16}, \frac{\pi}{32}, \frac{\pi}{64}, \frac{\pi}{128}}$.
The value of $\abs{\rho}$ is shown below in Figure \ref{fig:rho_special} (a). It can be observed that, for each $\alpha$, the value of $\abs{\rho}$ is monotonically increasing as the value of $a$ increases, which means {\em inertial slows down the speed of convergence}. 
In Figure \ref{fig:rho_special} (b), we consider the value of $\abs{\rho}$ under $\alpha \in [0, \pi/2]$ and $a \in [0, 1]$. We have
\begin{itemize}[leftmargin=2em]
	\item Similar to Figure \ref{fig:rho_alpha_a}, the {\em red} line stands for $\abs{\rho}=1$. For each $\alpha$, $\abs{\rho} < 1$ for all the choices of $a$ under the red line. 
	
	\item The {\em magenta} line stands for $\abs{\rho}=\abs{\eta}$. It can be observed that, except for $\alpha = 0$ where $\abs{\rho} = 1$ holds for all $a\in [0, 1]$, $\abs{\rho} = 1$ holds only for $a=0$ when $\alpha \in ]0, \pi/2]$. 
	
\end{itemize}
Therefore, we can conclude that when $R, J$ are locally polyhedral around the solution $\xsol, \ysol$, inertial scheme will not provide any acceleration.

\section{Preparatory materials}\label{sec:pre}

\subsection{Polynomial extrapolation}

Minimal polynomial extrapolation (MPE) \cite{cabay1976polynomial}:
Given $\{z_{k-j}\}_{j=0}^{q+1}$, let $\{v_{k-j}\}_{j=0}^q$ be the difference vectors, where $v_j \eqdef z_j - z_{j-1}$. Define $V_{k} = \begin{bmatrix}
v_k&\cdots & v_{k-q}
\end{bmatrix}$. 
\begin{enumerate}[leftmargin=2em]
\item Let $\{c_j\}_{j=1}^q\in \argmin_{c\in \RR^q} \norm{V_{k-1} c - v_k}$, define $c_0\eqdef 1$ and $\gamma_i = c_i/\sum_{i=0}^q c_i$ for $i=0,\ldots, q$.
\item The extrapolated point is then defined to be ${\zbar_{k}} \eqdef \sum_{i=0}^q \gamma_i z_{k-i-1}$.
\end{enumerate}
Reduced rank extrapolation (RRE) \cite{eddy1979extrapolating,mevsina1977convergence} is obtained by replacing the first step by $$\{\gamma_j\}_{j=0}^q\in \argmin_{\gamma\in\RR^{q+1}} \norm{V_k \gamma} \text{ subject to } \msum_i \gamma_i =1.$$ The motivation for the use of such methods for the acceleration of fixed point sequences $x_{k+1} = \Ff(\zk)$ come from considering the spectral properties of the linearization around the limit point. In particular, if $\zsol$ is the limit point and $\zkp -\zsol = T(\zk - \zsol)$ where $T\in \RR^{d\times d}$ and $q$ is the order of the minimal polynomial of $T$ with respect to $z_{k-q-1} - \zsol$ (i.e. $q$ is the monic polynomial of least degree such that $P(T)(z_{k-q-1} -\zsol) = 0$), then one can show that ${\zbar_{k}} = \zsol$. We refer to \cite{sidi1986acceleration,sidi1998upper,sidi2017vector} for details on these methods and their acceleration guarantees. 

\subsection{Angle between subspaces}
\label{P-sec:angles}

Let $T_1, T_2$ be two subspaces, and without the loss of generality, assume $$1\leq p\eqdef\dim(T_1) \leq q\eqdef\dim(T_2) \leq n-1.$$

\begin{definition}[Principal angles] \label{def:principal_angle}
The principal angles $\theta_k \in [0,\frac{\pi}{2}]$, $k=1,\ldots,p$ between subspaces $T_1$ and $T_2$ are defined by, with $u_0 = v_0 \eqdef 0$, and
\begin{align*}
\cos(\theta_k) \eqdef \iprod{u_k}{v_k} = \max \iprod{u}{v} ~~~\mathrm{s.t.}~~~
& u \in T_1, v \in T_2, \norm{u}=1, \norm{v}=1, \\
&\iprod{u}{u_i}=\iprod{v}{v_i}=0, ~ i=0,\dotsm,k-1 .
\end{align*}
The principal angles $\theta_k$ are unique and satisfy $0 \leq \theta_1 \leq \theta_2 \leq \dotsm \leq \theta_p \leq \pi/2$.
\end{definition}

\begin{definition}[Friedrichs angle]\label{def:friedrichs-angle}
The Friedrichs angle $\theta_{F} \in ]0,\frac{\pi}{2}]$ between $T_1$ and $T_2$ is
\begin{equation*}
\cos\Pa{ \theta_F(T_1,T_2) } \eqdef \max \iprod{u}{v} ~~~\mathrm{s.t.}~~~
u \in T_1 \cap (T_1 \cap T_2)^\perp, \norm{u}=1 ,~
v \in T_2 \cap (T_1 \cap T_2)^\perp, \norm{v}=1 .
\end{equation*}
\end{definition}

The following lemma shows the relation between the Friedrichs and principal angles, whose proof can be found in \cite[Proposition~3.3]{Bauschke14}.

\begin{lemma}[Principal angles and Friedrichs angle]
\label{lem:fapa}
The Friedrichs angle is exactly $\theta_{d+1}$ where $d \eqdef \dim(T_1 \cap T_2)$. Moreover, $\theta_{F}(T_1,T_2) > 0$.
\end{lemma}

\subsection{Riemannian Geometry}
\label{subsec:riemgeom}

Let $\calM$ be a $C^2$-smooth embedded submanifold of $\bbR^n$ around a point $x$. With some abuse of terminology, we shall state $C^2$-manifold instead of $C^2$-smooth embedded submanifold of $\bbR^n$. The natural embedding of a submanifold $\calM$ into $\bbR^n$ permits to define a Riemannian structure and to introduce geodesics on $\calM$, and we simply say $\calM$ is a Riemannian manifold. We denote respectively $\tanSp{\Mm}{x}$ and $\normSp{\Mm}{x}$ the tangent and normal space of $\Mm$ at point near $x$ in $\calM$.
%
%

\paragraph{Exponential map}
Geodesics generalize the concept of straight lines in $\bbR^n$, preserving the zero acceleration characteristic, to manifolds. Roughly speaking, a geodesic is locally the shortest path between two points on $\calM$. We denote by $\mathfrak{g}(t;x, h)$ the value at $t \in \bbR$ of the geodesic starting at $\mathfrak{g}(0;x,h) = x \in \calM$ with velocity $\dot{\mathfrak{g}}(t;x, h) = \sfrac{d\mathfrak{g}}{dt}(t;x,h) = h \in \tanSp{\Mm}{x}$ (which is uniquely defined). 
%
For every $h \in \tanSp{\calM}{x}$, there exists an interval $I$ around $0$ and a unique geodesic $\mathfrak{g}(t;x, h): I \to \calM$ such that $\mathfrak{g}(0; x, h) = x$ and $\dot{\mathfrak{g}}(0;x, h) = h$.
The mapping
 \[
\Exp_x 
: \tanSp{\calM}{x} \to \calM ,~~ h\mapsto \Exp_{x}(h) = \mathfrak{g}(1;x, h) ,
\]
is called \emph{Exponential map}.
Given $x, x' \in \calM$, the direction $h \in \tanSp{\calM}{x}$ we are interested in is such that 
\[
\Exp_x(h) = x' = \mathfrak{g}(1;x, h) .
\]

\paragraph{Parallel translation}

Given two points $x, x' \in \calM$, let $\tanSp{\calM}{x}, \tanSp{\calM}{x'}$ be their corresponding tangent spaces. Define 
\[
\tau : \tanSp{\calM}{x} \to \tanSp{\calM}{x'} ,
\]
the parallel translation along the unique geodesic joining $x$ to $x'$, which is isomorphism and isometry w.r.t. the Riemannian metric.

\paragraph{Riemannian gradient and Hessian}

For a vector $v \in \normSp{\calM}{x}$, the Weingarten map of $\calM$ at $x$ is the operator $\Wgtmap{x}\pa{\cdot, v}: \tanSp{\calM}{x} \to \tanSp{\calM}{x}$ defined by
\[
\Wgtmap{x}\pa{\cdot, v} = - \PT{\tanSp{\calM}{x}} \mathrm{d} V[h] ,
\]
where $V$ is any local extension of $v$ to a normal vector field on $\calM$. The definition is independent of the choice of the extension $V$, and $\Wgtmap{x}(\cdot, v)$ is a symmetric linear operator which is closely tied to the second fundamental form of $\calM$, see \cite[Proposition II.2.1]{chavel2006riemannian}.

Let $G$ be a real-valued function which is $C^2$ along the $\calM$ around $x$. The covariant gradient of $G$ at $x' \in \calM$ is the vector $\nabla_{\calM} G(x') \in \tanSp{\calM}{x'}$ defined by
\[
\iprod{\nabla_{\calM} G(x')}{h} = \sfrac{d}{dt} G\Pa{\PT{\calM}(x'+th)}\big|_{t=0} ,~~ \forall h \in \tanSp{\calM}{x'},
\]
where $\PT{\calM}$ is the projection operator onto $\calM$. 
The covariant Hessian of $G$ at $x'$ is the symmetric linear mapping $\nabla^2_{\calM} G(x')$ from $\tanSp{\calM}{x'}$ to itself which is defined as
\beq\label{eq:rh}
\iprod{\nabla^2_{\calM} G(x') h}{h} = \sfrac{d^2}{dt^2} G\Pa{\PT{\calM}(x'+th)}\big|_{t=0} ,~~ \forall h \in \tanSp{\calM}{x'} .
\eeq
This definition agrees with the usual definition using geodesics or connections \cite{miller2005newton}. 
Now assume that $\calM$ is a Riemannian embedded submanifold of $\bbR^{n}$, and that a function $G$ has a $C^2$-smooth restriction on $\calM$. This can be characterized by the existence of a $C^2$-smooth extension (representative) of $G$, \ie a $C^2$-smooth function $\widetilde{G}$ on $\bbR^{n}$ such that $\widetilde{G}$ agrees with $G$ on $\calM$. Thus, the Riemannian gradient $\nabla_{\calM}G(x')$ is also given by
\beq\label{eq:RieGradient}
\nabla_{\calM} G(x') = \PT{\tanSp{\calM}{x'}} \nabla \widetilde{G}(x') ,
\eeq
and $\forall h \in \tanSp{\calM}{x'}$, the Riemannian Hessian reads
\beq\label{eq:RieHessian}
\begin{aligned}
\nabla^2_{\calM} G(x') h
&= \PT{\tanSp{\calM}{x'}} \mathrm{d} (\nabla_{\calM} G)(x')[h]
= \PT{\tanSp{\calM}{x'}} \mathrm{d} \Pa{ x' \mapsto \PT{\tanSp{\calM}{x'}} \nabla_{\calM} \widetilde{G} }[h] \\ 
&= \PT{\tanSp{\calM}{x'}} \nabla^2 \widetilde{G}(x') h + \Wgtmap{x'}\Pa{h, \PT{\normSp{\calM}{x'}}\nabla \widetilde{G}(x')} ,
\end{aligned}
\eeq
where the last equality comes from \cite[Theorem~1]{absil2013extrinsic}. 
When $\calM$ is an affine or linear subspace of $\bbR^{n}$, then obviously $\calM = x + \tanSp{\calM}{x}$, and $\Wgtmap{x'}\pa{h, \PT{\normSp{\calM}{x'}} \nabla \widetilde{G}(x')} = 0$, hence \eqref{eq:RieHessian} reduces to
\[
\nabla^2_{\calM} G(x')
= \PT{\tanSp{\calM}{x'}} \nabla^2 \widetilde{G}(x') \PT{\tanSp{\calM}{x'}} .
\]
See \cite{lee2003smooth,chavel2006riemannian} for more materials on differential and Riemannian manifolds.

\subsection{Preparatory lemmas}

The following lemmas characterize the parallel translation and the Riemannian Hessian of nearby points in $\calM$.

\begin{lemma}[{\cite[Lemma 5.1]{liang2014local}}]\label{lem:proj-M}
Let $\calM$ be a $C^2$-smooth manifold around $x$. Then for any $x' \in \calM \cap \calN$, where $\calN$ is a neighborhood of $x$, the projection operator $\proj_{\calM}(x')$ is uniquely valued and $C^1$ around $x$, and thus 
\[
x' - x = \proj_{\tanSp{\calM}{x}}(x'-x) + o\pa{\norm{x'-x}} .
\]
If moreover $\calM = x + \tanSp{\calM}{x}$ is an affine subspace, then $x' - x = \proj_{\tanSp{\calM}{x}}(x'-x)$.
\end{lemma}
%

\begin{lemma}[{\cite[Lemma B.1]{liang2017activity}}]\label{lem:parallel-translation}
Let $x \in \calM$, and $\xk$ a sequence converging to $x$ in $\calM$. Denote $\tau_k : \tanSp{\calM}{\xk} \to \tanSp{\calM}{x}$ be the parallel translation along the unique geodesic joining $x$ to $\xk$. Then, for any bounded vector $u \in \bbR^n$, we have
\[
\pa{\tau_k \proj_{\tanSp{\calM}{\xk}} - \proj_{\tanSp{\calM}{x}}}u = o(\norm{u}) .
\]
\end{lemma}
%

The Riemannian gradient and Hessian of partly smooth functions are covered by the lemma below. 

\begin{lemma}[{\cite[Lemma B.2]{liang2017activity}}]\label{lem:taylor-expn}
Let $x, x'$ be two close points in $\calM$, denote $\tau : \tanSp{\calM}{x'} \to \tanSp{\calM}{x}$ the parallel translation along the unique geodesic joining $x$ to $x'$. The Riemannian Taylor expansion of $R \in C^2(\calM)$ around $x$ reads,
\beq\label{eq:taylor-expn}
\tau \nabla_{\calM} R(x') = \nabla_{\calM} R(x) + \nabla^2_{\calM} R(x)\proj_{\tanSp{\calM}{x}}(x'-x) + o(\norm{x'-x}) .
\eeq
\end{lemma}

\begin{lemma}[Riemannian gradient and Hessian]\label{def:riemannian-gradhess}
If $R \in \PSF{x}{\calM_x}$, then for any point $x' \in \calM_x$ near $x$
\[
\nabla_{\calM_x} R(x') = \proj_{T_{x'}}\pa{\partial R(x')} ,
\]
and this does not depend on the smooth representation of $R$ on $\calM_x$. In turn, for all $h \in T_{x'}$, let $\widetilde{R}$ be a smooth representative of $R$ on $\calM_x$,
\[
\nabla^2_{\calM_x} R(x')h = \proj_{T_{x'}} \nabla^2 \widetilde{R}(x')h + \Wgtmap{x'}\Pa{h, \proj_{T_{x'}^\perp}\nabla \widetilde{R}(x')} ,
\]
where $\Wgtmap{x}\pa{\cdot, \cdot}: T_x \times T_x^\perp \to T_x$ is the Weingarten map of $\calM_x$ at $x$.
\end{lemma}

\subsection{Linearization of proximal mapping}

In this part, we present one fundamental result led by partial smoothness, the linearization of proximal mapping. 
We first discuss the property of the Riemannian Hessian of a partly smooth function. Let $R \in \lsc(\bbR^n)$ be partly smooth at $\xbar$ relative to $\calM_{\xbar}$ and $\ubar \in \partial R(\xbar)$, define the following smooth perturbation of $R$
\[
\barR(x) \eqdef R(x) - \iprod{x}{\ubar} ,
\]
whose Riemannian Hessian at $\xbar$ reads $H_{\barR} \eqdef \proj_{T_{\xbar}} \nabla^2_{\calM_{\xbar}} \barR \pa{\xbar} \proj_{T_{\xbar}} $.

\begin{lemma}[{\cite[Lemma 4.2]{liang2017activity}}]\label{lem:riemhesspsd} 
Let $R \in \lsc(\bbR^n)$ be partly smooth at $\xbar$ relative to $\calM_{\xbar}$, then $H_{\barR}$ is symmetric positive semi-definite if either of the following is true:
\begin{itemize}[leftmargin=2em]
\item $\ubar \in \ri(\partial R(\xbar))$ is non-degenerate. 
\item $\calM_{\xbar}$ is an affine subspace.
\end{itemize}
In turn, $\Id + H_{\barR}$ is invertible and $\pa{\Id + H_{\barR}}^{-1}$ is symmetric positive definite with all eigenvalues in~$]0,1]$.
\end{lemma}

One consequence of Lemma \ref{lem:riemhesspsd} is that, we can linearize the generalized proximal mapping. 
For the sake of generality, let $\gamma > 0$, $R \in \lsc(\bbR^n)$ and $A \in \bbR^{p \times n}$, define the following generalized proximal mapping
\[
\prox_{\gamma R}^{A}(\cdot) \eqdef \argmin_{x\in \bbR^n} \gamma R(x) + \sfrac{1}{2} \norm{Ax - \cdot}^2 .
\]
Clearly, $\prox_{\gamma R}^{A}$ is a single-valued mapping when $A$ has full column rank. 
Denote $A_{T_{\xbar}} \eqdef A \circ \PT{T_{\xbar}}$, it is immediate that $A_{T_{\xbar}}^TA_{T_{\xbar}}$ is positive semidefinite and invertible along $T_{\xbar}$. In the following we denote $\pa{A_{T_{\xbar}}^TA_{T_{\xbar}}}^{-1}$ the inverse along $T_{\xbar}$ 
Denote $$ M_{\barR} = A_{T_{\xbar}} \pa{ \Id + (A_{T_{\xbar}}^T A_{T_{\xbar}})^{-1} H_{\barR} }^{-1} (A_{T_{\xbar}}^T A_{T_{\xbar}})^{-1} A_{T_{\xbar}}^T . $$ 

\begin{lemma}\label{lem:lin-generalised-ppa}
Let function $R \in \lsc(\bbR^n)$ be partly smooth at the point $\xbar$ relative to the manifold $\calM_{\xbar}$ and $\ubar \in \ri \pa{\partial R(\xbar)}$. Suppose that there exists $\gamma > 0$, full column rank $A \in \bbR^{p\times n}$ and $\wbar \in \bbR^p$ such that $\xbar = \prox_{\gamma R}^{A}(\wbar)$ and $\ubar = - A^T (A\xbar - \wbar)/\gamma$. 
Let $\seq{\wk}$ be a sequence such that $\wk \to \wbar$ and $\xk = \prox_{\gamma R}^{A}(\wk) \to \xbar$, then for all $k$ large enough, there hold $\xk \in \calM_{\xbar}$ and 
\beq\label{eq:lin-generalised-ppa}
A_{T_{\xbar}} (\xk-\xkm) = M_{\barR} \pa{\wk - \wkm} + o(\norm{\wk-\wkm}) . 
\eeq
\end{lemma}
\begin{remark}
When $A = \Id$, then $\prox_{\gamma R}^{A}$ reduces to the standard proximal mapping, and \eqref{eq:lin-generalised-ppa} simplifies to
\[
\xk-\xkm = \PT{T_{\xbar}} \Pa{\Id + H_{\barR}}^{-1} \PT{T_{\xbar}} \pa{\wk - \wkm} + o(\norm{\wk-\wkm}) . 
\]
In \cite{liang2016thesis} and references therein, to study the local linear convergence of first-order methods, linearization with respect to the limiting points is provided, that is
\[
\xk-\xbar = \PT{T_{\xbar}} \Pa{\Id + H_{\barR}}^{-1} \PT{T_{\xbar}} \pa{\wk - \wbar} + o(\norm{\wk-\wbar}) .
\]
\end{remark}
\begin{proof}
Since $R$ is proper convex and lower semi-continuous, we have $R(\xk) \to R(\xbar)$ and $\partial R(\xk) \ni \uk = - A^T (A\xk - \wk)/\gamma \to \ubar \in \ri \pa{\partial R(\xbar)}$, hence $\dist\pa{\uk, \partial R(\xbar)} \to 0$. As a result, we have $\xk \in \calM_{\xbar}$ owing to \cite[Theorem~5.3]{hare2004identifying} and $\uk \in \ri\pa{\partial R(\xk)}$ owing to \cite{vaiter2018model} for all $k$ large enough. 

Denote $T_{\xk}, T_{\xkm}$ the tangent spaces of $\calM_{\xbar}$ at $\xk$ and $\xkm$. Denote $\tau_{k} : T_{\xk} \to T_{\xkm}$ the parallel translation along the unique geodesic on $\calM_{\xbar}$ joining $\xk$ to $\xkm$. 
From the definition of $\xk$, let $\hk = \gamma \uk$, we get
\[
\hk \eqdef - A^T\pa{ A\xk - \wk } \in \gamma \partial R(\xk) 
\qandq
\hkm \eqdef - A^T\pa{ A\xkm - \wkm } \in \gamma \partial R(\xkm) .
\]
Projecting onto corresponding tangent spaces, applying Lemma \ref{def:riemannian-gradhess} 
and the parallel translation $\tau_{k}$ leads to
\[
\begin{aligned}
\gamma \tau_{k} \nabla_{\calM_{\xbar}} R(\xk) 
&= \tau_{k} \PT{T_{\xk}} \pa{ \hk } 
= \PT{T_{\xkm}} \pa{ \hk } + \Pa{\tau_{k} \PT{T_{\xk}} - \PT{T_{\xkm}}} \pa{ \hk } , \\
\gamma \nabla_{\calM_{\xbar}} R(\xkm) 
&= \PT{T_{\xkm}} \pa{ \hkm } .
\end{aligned}
\]
The difference of the above two equalities yields
\beq\label{eq:wk-xk}
\begin{aligned}
\gamma \tau_{k} \nabla_{\calM_{\xbar}} R(\xk) - \gamma \nabla_{\calM_{\xbar}} R(\xkm) - \Pa{\tau_{k} \PT{T_{\xk}} - \PT{T_{\xkm}}} \pa{ \hkm } 
&= \PT{T_{\xkm}} \pa{ \hk - \hkm } + \Pa{\tau_{k} \PT{T_{\xk}} - \PT{T_{\xkm}}} \pa{ \hk - \hkm } .
\end{aligned}
\eeq
Owing to the monotonicity of subdifferential, \ie $\iprod{\hk-\hkm}{\xk-\xkm} \geq 0$, we get
\[
\iprod{A^TA(\xk-\xkm)}{\xk-\xkm} 
\leq \iprod{A^T(\wk-\wkm)}{\xk-\xkm} 
\leq \norm{A}\norm{\wk-\wkm}\norm{\xk-\xkm} .
\]
Since $A$ has full column rank, $A^TA$ is symmetric positive definite, and there exists $\kappa > 0$ such that $\kappa \norm{\xk-\xkm}^2 \leq \iprod{A^TA(\xk-\xkm)}{\xk-\xkm}$. Back to the above inequality, we get $\norm{\xk-\xkm} \leq \frac{\norm{A}}{\kappa} \norm{\wk-\wkm}$. 
Therefore for $\norm{\hk-\hkm}$, we get
\[
\begin{aligned}
\norm{\hk - \hkm} 
=\norm{ A^T\pa{ A\xk - \wk } - A^T\pa{ A\xkm - \wkm } } 
\leq \norm{A}^2 \norm{\xk-\xkm} + \norm{A}\norm{\wk-\wkm} 
&\leq \Pa{ \sfrac{\norm{A}^3}{\kappa} + \norm{A}} \norm{\wk - \wkm} .
\end{aligned}
\]
As a result, owing to Lemma \ref{lem:parallel-translation}, we have for the term $\pa{\tau_{k} \PT{T_{\xk}} - \PT{T_{\xkm}}} \pa{ \hk - \hkm }$ in \eqref{eq:wk-xk} that
\[
\Pa{\tau_{k} \PT{T_{\xk}} - \PT{T_{\xkm}}} \pa{ \hk - \hkm }	\\
= o(\norm{ \hk-\hkm } ) 
= o(\norm{ \wk-\wkm } ) .
\] 
Define $\barR_{k-1}(x) \eqdef \gamma R(x) - \iprod{x}{\hkm} $ and $H_{\barR, k-1} \eqdef \PT{T_{\xkm}} \nabla_{\calM_{\xbar}}^2 \barR_{k-1}(\xkm) \PT{T_{\xkm}}$, then with Lemma \ref{lem:taylor-expn} the Riemannian Taylor expansion, we have for the first line of \eqref{eq:wk-xk}
\beq\label{eq:rie-hess-xk}
\begin{aligned}
 \gamma \tau_{k} \nabla_{\calM_{\xbar}} R(\xk) - \gamma \nabla_{\calM_{\xbar}} R(\xkm) - \Pa{\tau_{k} \PT{T_{\xk}} - \PT{T_{\xkm}}} \pa{\hkm} 
&= \tau_{k} \Pa{\gamma \nabla_{\calM_{\xbar}} R(\xk) - \PT{T_{\xk}} \pa{\hkm}} - \Pa{ \gamma \nabla_{\calM_{\xbar}} R(\xkm) - \PT{T_{\xkm}} \pa{\hkm} } \\
&= \tau_{k} \nabla_{\calM_{\xbar}} \barR_{k-1} (\xk) - \nabla_{\calM_{\xbar}} \barR_{k-1} (\xkm) \\
&= H_{\barR, k-1} (\xk-\xkm) + o(\norm{\xk-\xkm}) \\
&= H_{\barR, k-1} (\xk-\xkm) + o(\norm{\wk-\wkm}) .
\end{aligned}
\eeq
Back to \eqref{eq:wk-xk}, we get
\beq\label{eq:wk-xk-2}
\begin{aligned}
H_{\barR, k-1} (\xk-\xkm)
&= \PT{T_{\xkm}} \pa{ \hk - \hkm } + o(\norm{\wk-\wkm}) .
\end{aligned}
\eeq
Define $ \barR(x) \eqdef \gamma R(x) - \iprod{x}{ \bar{h} } $ and $H_{\barR} = \PT{T_{\xbar}} \nabla^2_{\calM_{\xbar}} \barR\pa{\xbar} \PT{T_{\xbar}}$, then from \eqref{eq:wk-xk-2} that
\beq\label{eq:xk-xkm-HbarR}
\begin{aligned}
H_{\barR} (\xk-\xkm) + \Pa{ H_{\barR, k-1} - H_{\barR} } (\xk-\xkm) 
&= \PT{T_{\xbar}} \pa{ \hk - \hkm } + \Pa{ \PT{T_{\xkm}} - \PT{T_{\xbar}} } \pa{ \hk - \hkm } + o(\norm{\wk-\wkm}) .
\end{aligned}
\eeq
Owing to continuity, we have $H_{\barR, k-1} \to H_{\barR}$ and $\PT{T_{\xkm}} \to \PT{T_{\xbar}}$, 
\[
\begin{gathered}
\lim_{k\to\pinf} \tfrac{ \norm{ \pa{ H_{\barR, k-1} - H_{\barR} } (\xk-\xkm) } }{ \norm{\xk-\xkm} }
\leq \lim_{k\to\pinf} \tfrac{ \norm{ H_{\barR, k-1} - H_{\barR} } \norm{ \xk-\xkm } }{ \norm{\xk-\xkm} }
= \lim_{k\to\pinf} \norm{ H_{\barR, k-1} - H_{\barR} } = 0 , \\
\lim_{k\to\pinf} \tfrac{ \norm{ \pa{ \PT{T_{\xkm}} - \PT{T_{\xbar}} } \pa{\wk - \wkm} } }{ \norm{\wk-\wkm} }
\leq \lim_{k\to\pinf} \tfrac{ \norm{ \PT{T_{\xkm}} - \PT{T_{\xbar}} } \norm{ \wk-\wkm } }{ \norm{\wk-\wkm} }
= \lim_{k\to\pinf} \norm{ \PT{T_{\xkm}} - \PT{T_{\xbar}} } = 0 , 
\end{gathered}
\]
and $\lim_{k\to\pinf} \tfrac{ \norm{ \pa{ \PT{T_{\xkm}} - \PT{T_{\xbar}}} \pa{\xk - \xkm} } }{ \norm{\xk-\xkm} } = 0$. 
Combining this with the definition of $\uk$, the fact that $\xk-\xkm = \PT{T_{\xbar}}(\xk-\xkm) + o(\norm{\xk-\xkm})$ from Lemma \ref{lem:proj-M}, and denoting $A_{T_{\xbar}} = A \circ \PT{T_{\xbar}}$, equation \eqref{eq:xk-xkm-HbarR} can be written as
\beq\label{eq:aaa}
\begin{aligned}
H_{\barR} (\xk-\xkm) 
= \PT{T_{\xbar}} \pa{ \uk - \ukm } + o(\norm{\wk-\wkm}) 
&= - \PT{T_{\xbar}} \pa{ A^T\pa{ A\xk - \wk } - A^T\pa{ A\xkm - \wkm } } + o(\norm{\wk-\wkm}) \\
&= - \PT{T_{\xbar}} A^T A \pa{ \xk - \xkm } + \PT{T_{\xbar}} A^T \pa{ \wk - \wkm } + o(\norm{\wk-\wkm}) \\
&= - A_{T_{\xbar}}^T A_{T_{\xbar}} \pa{ \xk - \xkm } + A_{T_{\xbar}}^T \pa{ \wk - \wkm } + o(\norm{\wk-\wkm}) .
\end{aligned}
\eeq
Since $A$ has full rank, so is $A_{T_{\xbar}}$. Hence $A_{T_{\xbar}}^TA_{T_{\xbar}}$ is invertible along $T_{\xbar}$ and from above we have
\[
\begin{aligned}
\Pa{ \Id + (A_{T_{\xbar}}^T A_{T_{\xbar}})^{-1} H_{\barR} } (\xk-\xkm) 
&= (A_{T_{\xbar}}^T A_{T_{\xbar}})^{-1} A_{T_{\xbar}}^T \pa{\wk-\wkm} + o(\norm{\wk-\wkm}) .
\end{aligned}
\]
Denote $M_{\barR} = A_{T_{\xbar}} \pa{ \Id + (A_{T_{\xbar}}^T A_{T_{\xbar}})^{-1} H_{\barR} }^{-1} (A_{T_{\xbar}}^T A_{T_{\xbar}})^{-1} A_{T_{\xbar}}^T$, then
\beq\label{eq:xkxkm-MbarR-wkwkm}
\begin{aligned}
A_{T_{\xbar}} (\xk-\xkm) 
&= M_{\barR} \pa{\wk-\wkm} + o(\norm{\wk-\wkm}) ,
\end{aligned}
\eeq 
which concludes the proof. \qedhere

\end{proof}

\section{Trajectory of ADMM}
\label{P-sec:tra-admm}

We first provide the fixed-point characterization of ADMM based on the equivalence between ADMM and Douglas--Rachford \cite{douglas1956numerical} and Peaceman--Rachford splitting \cite{peaceman1955numerical} methods, and then present the proofs for the trajectory of ADMM.

\subsection{Fixed-point characterization and convergence of ADMM} \label{sec:admm-dr-pr}

The dual problem of \eqref{eq:problem-admm} reads 
\beqn\tag{$\calD_{\mathrm{ADMM}}$}
\max_{\psi\in\bbR^p } - \Pa{R^*(-A^T \psi) + J^*(-B^T \psi) + \iprod{\psi}{b} } ,
\eeqn 
where $R^*(v) \eqdef \sup_{x\in\bbR^n} ~ \pa{ \iprod{x}{v} - R(x) }$ is called the Fenchel conjugate, or simply conjugate, of $R$. 

\subsubsection{Relaxed ADMM and Douglas--Rachford splitting}\label{subsec:admm-dr}
It is well-known that ADMM is equivalent to applying Douglas--Rachford splitting \cite{douglas1956numerical} to the dual problem \eqref{eq:problem-admm-dual}. 
Below we first recall the equivalence between ADMM and Douglas--Rachford which was first established in \cite{gabay1983chapter}, and then use the convergence of Douglas--Rachford splitting which is well established in the literature \cite{bauschke2011convex} to conclude the convergence of ADMM. 

For the sake of generality, we consider the following so called \emph{relaxed ADMM}
\[
\begin{aligned}
\xk &= \argmin_{x\in\bbR^n }~ R(x) + \tfrac{\gamma}{2} \norm{Ax+B\ykm-b + \tfrac{1}{\gamma}\psikm }^2 , \\
\xbark &= \phi A\xk - (1-\phi)(B\ykm-b) , \\
\yk &= \argmin_{y\in\bbR^m }~ J(y) + \tfrac{\gamma}{2} \norm{\xbark+By-b + \tfrac{1}{\gamma}\psikm }^2 , \\
\psik &= \psikm + \gamma\pa{\xbark + B\yk - b}  ,
\end{aligned}
\]
where $\phi \in [0, 2]$ is the relaxation parameter. When $\phi = 1$, the relaxed ADMM recovers the standard ADMM \eqref{eq:admm2}. Below show demonstrate that the relaxed ADMM is equivalent to the relaxed Douglas--Rachford applying to solve \eqref{eq:problem-admm-dual}. 
\begin{itemize}[leftmargin=2em]
\item 
Define $\zk = \psik - \gamma (B\yk - b)$, we have
\[
\begin{aligned}
\zk 
= \psik - \gamma B\yk + \gamma b 
= \psikm + \gamma \xbark 
&= \phi \psikm + \phi \gamma A\xk + (1-\phi) \psikm - (1-\phi) \gamma (B\ykm-b)  \\
&= (1-\phi)\zkm + \phi(\psikm + \gamma A\xk)  \\ 
&= (1-\phi)\zkm + \phi(\zkm + \uk - \psikm)  .
\end{aligned}
\]
When $\phi = 1$, we have $\zk = \psikm + \gamma A\xk$.

\item 
For the update of $\xk$, denote $\uk = \psikm + \gamma (A\xk + B\ykm - b)$. Since $A$ has full column rank, we have $\xk$ is the unique minimiser of $R(x) + \frac{\gamma}{2} \norm{Ax+B\ykm-b + \tfrac{1}{\gamma}\psikm }^2$. Let $R^*$ be the conjugate of $R$, then owing to duality, we get
\[
\begin{aligned}
\xk = \argmin_{x\in\bbR^n }~ R(x) + \sfrac{\gamma}{2} \norm{Ax+B\ykm-b + \tfrac{1}{\gamma}\psikm }^2 
\Longleftrightarrow\enskip & 
0 \in \partial R(\xk) + \gamma A^T\Pa{ A\xk + B\ykm - b + \tfrac{1}{\gamma}\psikm } \\
\enskip\Longleftrightarrow\enskip & 
- A^T \uk \in \partial R(\xk) \\
\enskip\Longleftrightarrow\enskip & 
\xk \in \partial R^*( - A^T \uk ) \\
\enskip\Longleftrightarrow\enskip & 
\uk - \gamma A \xk \in \uk + \gamma \partial (R^* \circ - A^T )( \uk ) \\
\enskip\Longleftrightarrow\enskip & 
\uk = \Pa{ \Id + \gamma \partial (R^* \circ - A^T ) }^{-1} ( \uk - \gamma A \xk ) \\
\enskip\Longleftrightarrow\enskip & 
\uk = \Pa{ \Id + \gamma \partial (R^* \circ - A^T ) }^{-1} ( 2\psikm - \zkm ) .
\end{aligned}
\]

\item For the update of $\yk$, the full column rank of $B$ also ensures that $\yk$ is the unique minimiser of $J(y) + \sfrac{\gamma}{2} \norm{\xbark+By-b + \frac{1}{\gamma}\psikm }^2$. Since $\psik = \psikm + \gamma \pa{\xbark+B\yk - b}$, then
\[
\begin{aligned}
 \yk = \argmin_{y\in\bbR^m }~ J(y) + \sfrac{\gamma}{2} \norm{\xbark+By-b + \tfrac{1}{\gamma}\psikm }^2 
\Longleftrightarrow\enskip & 
0 \in \partial J(\yk) + \gamma B^T\Pa{ \xbark + B\yk - b + \tfrac{1}{\gamma}\psikm } \\
\enskip\Longleftrightarrow\enskip & 
- B^T \psik \in \partial J(\yk) \\
\enskip\Longleftrightarrow\enskip & 
\yk \in \partial J^*(- B^T \psik) \\
\enskip\Longleftrightarrow\enskip & 
\psik - \gamma B \yk \in \psik + \gamma \partial (J^* \circ - B^T) (\psik) \\
\enskip\Longleftrightarrow\enskip & 
\psik = \Pa{\Id + \gamma \partial (J^* \circ - B^T)}^{-1} (\psik - \gamma B \yk) \\
\enskip\Longleftrightarrow\enskip & 
\psik = \Pa{\Id + \gamma \partial (J^* \circ - B^T)}^{-1} (\zk - \gamma b) .
\end{aligned}
\]

\item 
Combining all the relations we get
\beq\label{eq:dr-on-dual}
\begin{aligned}
\uk &= \Pa{ \Id + \gamma \partial (R^* \circ - A^T ) }^{-1} ( 2\psikm - \zkm ) , \\
\zk &= (1-\phi)\zkm + \phi(\zkm + \uk - \psikm) , \\
\psik &= \Pa{\Id + \gamma \partial (J^* \circ - B^T)}^{-1} (\zk - \gamma b ) ,
\end{aligned}
\eeq
which is exactly the iteration of Douglas--Rachford splitting applied to solve the dual problem \eqref{eq:problem-admm-dual} with . 

\end{itemize}
Define the following operator 
\[
\fDR = \sfrac{1}{2} \Id + \sfrac{1}{2} \bPa{ 2 \Pa{ \Id + \gamma \partial (R^* \circ - A^T ) }^{-1} - \Id } \bPa{ 2\Pa{\Id + \gamma \partial (J^* \circ - B^T)}^{-1} - \Id } ,
\]
then \eqref{eq:dr-on-dual} can be written as the fixed-point iteration in terms of $\zk$, that is
\[
\zk = \fDR(\zkm) . 
\]
It should be noted that for $\zk$ we have $\zk = \psik - \gamma B\yk + \gamma b = \psikm + \gamma A \xk$ which is the same as in \eqref{eq:admm2}. 
Owing to \cite{bauschke2011convex}, we have that $\fDR$ is firmly non-expansive with the set of fixed-points $\fix(\fDR)$ being non-empty, and there exists a fixed-point $\zsol \in \fix(\fDR)$ such that $\zk \to \zsol$ which concludes the convergence of $\seq{\zk}$. Then we have $\uk, \psik$ converging to $\psisol = \Pa{\Id + \gamma \partial (J^* \circ - B^T)}^{-1} (\zsol - \gamma b )$ which is a dual solution of the problem \eqref{eq:problem-admm-dual}. 
The convergence of the primal ADMM sequences $\seq{\xk}$ and $\seq{\yk}$ follows immediately.

Owing to the above equivalence between ADMM and Douglas--Rachford splitting, we get the following relations
\beq\label{eq:relations}
\begin{aligned}
\norm{\zk - \zkm}
&\leq \norm{\zkm - \zkmm} , \\
\norm{\psik - \psikm} 
&\leq \norm{\zk - \zkm} \leq \norm{\zkm - \zkmm} , \\
\norm{\uk-\ukm}
&\leq \norm{2\psikm - \zkm - 2\psikmm + \zkmm} \leq 3\norm{\zkm-\zkmm} , \\
\gamma\norm{A\xk - A\xkm}
&\leq \norm{\zk-\zkm} + \norm{\psikm-\psikmm} 
\leq 2\norm{\zkm-\zkmm} , \\
\gamma\norm{B\yk - B\ykm}
&\leq \norm{\zk-\zkm} + \norm{\psik-\psikm} \leq 2 \norm{\zkm-\zkmm} ,
\end{aligned}
\eeq
which are needed in the proofs below.

\subsubsection{Symmetric ADMM and Peaceman--Rachford splitting}\label{subsec:admm-pr}

Below we present a short discussion on the relation between the symmetric ADMM and Peaceman--Rachford splitting method, which was first established in \cite{gabay1983chapter}. 
\begin{itemize}[leftmargin=2em]

\item 
For the update of $\xk$, let $\uk = \psikmh = \psikm + \gamma (A\xk + B\ykm - b)$ and $\zk = \psik - \gamma B\yk + \gamma b$. As $A$ has full column rank, $\xk$ is the unique minimiser of $R(x) + \frac{\gamma}{2} \norm{Ax+B\ykm-b + \tfrac{1}{\gamma}\psikm }^2$. Then owing to duality,
\[
\begin{aligned}
\xk = \argmin_{x\in\bbR^n }~ R(x) + \tfrac{\gamma}{2} \norm{Ax+B\ykm-b + \tfrac{1}{\gamma}\psikm }^2 
\enskip\Longleftrightarrow\enskip & 
- A^T \uk \in \partial R(\xk) \\
\enskip\Longleftrightarrow\enskip & 
\xk \in \partial R^*( - A^T \uk ) \\
\enskip\Longleftrightarrow\enskip & 
\uk = \Pa{ \Id + \gamma \partial (R^* \circ - A^T ) }^{-1} ( \uk - \gamma A \xk ) \\
\enskip\Longleftrightarrow\enskip & 
\uk = \Pa{ \Id + \gamma \partial (R^* \circ - A^T ) }^{-1} ( 2\psikm - \zkm ) .
\end{aligned}
\]

\item For $\yk$, the full column rank of $B$ ensures the uniqueness of $\yk$. Since $\psik = \psikmh + \gamma \pa{A\xk+B\yk - b}$, then
\[
\begin{aligned}
 \yk = \argmin_{y\in\bbR^m }~ J(y) + \sfrac{\gamma}{2} \norm{A\xk+By-b + \tfrac{1}{\gamma}\psikmh }^2 
\enskip\Longleftrightarrow\enskip & 
- B^T \psik \in \partial J(\yk) \\
\enskip\Longleftrightarrow\enskip & 
\yk \in \partial J^*(- B^T \psik) \\
\enskip\Longleftrightarrow\enskip & 
\psik = \Pa{\Id + \gamma \partial (J^* \circ - B^T)}^{-1} (\psik - \gamma B \yk) \\
\enskip\Longleftrightarrow\enskip & 
\psik = \Pa{\Id + \gamma \partial (J^* \circ - B^T)}^{-1} (\zk - \gamma b) .
\end{aligned}
\]

\item For $\zk$, since $\uk = \psikmh$, 
\[
\begin{aligned}
\zk
= \psik - \gamma B\yk + \gamma b
= \uk + \gamma A\xk
= 2\uk - \psikm - \gamma(B\ykm - b) 
= \zkm + 2(\uk - \psikm)   .
\end{aligned}
\]
Combining the above relations we get
\beq\label{eq:pd-on-dual}
\begin{aligned}
\uk &= \Pa{ \Id + \gamma \partial (R^* \circ - A^T ) }^{-1} ( 2\psikm - \zkm ) , \\
\zk &= \zkm + 2(\uk - \psikm) , \\
\psik &= \Pa{\Id + \gamma \partial (J^* \circ - B^T)}^{-1} (\zk - \gamma b ) ,
\end{aligned}
\eeq
which is the iteration of Peaceman--Rachford splitting when applied to solve \eqref{eq:problem-admm-dual}. 

\end{itemize}
Define the following operator 
\[
\fPR = \Pa{ 2 \Pa{ \Id + \gamma \partial (R^* \circ - A^T ) }^{-1} - \Id } \bPa{ 2\Pa{\Id + \gamma \partial (J^* \circ - B^T)}^{-1} - \Id } ,
\]
then \eqref{eq:pd-on-dual} can be written as the fixed-point iteration in terms of $\zk$, that is
\[
\zk = \fPR(\zkm) . 
\]
It should be noted that for $\zk$ we have $\zk = \psik - \gamma B\yk + \gamma b = \psikm + \gamma A \xk$ which is the same as in \eqref{eq:admm_symmetric2}. 
Different to the case of Douglas--Rachford, the operator $\fPR$ is only non-expansive \cite{bauschke2011convex}, hence the conditions for $\zk$ to be convergent is stronger than that of $\fDR$. However, when it converges, it tends to be faster than Douglas--Rachford splitting \cite{gabay1983chapter}.

\subsection{Trajectory of ADMM: both $R, J$ are non-smooth}

Given a saddle point $(\xsol, \ysol, \psisol)$ of $\calL(x,y; \psi)$, the first-order optimality condition entails $- A^T \psisol \in {\partial R(\xsol)} $ and $- B^T \psisol \in {\partial J(\ysol)} $. 
Below we impose a stronger condition 
\beq\label{eq:ndc-admm}\tag{$\textrm{ND}$}
- A^T \psisol \in \ri\Pa{\partial R(\xsol)} 
\qandq
- B^T \psisol \in \ri\Pa{\partial J(\ysol)} .
\eeq
Suppose $R \in \PSF{\xsol}{\MmRx}, J \in \PSF{\ysol}{\MmJy}$ are partly smooth, denote $T_{\xsol}^R, T_{\ysol}^J$ the tangent spaces of $\MmRx, \MmJy$ at $\xsol, \ysol$, respectively. 
Define the following smooth perturbation of $R, J$,
\beq\label{eq:barR-barJ-admm}
\barR(x) \eqdef \sfrac{1}{\gamma} \Pa{ R(x) - \iprod{x}{-A^T \psisol } } ,~~~
\barJ(y) \eqdef \sfrac{1}{\gamma} \Pa{ J(y) - \iprod{w}{-B^T \psisol} } ,
\eeq
their Riemannian Hessian $H_{\barR} \eqdef \PR \nabla^2_{\MmRx} \barR\pa{\xsol} \PR, ~
H_{\barJ} \eqdef \PJ \nabla^2_{\MmJy} \barJ\pa{\ysol} \PJ$ and 
\beq\label{eq:BarBarQ}
\begin{gathered}
%
M_{\barR} \eqdef A_{R} \Pa{ \Id+(A_{R}^TA_{R})^{-1} {H_{\barR}} }^{-1} (A_{R}^TA_{R})^{-1} A_{R}^{T} , \\
M_{\barJ} \eqdef B_{J} \Pa{ \Id+(B_{J}^TB_{J})^{-1} {H_{\barJ}} }^{-1} (B_{J}^TB_{J})^{-1} B_{J}^{T} ,
\end{gathered}
\eeq
where $A_{R} \eqdef A \circ \PTR{\xsol}, B_J \eqdef B \circ \PTJ{\ysol}$. 
%
Finally, define 
\beq\label{eq:mtx-M-admm}
\mADMM 
\eqdef 
\sfrac12\Id + \sfrac12\pa{2M_{\barR}-\Id}\pa{2 M_{\barJ} - \Id} .
\eeq

\begin{proof}[Proof of Theorem \ref{prop:trajectory-admm}]

The proof of Theorem \ref{prop:trajectory-admm} is split into several steps: finite manifold identification of ADMM, local linearization based on partial smoothness, spectral properties of the linearized matrix, and the trajectory of $\seq{\zk}$. Let $(\xsol, \ysol, \psisol)$ be a saddle-point of $\calL(x,y; \psi)$.


\paragraph{1. Finite manifold identification of ADMM}
The finite manifold identification of ADMM is already discussed in \cite{liang2017localDR}, below we present a short discussion for the sake of self-consistency. 
At convergence of ADMM, owing to \eqref{eq:admm2} we have
\[
\begin{aligned}
A^T \psisol &= \gamma A^T\Pa{ A\xsol - \tfrac{1}{\gamma } \pa{ \zsol - 2\psisol } } \qandq
B^T \psisol &= \gamma B^T\Pa{ B\ysol - \tfrac{1}{\gamma } \pa{ \zsol - \gamma b } } .
\end{aligned}
\]
From the update of $\xk, \yk$ in \eqref{eq:admm2}, we have the following monotone inclusions
\beqn
\begin{aligned}
- \gamma A^T\Pa{ A\xk - \tfrac{1}{\gamma } \pa{ \zkm - 2\psikm } } &\in \partial R(\xk) &\!\!\textrm{and}\!\!&&
- \gamma B^T\Pa{ B\yk - \tfrac{1}{\gamma } \pa{ \zk - \gamma b } } &\in \partial J(\yk) , \\
- \gamma A^T\Pa{ A\xsol - \tfrac{1}{\gamma } \pa{ \zsol - 2\psisol } } &\in \partial R(\xsol) &\!\!\textrm{and}\!\!&&
- \gamma B^T\Pa{ B\ysol - \tfrac{1}{\gamma } \pa{ \zsol - \gamma b } } &\in \partial J(\ysol) .
\end{aligned}
\eeqn
Since $A$ is bounded, 
it then follows that
\[
\begin{aligned}
\dist\Pa{-A^T\psisol, \partial R(\xk)}
&\leq \gamma \norm{ A^T\Pa{ A\xk - \tfrac{1}{\gamma } \pa{ \zkm - 2\psikm } } - A^T\Pa{ A\xsol - \tfrac{1}{\gamma } \pa{ \zsol - 2\psisol } } } \\
&\leq \gamma \norm{A}\norm{ A(\xk-\xsol) - \tfrac{1}{\gamma }(\zkm-\zsol) + \tfrac{2}{\gamma } (\psikm -\psisol) } \\
&\leq \gamma \norm{A}\Pa{ \norm{A}\norm{\xk-\xsol} + \tfrac{1}{\gamma }\norm{\zkm-\zsol} + \tfrac{2}{\gamma } \norm{\psikm -\psisol} } 
\to 0 .
\end{aligned}
\]
and similarly
\[
\dist\Pa{-B^T\psisol, \partial J(\yk)}
\leq \gamma \norm{B}\Pa{ \norm{B}\norm{\yk-\ysol} + \tfrac{1}{\gamma }\norm{\zk-\zsol} } 
\to 0 .
\]
Since $R\in \lsc\pa{\bbR^n}$ and $J \in \lsc\pa{\bbR^m}$, then by the sub-differentially continuous property of them we have $R(\xk) \to R(\xsol)$ and $J(\yk) \to J(\ysol)$. 
Hence the conditions of \cite[Theorem~5.3]{hare2004identifying} are fulfilled for $R$ and $J$, and there exists $K$ large enough such that for all $k \geq K$, there holds
\[
(\xk, \yk) \in \MmRx \times \MmJy ,
\]
which is the finite manifold identification.


\paragraph{2. linearization of ADMM} 
For convenience, denote $\beta = 1/\gamma$. 
%
%
%
%
For the update of $\yk$, define $\wk = - \beta \pa{ \zk - \gamma b }$, we have from \eqref{eq:admm2} that
\[
\yk = \argmin_{y\in\bbR^m }~ \beta J(y) + \sfrac{1}{2} \norm{By - \wk }^2  . 
\]
Owing to the optimality condition of a saddle point, define $ \barJ(y) \eqdef { \beta J(y) - \iprod{y}{ - \beta B^T \psisol } }$ and its Riemannian Hessian $H_{\barJ} = \PTJ{\ysol} \nabla^2_{\MmJy} \barJ \pa{\ysol} \PTJ{\ysol}$. For $B$, define $B_J = B \circ \PTJ{\ysol}$, and $ M_{\barJ} = B_{J} \pa{\Id + (B_{J}^TB_{J})^{-1}H_{\barJ} }^{-1} (B_{J}^TB_{J})^{-1} B_{J}^T $. 
Then owing to Lemma \ref{lem:lin-generalised-ppa}, we get 
\beq\label{eq:ukp-usol-admm}
\begin{aligned}
B_J (\yk-\ykm) 
&= M_{\barJ} (\wk - \wkm) + o(\norm{\wk-\wkm}) \\
&= - \beta M_{\barJ} (\zk - \zkm) + o(\norm{\zk-\zkm}) .
\end{aligned}
\eeq
%
%
%
Now consider $\xk$ and let $\wk = \beta \pa{ \zkm - 2\psikm }$, we get from \eqref{eq:admm2} that 
\[
\xk = \argmin_{x\in\bbR^n }~ \beta R(x) + \sfrac{1}{2} \norm{Ax - \wk }^2 .
\]
Define $ \barR(x) \eqdef \beta R(x) - \iprod{x}{ - \beta A^T\psisol } $ and its Riemannian Hessian $H_{\barR} = \PTR{\xsol} \nabla^2_{\MmRx} \barR\pa{\xsol} \PTR{\xsol}$. 
Denote $A_{R} = A \circ \PTR{\xsol}$, and $ M_{\barR} = A_{R} \pa{ \Id + (A_{R}^T A_{R})^{-1} H_{\barR} }^{-1} (A_{R}^T A_{R})^{-1} A_{R}^T $. 
Note from \eqref{eq:admm2} that $\psikm - \psikmm = \zkm - \zkmm + \gamma B (\ykm - \ykmm)$, then
\[
\begin{aligned}
\wk - \wkm
&= \beta (\zkm-\zkmm) - 2\beta (\psikm-\psikmm) \\
&= - \beta (\zkm-\zkmm) - 2\beta \gamma B (\ykm-\ykmm) \\
&= - \beta (\zkm-\zkmm) - 2 B_{J} (\ykm-\ykmm) + o(\norm{\ykm-\ykmm}) ,
\end{aligned}
\]
where $\ykm-\ykmm = \PTR{\xsol}(\ykm-\ykmm) + o( \norm{\ykm-\ykmm} )$ from Lemma \ref{lem:proj-M} is applied. 
From \eqref{eq:relations}, we have $o(\norm{\ykm-\ykmm}) = o(\norm{\zkm-\zkmm})$ and $o(\norm{\wkm-\wkmm}) = o(\norm{\zkm-\zkmm})$, then applying Lemma \ref{lem:lin-generalised-ppa} yields, 
\beq\label{eq:xkzk-admm}
\begin{aligned}
A_{R} (\xk-\xkm) 
&= M_{\barR} (\wk - \wkm) + o(\norm{\wk-\wkm}) \\
&= - \beta M_{\barR} \pa{\zkm-\zkmm} + 2 M_{\barR} B_{J} \pa{\ykm-\ykmm} + o(\norm{ \zkm-\zkmm } ) \\
&= - \beta M_{\barR} \pa{\zkm-\zkmm} + 2 \beta M_{\barR} M_{\barJ} (\zkm - \zkmm) + o(\norm{ \zkm-\zkmm } ) . 
\end{aligned}
\eeq 
Finally, from \eqref{eq:admm2}, \eqref{eq:ukp-usol-admm} and \eqref{eq:xkzk-admm}, we have that
\[
\begin{aligned}
\zk - \zkm
&= \Pa{ \zkm + \gamma\pa{A\xk + B\ykm - b} } - \Pa{ \zkmm + \gamma\pa{A\xkm + B\ykmm - b} } \\
&= (\zkm-\zkmm) + \gamma A(\xk-\xkm) + \gamma B (\ykm-\ykmm) \\
&= (\zkm-\zkmm) + \gamma A_{R}(\xk-\xkm) + \gamma B_J (\ykm-\ykmm) + o(\norm{ \zkm-\zkmm } ) \\
&= (\zkm-\zkmm) - M_{\barR} \pa{\zkm-\zkmm} + 2 M_{\barR} M_{\barJ} (\zkm-\zkmm) + M_{\barJ} (\zkm-\zkmm) + o(\norm{\zkm-\zkmm} ) \\
&= \Pa{ \Id + 2M_{\barR}M_{\barJ} - M_{\barR} - M_{\barJ} } (\zkm-\zkmm) + o(\norm{ \zkm-\zkmm } ) ,
\end{aligned}
\]
which is the desired linearization of ADMM.


\paragraph{3. Spectral properties of $\mADMM$}

Consider first the case where both $R, J$ are general partly smooth functions, under which we can shown the non-expansiveness of $\mADMM$. 
For $M_{\barR}$, since $A$ is injective, so is $A_{R}$, then $A_{R}^TA_{R}$ is {symmetric positive definite}. Therefore, we have the following similarity result for $M_{\barR}$,
\beq\label{eq:BarU-sim}
\begin{aligned}
M_{\barR}
&= A_{R} \bPa{ (A_{R}^T A_{R})^{-\frac{1}{2}} \Pa{\Id + (A_{R}^T A_{R})^{-\frac{1}{2}} {H_{\barR}} (A_{R}^T A_{R})^{-\frac{1}{2}} } (A_{R}^T A_{R})^{\frac{1}{2}} }^{-1} (A_{R}^T A_{R})^{-1}A_{R}^T \\
&= A_{R} (A_{R}^T A_{R})^{-\frac{1}{2}} \Pa{\Id + (A_{R}^T A_{R})^{-\frac{1}{2}} {H_{\barR}} (A_{R}^T A_{R})^{-\frac{1}{2}} }^{-1} (A_{R}^T A_{R})^{\frac{1}{2}} (A_{R}^T A_{R})^{-1}A_{R}^T \\
&= A_{R} (A_{R}^T A_{R})^{-\frac{1}{2}} \Pa{\Id + (A_{R}^T A_{R})^{-\frac{1}{2}} {H_{\barR}} (A_{R}^T A_{R})^{-\frac{1}{2}} }^{-1} (A_{R}^T A_{R})^{-\frac{1}{2}} A_{R}^T .
\end{aligned}
\eeq
Since $(A_{R}^T A_{R})^{-\frac{1}{2}} {H_{\barR}} (A_{R}^T A_{R})^{-\frac{1}{2}}$ is symmetric positive definite, hence maximal monotone, then the matrix $$\pa{\Id + (A_{R}^T A_{R})^{-\frac{1}{2}} {H_{\barR}} (A_{R}^T A_{R})^{-\frac{1}{2}} }^{-1}$$ is firmly non-expansive. 
Let $A_{R} = USV^T$ be the SVD of $A_R$, then we have
\[
\begin{aligned}
\norm{A_{R} (A_{R}^T A_{R})^{-\frac{1}{2}}}
= \norm{USV^T (VSU^T USV^T)^{-\frac{1}{2}}} 
= \norm{USV^T (VS^2V^T)^{-\frac{1}{2}}} 
= \norm{USV^T VS^{-1} V^T} 
= 1 .
\end{aligned}
\]
Then owing to \cite[Example 4.14]{bauschke2011convex}, $M_{\barR}$ is firmly non-expansive. Similarly, $M_{\barJ}$ is firmly non-expansive, and so is $\mADMM$ \cite[Proposition 4.31]{bauschke2011convex}. Therefore, the power $\mADMM^k$ is convergent.

Now suppose that both $R, J$ are locally polyhedral around $(\xsol, \ysol)$, then $M_{\barR}$ and $M_{\barJ}$ become
\[
M_{\barR} = A_{R} (A_{R}^T A_{R})^{-1} A_{R}^T 
\qandq
M_{\barJ} = B_{J} (B_{J}^TB_{J})^{-1} B_{J}^T ,
\]
which are projection operators onto the ranges of $A_R$ and $B_J$, respectively. 
Denote these two subspaces by $T_{A_R}$ and $T_{B_J}$, and correspondingly $\proj_{T_{A_R}} \eqdef A_{R} (A_{R}^T A_{R})^{-1} A_{R}^T$ and $\proj_{T_{B_J}} \eqdef B_{J} (B_{J}^TB_{J})^{-1} B_{J}^T$. Then
\[
\mADMM = \proj_{T_{A_R}} \proj_{T_{B_J}} + (\Id - \proj_{T_{A_R}})(\Id - \proj_{T_{B_J}}) .
\]
Denote the dimension of $T_{A_R}, T_{B_J}$ by $\dim(T_{A_R}) = p, \dim(T_{B_J}) = q$, and the dimension of the intersection $\dim(T_{A_R} \cap T_{B_J}) = d$. Without the loss of generality, we assume that $1 \leq p \leq q \leq n$. 
Consequently, there are $r=p-d$ principal angles $(\zeta_i)_{i=1,...,r}$ between $T_{A_R}$ and $T_{B_J}$ that are strictly greater than $0$ and smaller than $\pi/2$. Suppose that $\zeta_1\leq \dotsm \leq \zeta_r$. 
Define the following two diagonal matrices
\[
C = \diag\Pa{ \cos(\zeta_1), \dotsm, \cos(\zeta_r) } \qandq
S = \diag\Pa{ \sin(\zeta_1), \dotsm, \sin(\zeta_r) } .
\]
Owing to \cite{bauschke2014rate,demanet2016eventual}, there exists a real orthogonal matrix $U$ such that
\[
\mADMM
= U
\left[
\begin{array}{cc|cc}
C^2 & CS & 0 & 0\\
- CS & C^2 & 0 & 0\\
\hline
0 & 0 & 0_{q-p+2d} & 0\\
0 & 0 & 0 & \Id_{n-p-q}
\end{array}
\right]
U^T ,
\]
which indicates $\mADMM$ is normal and all its eigenvalues are inside unit disc.

Let $\mADMMinf = \lim_{k\to+\infty}\mADMM^k$ and $\tmADMM = \mADMM - \mADMMinf$, then we have
\beq\label{eq:tmADMM}
\tmADMM
= U
\left[
\begin{array}{cc|c}
C^2 & CS & 0 \\
- CS & C^2 & 0 \\
\hline
0 & 0 & 0_{n-2r}
\end{array}
\right]
U^T .
\eeq


\paragraph{4. Trajectory of ADMM}

Owing to the polyhedrality of $R$ and $J$, all the small $o$-terms in the linearization proof vanish and we get directly
\beq\label{eq:zkp-zk-poly}
\zk - \zkm
= \mADMM (\zkm - \zkmm) 
= \mADMM^k (z_{0} - z_{-1}) . 
\eeq
As $\vk \eqdef \zk-\zkm \to 0$, passing to the limit we get from above 
\[
0 = \lim_{k\to+\infty} \mADMM^k v_{0} = \mADMMinf v_{0} , 
\]
which means $v_{0} \in \ker(\mADMM)$ where $\ker(\mADMM)$ denotes the kernel of $\mADMM$. Since $\mADMMinf \mADMM^k = \mADMMinf$, we have $\vk \in \ker(\mADMM)$ holds for any $k\in\bbN$. Then from \eqref{eq:zkp-zk-poly} we have
\[
\vk
= (\mADMM - \mADMMinf) \vk
= \tmADMM \vkm .
\]
The block diagonal property of \eqref{eq:tmADMM} indicates that there exists an elementary transformation matrix $E$ such that
\[
\tmADMM
= U E
\begin{bmatrix}
B_{1} & & &\\
& \ddots & & \\
& & B_{r} &\\
& & & 0_{n-2r}\\
\end{bmatrix} 
E U^T ,
\]
where for each $i=1,...,r$, we have
\[
B_i = 
\cos(\zeta_i) \begin{bmatrix} \cos(\zeta_i) & \sin(\zeta_i) \\ - \sin(\zeta_i) & \cos(\zeta_i) \end{bmatrix}
\]
which is rotation matrix scaled by $\cos(\zeta_i)$. 
It is easy to show that, for each $i=1,...,d$, there holds
\[
\lim_{k\to\pinf} B_i^k = 0 ,
\]
since the spectral radius of $B_i$ is $\rho(B_i) = \cos(\zeta_i) < 1$.

\vgap

Suppose for some $1\leq e < r$, we have
\[
\zeta = \zeta_1 = \dotsm = \zeta_e < \zeta_{e+1} \leq \dotsm \leq \zeta_{r} .
\]
Consider the following decompositions
\[
\Gamma_1
= 
\begin{bmatrix}
B_{1} & & &\\
& \ddots & & \\
& & B_{e} &\\
& & & 0_{n-2e}
\end{bmatrix}
\qandq
\Gamma_2
= 
\begin{bmatrix}
B_{1} & & &\\
& \ddots & & \\
& & B_{r} &\\
& & & 0_{n-2r}\\
\end{bmatrix} 
 - \Gamma_1 .
\]
Denote $\eta = \frac{\cos(\zeta_{e+1})}{\cos(\zeta)}$, it is immediate to see that $\frac{1}{\cos^k(\zeta)} \Gamma_2^k = O( \eta^k ) \to 0$, and for each $i=1,...,e$
\[
\sfrac{1}{\cos(\zeta)} B_i = 
\begin{bmatrix} \cos(\zeta) & \sin(\zeta) \\ - \sin(\zeta) & \cos(\zeta) \end{bmatrix} 
\]
which is a circular rotation. 
Therefore, $\frac{1}{\cos(\zeta)} \Gamma_1$ is a rotation with respect to the first $2e$ elements. 
Denote $\uk = E U^T\vk$, then from $\vk = \tM \vkm = U E (\Gamma_1 + \Gamma_2) E U^T \vk$, we get 
\[
\uk 
= (\Gamma_1 + \Gamma_2) \uk
= (\Gamma_1 + \Gamma_2)^k u_0
= \Gamma_1^k u_0 + \Gamma_2^k u_0 ,
\]
which is an orthogonal decomposition of $\uk$. Define 
\[
s_k = \tfrac{1}{\cos^k(\zeta)} \Gamma_1^k u_1
\qandq
t_k = \tfrac{1}{\cos^k(\zeta)} \Gamma_2^k u_1 ,
\]
then we have that $\norm{s_{k}} = \norm{s_{k-1}}$ and $\iprod{s_{k}}{s_{k-1}} = \cos(\zeta) \norm{s_k}^2 $, and $t_k = O(\eta^k)$. 
As a result, for $\cos(\theta_k)$ we have
\beq\label{eq:cos-thetak-i}
\begin{aligned}
\cos(\theta_{k})
= \sfrac{ \iprod{\vk}{\vkm} }{ \norm{\vk} \norm{\vkm} } 
= \sfrac{ \iprod{\uk}{\ukm} }{ \norm{\uk} \norm{\ukm} } 
&= \sfrac{\iprod{ s_{k} + t_{k} }{ s_{k-1} + t_{k-1} }}{\norm{ s_{k} + t_{k} }\norm{ s_{k-1} + t_{k-1} }} \\ 
&= \sfrac{\iprod{ s_{k} }{ s_{k-1} }}{\norm{ s_{k} + t_{k} }\norm{ s_{k-1} + t_{k-1} }} + \sfrac{\iprod{ t_{k} }{ t_{k-1} }}{\norm{ s_{k} + t_{k} }\norm{ s_{k-1} + t_{k-1} }} \\
&= \sfrac{ \norm{s_k}^2\cos(\zeta) }{\norm{ s_k }^2 + \norm{t_k }^2} \cdot \sfrac{\norm{s_k + t_k}}{\norm{ s_{k-1} + t_{k-1} }} + O(\eta^{2k-1}) .
\end{aligned}
\eeq
Using the fact that
\[
\sfrac{ \norm{s_k}^2 \cos(\zeta) }{\norm{ s_k }^2 + \norm{t_k }^2} 
= \cos(\zeta) \Pa{1 - \norm{t_k }^2 + O(\norm{t_k }^4) } = \cos(\zeta) + O(\eta^{2k}) 
\qandq
\sfrac{\norm{s_k+t_k}}{\norm{ s_{k-1} + t_{k-1} }} \to 1 
\]
we conclude that $\cos(\theta_{k}) \to \cos(\zeta)$. 
As a matter of fact, we have $\cos(\theta_{k}) - \cos(\zeta) = O(\eta^{2k})$ which shows how fast $\cos(\theta_k)$ converges to $\cos(\zeta)$. 
\qedhere
\end{proof}

\subsection{Trajectory of ADMM: $R$ or/and $J$ is smooth}

Now we consider the case that at least one function out of $R, J$ is smooth. For simplicity, consider that $R$ is smooth and $J$ remains non-smooth. 
Assume that $R$ is locally $C^2$-smooth around $\xsol$, the Hessian of $R$ at $\xsol$ reads $\nabla^2 R(\xsol)$ which is positive semi-definite owing to convexity. 
Define $M_R \eqdef A \Pa{ \Id+ \tfrac{1}{\gamma} (A^TA)^{-1} \nabla^2 R(\xsol) }^{-1} (A^TA)^{-1} A^{T} $, and redefine
\beq\label{eq:mtx-dr-Rc2}
\mADMM \eqdef \sfrac12\Id + \sfrac12\pa{2M_R-\Id}\pa{2 \MJ - \Id} .
\eeq
%


\begin{proof}[{Proof of Proposition \ref{prop:trajectory-admm-real}}]

We prove the corollary in two steps.


\paragraph{1. Linearization of ADMM}

Following the above proof, we have for $\yk$ that
\[
B_J (\yk-\ykm) 
= \beta M_{\barJ} (\zk - \zkm) + o(\norm{\zk-\zkm}) .
\]
From \eqref{eq:admm2}, for $\xkp$ and $\xk$, since $R$ is globally smooth differentiable
\[
- A^T\Pa{ A\xk - \beta \pa{ \zkm - 2\psikm } } \in \beta \nabla R(\xk) 
\qandq
- A^T\Pa{ A\xkm - \beta \pa{ \zkmm - 2\psikmm } } \in \beta \nabla R(\xkm) ,
\]
which leads to, applying the local $C^2$-smoothness of $R$ around $\xsol$
\[
\begin{aligned}
&- A^T\Pa{ A\xk - \beta \pa{ \zkm - 2\psikm } } + A^T\Pa{ A\xkm - \beta \pa{ \zkmm - 2\psikmm } } \\
&= \beta \nabla R(\xk) - \beta \nabla R(\xkm) \\
&= \beta \nabla^2 R(\xkm) (\xk-\xkm) + o(\norm{\xk-\xkm}) \\
&= \beta \nabla^2 R(\xsol) (\xk-\xkm) + \beta \Pa{ \nabla^2 R(\xkm) - \nabla^2 R(\xsol) } (\xk-\xkm) + o(\norm{\xk-\xkm}) \\
&= \beta \nabla^2 R(\xsol) (\xk-\xkm) + o(\norm{\zkm-\zkmm}) .
\end{aligned}
\]
Using the fact that $A^TA$ is invertible and rearranging terms, we arrive at
\[
\begin{aligned}
& \Pa{ \Id + \beta (A^TA)^{-1} \nabla^2 R(\xsol) } (\xk-\xkm) + o(\norm{\zkm-\zkmm}) \\
&= \beta (A^TA)^{-1} A^T \pa{ \zkm - \zkmm } - 2 \beta (A^TA)^{-1} A^T \pa{ \psikm - \psikmm } + o(\norm{\zkm - \zkmm}) \\
&= - \beta (A^TA)^{-1} A^T \pa{ \zkm - \zkmm } + 2 (A^TA)^{-1} A^T B_{J} \pa{ \ykm - \ykmm } + o(\norm{\zkm - \zkmm}) ,
\end{aligned}
\]
which further leads to, denote $M_{R} = A \pa{ \Id + (A^T A)^{-1} H_{R} }^{-1} (A^T A)^{-1} A^T$
\[
\begin{aligned}
A (\xk-\xkm) 
&= - \beta M_{R} \pa{ \zkm - \zkmm } + 2 M_{R} B_{J} \pa{ \ykm - \ykmm } + o(\norm{\zkm - \zkmm}) \\
&= - \beta M_{R} \pa{\zkm - \zkmm} + 2 \beta M_{R} M_{\barJ} (\zkm - \zkmm) + o(\norm{ \zkm - \zkmm } ) .
\end{aligned}
\]
Finally, from \eqref{eq:admm2}, we have that
\[
\begin{aligned}
\zk - \zkm
&= \Pa{ \Id + 2M_{R}M_{\barJ} - M_{R} - M_{\barJ} } (\zkm - \zkmm) + o(\norm{ \zkm - \zkmm } ) .
\end{aligned}
\]



\paragraph{2. Trajectory of ADMM}

Since $A$ is full rank square matrix and hence invertible, from \eqref{eq:BarU-sim} we have
\[
\begin{aligned}
M_{R}
&= A \pa{ \Id + \tfrac{1}{\gamma} (A^T A)^{-1} \nabla^2 R(\xsol) }^{-1} (A^T A)^{-1} A^T \\
&= A (A^T A)^{-\frac{1}{2}} \bPa{\Id + \tfrac{1}{\gamma} (A^T A)^{-\frac{1}{2}} {\nabla^2 R(\xsol)} (A^T A)^{-\frac{1}{2}} }^{-1} (A^T A)^{-\frac{1}{2}} A^T \\
&\sim \bPa{\Id + \tfrac{1}{\gamma} (A^T A)^{-\frac{1}{2}} {\nabla^2 R(\xsol)} (A^T A)^{-\frac{1}{2}} }^{-1} ,
\end{aligned}
\]
where $\Pa{\Id + \tfrac{1}{\gamma} (A^T A)^{-\frac{1}{2}} {\nabla^2 R(\xsol)} (A^T A)^{-\frac{1}{2}} }^{-1}$ is symmetric positive definite. 
If we choose $\gamma$ such that $$\tfrac{1}{\gamma} \norm{(A^T A)^{-\frac{1}{2}} {\nabla^2 R(\xsol)} (A^T A)^{-\frac{1}{2}}} < 1, $$ then all the eigenvalues of $M_R$ are in $]1/2, 1]$, hence $W_R \eqdef 2 M_R - \Id$ is symmetric positive definite. Therefore, we get
\[
\begin{aligned}
\sfrac{1}{2}\Id + \sfrac{1}{2} W_R \Pa{2\MJ - \Id} 
&= W_R^{1/2} \bPa{ \sfrac{1}{2}\Id + \sfrac{1}{2} W_R^{1/2} \Pa{2\MJ - \Id} W_R^{1/2} } W_R^{-1/2} \\
&\sim \sfrac{1}{2}\Id + \sfrac{1}{2} W_R^{1/2} \Pa{2\MJ - \Id} W_R^{1/2} ,
\end{aligned}
\]
and $\bmADMM \eqdef \frac{1}{2}\Id + \frac{1}{2} W_R^{1/2} \pa{2\MJ - \Id} W_R^{1/2}$ is symmetric positive semi-definite with all eigenvalues in $[0,1]$. Hence, by similarity, the eigenvalues of $M$ are all real and contained in $[0,1]$.\qedhere

\end{proof}

\section{Adaptive acceleration for ADMM}
\label{sec:proof_a3dmm}

\subsection{Convergence of A$\!^3$DMM}

\begin{proof}[Proof of Proposition \ref{prop:convergence-lp}]
From \eqref{eq:perturbation-km}, we have that
\[
\zk 
= \calF(\zkm + \varepsilon_{k-1})
= \calF(\zkm) + \Pa{ \calF(\zkm + \varepsilon_{k-1}) - \calF(\zkm) } .
\]
Given any $\zsol \in \fix(\calF)$, since $\calF$ is firmly non-expansive, hence non-expansive, we have
\[
\norm{\zk - \zsol}
\leq \norm{ \calF(\zkm) - \calF(\zsol)} + \norm{ \calF(\zkm + \varepsilon_k) - \calF(\zkm) } 
\leq \norm{\zkm - \zsol} + \norm{ \varepsilon_{k-1} } ,
\]
which means that $\seq{\zk}$ is quasi-Fej\'er monotone with respect to $\fix(\calF)$. Then invoke \cite[Proposition 5.34]{bauschke2011convex} we obtain the convergence of the sequence $\seq{\zk}$. 
\qedhere
\end{proof}

\subsection{Acceleration guarantee of A$\!^3$DMM}

Recall the definition of $V_{k-1}, c_k, C_k$ and $\zbar_{k,s}$ in the beginning of the section. 
By definition,
\begin{equation}\label{eq:linear_rel}
V_k = M V_{k-1} .
\end{equation}
Define $E_{k,j} \eqdef V_k C_k^j - V_{k+1}$ for $j\geq 1$ and
\begin{equation}\label{eq:e0}
E_{k,0}\eqdef V_{k-1} C_k - V_{k} = \begin{bmatrix}
(V_{k-1} c_k - v_{k}) & 0 & \cdots & 0
\end{bmatrix}.
\end{equation}
We obtain the relation between the extrapolated point $\zbar_{k,s}$ and the $(k+s)$'th point of $\seq{\zk}$
$$
\zbar_{k,s} = z_{k} + \msum_{j=1}^s \pa{v_{j+k} + (E_{k,j})_{(:,1)}}
= z_{k+s} + \msum_{j=1}^s (E_{k,j})_{(:,1)}
$$

In the following, given a matrix $M$, we let $\rho(M)$ denote the spectral radius of $M$ and $\lambda(M)$ denote its spectrum.
\begin{proof}[Proof of Proposition \ref{prop:extrap-error}]
We first prove $(i)$ that the extrapolation error is controlled by the coefficients fitting error. 
%
%
Since $k\in \NN$ is fixed, for ease of notation, we also write $E_\ell = E_{k,\ell}$ and $C=C_k$.
We first show that
for $\ell \in \NN$, we have
\begin{equation}\label{eq:En}
E_\ell = \msum_{j=1}^\ell M^j E_0 C^{\ell-j}.
\end{equation}
We prove this by induction.
Note that
\begin{align*}
V_{k} C &\overset{\eqref{eq:linear_rel}}{=}\pa{M V_{k-1}} C \overset{\eqref{eq:e0}}{=} M V_k + M E_0 \overset{\eqref{eq:linear_rel}}{=} V_{k+1} + M E_0.
\end{align*}
Therefore, $E_1 = M E_0 $ as required. Assume that \eqref{eq:En} is true up to $\ell=m$. Then,
\begin{align*}
V_k C^{m+1} 
\overset{\eqref{eq:linear_rel}}{=}\pa{M V_{k-1} } C^{m+1} 
&\overset{\eqref{eq:e0}}{=} M V_k C^{m} + M E_0C^{m}= M \pa{V_{m+k} + E_{m}} + M E_0C^{m} \\
&\overset{\eqref{eq:linear_rel}}{=} V_{m+2} + M E_{m} + M E_0C^{m} 
\end{align*}
So, plugging in our assumption on $E_m$, we have
\begin{align*}
E_{m+1} 
&= M E_{m} + M E_0C^{m} 
= M E_0C^{m} + M \Pa{ \msum_{j=1}^m M^j E_0 C^{m-j} } 
= \msum_{j=1}^{m+1} M^{j} E_0 C^{m+1-j}.
\end{align*}
To bound the extrapolation error,
\begin{align*}
 \msum_{m=1}^s E_m 
 &= \msum_{m=1}^s\pa{ \msum_{j=1}^m M^j E_0 C^{m-j} }
 = \msum_{\ell=0}^{s-1} \Pa{\msum_{j=1}^{s-\ell} M^j} E_0 C^\ell 
 = \msum_{\ell=1}^{s} M^\ell E_0 \Pa{\msum_{i=0}^{s-\ell} C^i}
\end{align*}
Therefore, 
\begin{align*}
&\norm{\zbar_{k,s} - \zsol} \leq \norm{z_{k+s} - \zsol} + \msum_{\ell=1}^s \norm{M^\ell} \norm{E_0} \norm{ \msum_{i=0}^{s-\ell} C^i_{(1,1)}}.
\end{align*}
In the case of $s=+\infty$, we have
\begin{align*}
&\norm{\zbar_{k,\infty} - \zsol} 
\leq \msum_{\ell=1}^\infty \norm{M^\ell} \norm{E_0 (\Id - C)^{-1}_{(:,1)}} 
= \sfrac{\norm{E_0}}{1-\ssum_i c_i}\msum_{\ell=1}^\infty \norm{M^\ell} .
\end{align*}
The fact that $B_s$ is uniformly bounded in $s$ if $\rho(M)<1$ and $\rho(C)<1$ follows because this implies that $\sum_{\ell=1}^\infty \norm{M^\ell}<\infty$ thanks to the Gelfand formula, and $\sum_{i=0}^\infty C^i = (\Id - C)^{-1}$ and its $(1,1)^{th}$ entry is precisely $\frac{1}{1-\sum_i c_i}$.
Since $k\in \NN$ is fixed, for ease of notation, we also write $E_\ell = E_{k,\ell}$ and $C=C_k$.
We first show that
for $\ell \in \NN$, we have
\begin{equation}\label{eq:En}
E_\ell = \msum_{j=1}^\ell M^j E_0 C^{\ell-j}.
\end{equation}
We prove this by induction.
Note that
\begin{align*}
V_{k} C 
&\overset{\eqref{eq:linear_rel}}{=}\pa{M V_{k-1}} C \overset{\eqref{eq:e0}}{=} M V_k + M E_0 \overset{\eqref{eq:linear_rel}}{=} V_{k+1} + M E_0.
\end{align*}
Therefore, $E_1 = M E_0 $ as required. Assume that \eqref{eq:En} is true up to $\ell=m$. Then,
\begin{align*}
V_k C^{m+1} &\overset{\eqref{eq:linear_rel}}{=}\pa{M V_{k-1} } C^{m+1} \\
&\overset{\eqref{eq:e0}}{=} M V_k C^{m} + M E_0C^{m}= M \pa{V_{m+k} + E_{m}} + M E_0C^{m} \\
&\overset{\eqref{eq:linear_rel}}{=} V_{m+2} + M E_{m} + M E_0C^{m} .
\end{align*}
So, plugging in our assumption on $E_m$, we have
\begin{align*}
E_{m+1} 
&= M E_{m} + M E_0C^{m} = M E_0C^{m} + M \Pa{ \msum_{j=1}^m M^j E_0 C^{m-j} } 
= \msum_{j=1}^{m+1} M^{j} E_0 C^{m+1-j}.
\end{align*}
To bound the extrapolation error,
\begin{align*}
 \sum_{m=1}^s E_m 
 &= \sum_{m=1}^s\Pa{ \msum_{j=1}^m M^j E_0 C^{m-j} }
 = \sum_{\ell=0}^{s-1} \Pa{\msum_{j=1}^{s-\ell} M^j} E_0 C^\ell 
 = \sum_{\ell=1}^{s} M^\ell E_0 \Pa{\msum_{i=0}^{s-\ell} C^i}
\end{align*}
Therefore, 
\begin{align*}
&\norm{\bar z_{k,s} - \zsol} 
\leq \norm{z_{k+s} - \zsol} + \msum_{\ell=1}^s \norm{M^\ell} \norm{E_0} \norm{ \msum_{i=0}^{s-\ell} C^i_{(1,1)}}.
\end{align*}
In the case of $s=+\infty$, we have
\begin{align*}
&\norm{\bar z_{k,\infty} - \zsol} \leq \sum_{\ell=1}^\infty \norm{M^\ell} \norm{E_0 (\Id - C)^{-1}_{(:,1)}} 
= \mfrac{\norm{E_0}}{1-\ssum_i c_i}\msum_{\ell=1}^\infty \norm{M^\ell} .
\end{align*}
The fact that $B_s$ is uniformly bounded in $s$ if $\rho(M)<1$ and $\rho(C)<1$ follows because this implies that $\sum_{\ell=1}^\infty \norm{M^\ell}<\infty$ thanks to the Gelfand formula, and $\sum_{i=0}^\infty C^i = (\Id - C)^{-1}$ and its $(1,1)^{th}$ entry is precisely $\frac{1}{1-\sum_i c_i}$.

%
%
%
%
%
%
To control the 
coefficients fitting error $\epsilon_k$, 
we follow closely the arguments of \cite[Section 6.7]{sidi2017vector}, since this amounts to understanding the behaviour of the coefficients $c_k$, which are precisely the MPE coefficients. 
Recall our assumption that $M$ is diagonalisable, so $M = U^\top \Sigma U$ where $U$ is an orthogonal matrix and $\Sigma$ is a diagonal matrix with the eigenvalues of $M$ as its diagonal. Then, letting $u_k \eqdef U v_k$,
\begin{align*}
\epsilon_k &= \min_{c\in \RR^{q}} \norm{\msum_{i=1}^{q} c_i v_{k-i} - v_{k}}\\
&= \min_{c\in \RR^{q}} \norm{\msum_{i=1}^{q} c_i \Sigma^{k-i} u_0 - \Sigma^{k} u_0}
= \min_{g\in\calP_q}\norm{ \Sigma^{k-q} g(\Sigma) u_0} 
\leq \norm{u_0} \min_{g\in \Pp_q}\max_{z\in \lambda(M)} \abs{z}^{k-q} \abs{g(z)} 
\end{align*}
where $\calP_q$ is the set of monic polynomials of degree $q$ and $\lambda(M)$ is the spectrum of $M$. Choosing $g = \prod_{j=1}^{q} (z-\lambda_j)$, we have $g(\lambda_j) = 0$ for $j=1,\ldots,q$, so
\begin{equation}
\epsilon_k \leq \norm{u_0} \abs{\lambda_{q+1}}^{k-q} \max_{\ell>q}\prod_{j=1}^{q} \abs{\lambda_j-\lambda_\ell}.
\end{equation}
The claim that $\rho(C_k)<1$ holds since the eigenvalues of $C$ are precisely the roots of the polynomial $Q(z) = z^{k-1} - \sum_{i=1}^{k-1} c_j z^{k-1-i}$, and from \cite{sidi2017vector}, if $\abs{\lambda_{q}}>\abs{\lambda_{q+1}}$, then $Q$ has precisely $q$ roots $r_1,\ldots, r_q$ satisfying $r_j = \lambda_j + \Oo(\abs{\lambda_{q+1}/\lambda_j}^k)$. So, $\abs{r_j}<1$ for all $k$ sufficiently large.
To prove the non-asymptotic bounds on $\epsilon_k$, first
observe that $z_{k+1} - z_{k} = M(z_k - z_{k-1})$ implies $z_{k+1} - \zsol = M(z_k - z_{*})$ and $z_{k+1} - z_k = (M-\Id)(z_k - \zsol)$. So, letting $\gamma_i = -c_{k,i}/(1-\sum_i c_{k,i})$ for $i=1,\ldots, q$ and $\gamma_0 = 1/(1-\sum_i c_{k,i})$, we have
\begin{equation}\label{eq:coefficient_vec}
\mfrac{1}{1-\ssum_i c_{k,i}} \Pa{ v_{k} -\msum_{i=1}^{q} c_{k,i} v_{k-i} } 
= \msum_{i=0}^{q} \gamma_i v_{k-i} 
= (M-\Id) \msum_{i=0}^{q} \gamma_i( z_{k-i-1} - \zsol).
\end{equation}
Now, $y\eqdef \sum_{i=0}^{q} \gamma_i z_{k-i-1} $ is precisely the MPE update and norm bounds on this are presented in \cite{sidi2017vector}. For completeness, we reproduce their arguments here: Let $A\eqdef \Id - M$, by our assumption of $\lambda(M) \subset (-1,1)$, we have that $A$ is positive definite. Then,
\begin{align*}
\norm{A^{1/2}(y - \zsol)}^2 
&= \dotp{A (y-\zsol)}{(y-\zsol)} \\
&= - \dotp{\ssum_{i=0}^{q} \gamma_i v_{k-i} }{(y-\zsol) + w}
\end{align*}
where $w = \sum_{j=1}^{q} a_j v_{k-j}$ with $a\in \RR^{q}$ being arbitrary, since by definition of $\gamma$, $\dotp{\sum_{i=0}^{q} \gamma_i v_{k-i}}{v_\ell}=0$ for all $\ell=k-q,\ldots, k-1$. We can write
$$
w = \sum_{j=1}^{q} a_j (M-\Id)(z_{k-j-1}-\zsol) = \sum_{j=1}^{q} a_j (M-\Id)M^{k-j-1}(z_{0}-\zsol) = f(M) (z_0 - \zsol)
$$
where $f(z) = z^{k-q-1} (z-1) \sum_{j=1}^{q} a_j z^{q-j}$, and we can write
$$
y - \zsol = \sum_{i=0}^{q} \gamma_i M^{k-i-1} (z_0- \zsol) = g(M)(z_0-\zsol)
$$
where $g(z) = z^{k-q-1}\sum_{i=0}^{q} \gamma_i z^{q-i}$. Therefore, $f(z) + g(z) = z^{k-1-q} h(z)$, where $h$ is a polynomial of degree $q$ such that $h(1) = 1$. Moreover, since the coefficients $a_j$ are arbitrary, $h$ can be considered as an arbitrary element of $\tilde \Pp_q$, the set of all polynomials of degree $q$ such that $h(1) = 1$. Therefore
\begin{align*}
\norm{A^{-1/2}(y - \zsol)}^2 &\leq \norm{A^{-1/2}(y - \zsol)} \min_{h\in \tilde \Pp_q}\norm{M^n h(M) (z_0 - \zsol)} \\
&\leq \norm{A^{-1/2}(y - \zsol)} \min_{h\in \tilde \Pp_q}\max_{t\in \lambda(M)}\abs{t^n h(t)} \norm{z_0 - \zsol}
\end{align*}
In particular, combining this with \eqref{eq:coefficient_vec}, we have
$$
\mfrac{\epsilon_k}{\abs{1-\ssum_i c_{k,i}}}\leq \norm{z_0-\zsol} \norm{(\Id-M)^{1/2}} \rho(M)^n \min_{h\in \tilde \Pp_q}\max_{t\in \lambda(M)}\abs{ h(t)} 
$$
Finally, in our case where $\lambda(M) =[\alpha,\beta]$ with $1>\beta>\alpha >-1$, it is well known that $\min_{h\in \tilde \Pp_q}\max_{t\in \lambda(M)}\abs{ h(t)}$ has an explicit expression (see, for example, \cite{borwein2003monic} or \cite[Section 7.3.1]{sidi2017vector}):
$$
\min_{h\in \tilde \Pp_q}\max_{z\in \lambda(M)} \abs{h(z)} \leq \max_{z\in \lambda(M)} \abs{h_*(z)},
$$
where $h_*(z) \eqdef \frac{T_q\pa{\frac{2z-\alpha-\beta}{\beta-\alpha}}}{T_q\pa{\frac{2-\alpha-\beta}{\beta-\alpha}}}$ where $T_q(x)$ is the $q^{th}$ Chebyshev polynomial and it is well known that
\begin{equation}
\min_{h\in \tilde \Pp_q}\max_{z\in [\alpha,\beta]} \abs{h(z)} 
\leq 2 \BPa{\sfrac{\sqrt{\eta}-1}{\sqrt{\eta}+1}}^q
\end{equation}
where $\eta = \frac{1-\alpha}{1-\beta}$. 
\qedhere

\end{proof}